\documentclass{article}
\usepackage{amsfonts}

\usepackage{braket}
\usepackage{mathrsfs}
\usepackage{amssymb}
\usepackage{booktabs}
\usepackage{multirow}
\usepackage{comment}

\usepackage{amsmath}
\makeatletter
\renewcommand*\env@matrix[1][\arraystretch]{%
  \edef\arraystretch{#1}%
  \hskip -\arraycolsep
  \let\@ifnextchar\new@ifnextchar
  \array{*\c@MaxMatrixCols c}}
\makeatother

\usepackage{upgreek}
\usepackage{caption}
\usepackage{xcolor}
\usepackage{authblk}
\usepackage[english]{babel}
\usepackage[utf8]{inputenc}
\usepackage{todonotes}

\usepackage[normalem]{ulem}
\usepackage[paper=a4paper,margin=1in]{geometry}

\usepackage[hypertexnames=false]{hyperref}
\usepackage[nameinlink]{cleveref}

\usepackage[numbers,sort]{natbib}

\expandafter\def\csname ver@etex.sty\endcsname{3000/12/31}

\usepackage{autonum}
\title{CutsAndLiftingsForGeneralizedComplexCutPolytope}
\date{ }
 \newcommand{\strFunc}{\normalfont{\texttt{str}}}

\newcommand{\BlueUpdate}[1]{{#1}}
\newcommand{\conjTrans}{^\mathrm{H}}
\newcommand{\imagUnit}{\mathbf{i}}
\newcommand{\ConvHull}{\mathsf{Conv}}

\newcommand{\rankOp}{\mathrm{rk}}
\newcommand{\triangleOp}{\mathbf{T}}
\newcommand{\zeroPrecisionMimo}{10^{-6}}

\newcommand{\cutPolytope}{\mathrm{CUT}^n_m}
\newcommand{\cutPolytopeNoSS}{\mathrm{CUT}}
\newcommand{\elliptopeNoSS}{\mathcal{E}}
\newcommand{\numRunsMimo}{600}

\newcommand{\elliptope}{\mathcal{E}^n_m}
\newcommand{\friezeVec}{\mathbf{a}}
\newcommand{\Hadamard}{\odot}
\newcommand{\polySet}{\widetilde{\mathbb{C}}[x]}

\newcommand{\liftingSet}[1]{\mathbf{L}(\mathscr{B}_{#1})}

\newcommand{\complexTriSet}{\triangleOp\left(\mathcal{E}^n_3\right)}

\newcommand{\realTriSet}{\mathrm{Re}\left( \complexTriSet \right)}

\DeclareMathOperator*{\argmin}{arg\,min}
\DeclareMathOperator*{\argmax}{arg\,max}

\usepackage{amsthm}
\usepackage{thmtools,thm-restate}

\usepackage{listings}
\lstdefinelanguage{Sage}[]{Python}
{morekeywords={False,sage,True},sensitive=true}
\definecolor{dblackcolor}{rgb}{0.0,0.0,0.0}
\definecolor{dbluecolor}{rgb}{0.01,0.02,0.7}
\definecolor{dgreencolor}{rgb}{0.2,0.4,0.0}
\definecolor{dgraycolor}{rgb}{0.30,0.3,0.30}

\newtheorem{theorem}{Theorem}%
\newtheorem{corollary}{Corollary}%
\newtheorem{lemma}{Lemma}
\newtheorem{proposition}{Proposition}
\newtheorem{conjecture}{Conjecture}
\newtheorem{remark}{Remark}
\newtheorem{example}{Example}
\newtheorem{definition}{Definition}

\crefname{conjecture}{conjecture}{conjectures}

\Crefname{lstlisting}{Listing}{Listings}

\begin{document}

\title{Cuts and semidefinite liftings for the complex cut polytope} 

\author{Lennart Sinjorgo \thanks{Department of Econometrics and OR, Tilburg University, CentER, The Netherlands, {\tt l.m.sinjorgo@tilburguniversity.edu}}
\quad  {Renata Sotirov} \thanks{Department of Econometrics and OR, Tilburg University, CentER, The Netherlands, {\tt r.sotirov@tilburguniversity.edu}}
\quad {Miguel F. Anjos} \thanks{School of Mathematics, University of Edinburgh, Scotland, United Kingdom, {\tt anjos@stanfordalumni.org}}}
\maketitle

\begin{abstract}
    We consider the complex cut polytope: the convex hull of Hermitian rank~1 matrices $xx\conjTrans$, where the elements of $x \in \mathbb{C}^n$ are $m$th unit roots. These polytopes have applications in MAX-3-CUT, digital communication technology, angular synchronization and more generally, complex quadratic programming. 
    For \mbox{$m=2$},  the complex cut polytope corresponds to  the well-known {cut polytope}.   
    We generalize valid cuts for this polytope to cuts for any complex cut polytope with finite $m>2$ and
    provide a framework to compare them.
    Further, we consider  a second  semidefinite lifting of the complex cut polytope for $m=\infty$.
    This lifting is proven to be equivalent to other complex Lasserre-type liftings of the same order proposed in the literature, while being of smaller size. Our theoretical findings are supported by numerical experiments on various optimization problems.   
\end{abstract}

{\bf Keywords.}   
 complex semidefinite programming,  complex cut polytope, polyhedral combinatorics,  MIMO, angular synchronization \\

{\bf MSC2020 codes.}  52B05, 60G35, 90C20, 90C22

\section{Introduction}
 
    The maximum-cut problem (MAX-CUT) on a graph, is to find a partition of vertices in two disjoint subsets, that maximizes the number of edges that cross the partition. MAX-CUT finds applications in VLSI design and physics \cite{barahona1988application}, data science \cite{ding2001min}, and is $\mathcal{NP}$-hard. The convex hull of the rank 1 matrices representing all partitions is known as the cut polytope. This polytope admits an exponential number (in $n$) of extreme points, and  it cannot be efficiently described, in contrast to its positive semidefinite (PSD) approximation, the elliptope \cite{Laurent1995OnAP}.

    We consider here complex generalizations of the cut polytope and elliptope, namely the complex cut polytope, denoted $\cutPolytope$, and the complex elliptope, denoted $\elliptope$. For fixed integers $m$ and $n$, $\cutPolytope$ is defined as the convex hull of Hermitian rank 1 matrices $xx\conjTrans$, where the elements of the vectors $x \in \mathbb{C}^n$ are $m$th unit roots. For $m = 2$, $\cutPolytope$ corresponds to the cut polytope. 
   The set $\cutPolytope$ finds applications in the multiple-input multiple-output detection problem (MIMO) \cite{jiang2021tightness,lu2019tightness,mobasher2007near,zhang2011optimal}, angular synchronization \cite{bandeira2017tightness}, phase retrieval \cite{waldspurgerPhaseRecovery}, radar signal processing \cite{Soltanalian2013DesigningUC,lu2020enhanced}, and for $m = 3$, it can be used to model MAX-3-CUT \cite{goemans2001approximation}. 
    For finite $m \geq 3$, algorithms for optimization over $\cutPolytope$ are proposed in \cite{lu2018argument,lu2020enhanced}, and approximation ratios are studied in \cite{so2007approximating,zhang2006complex}.

    In this work, we derive novel cuts in the complex plane that separate $\mathcal{E}^n_m$ from $\cutPolytope$.  In particular, 
    we derive all facets of $\cutPolytopeNoSS^3_3$ to obtain an exact description.
    We define a function $\strFunc$, that provides the approximation ratio of maximization over $\elliptope$ and maximization over $\cutPolytope$, for given problem instances.
    This function is used for numerically evaluating the effect of adding valid cutting planes to $\elliptope$. We prove that the  {cuts introduced here} are invariant under rotations and taking the conjugate. We also investigate the effect of adding  cuts to $\elliptope$ for various optimization problems. 
    
      Optimization over $\elliptope$ can be done in polynomial time (for fixed precision), by solving a complex semidefinite program (CSDP).
    CSDPs have recently received much attention in the literature \cite{josz2018lasserre,wang2023real,wang2022exploiting,wang2023more,magron2022exact,zhong2018near}. 
    CSDPs with matrix variables of order $n$ are solved by SDP solvers as real SDPs with matrix variables of order $2n$. 
    In \cite[Corollary 2.5.2]{dobre2011semidefinite} and \cite{wang2023real}, conditions are provided under which this doubling of the size can be avoided. In this work, we extend these conditions. Specifically, we show that CSDPs can be reformulated as real SDPs of same size, when the objective function contains only real coefficients, and the feasible set of the CSDP is closed under complex conjugation.  In particular, we show that this is the case for CSDPs over $\elliptope$, and that the derived complex facets can be equivalently reformulated to real facets.

    The set $\cutPolytopeNoSS^n_\infty$ is studied in \cite{jarre2020set}. The first semidefinite lifting of $\cutPolytopeNoSS^n_\infty$, denoted  ${\mathcal E}^n_\infty$,
    is also known as the set of correlation matrices \cite{grone1990extremal,li1994note}. Here, we extend the results of \cite{jarre2020set}. In particular, we  consider second semidefinite Lasserre-type liftings of $\cutPolytopeNoSS^n_\infty$. 
     Such liftings are defined in terms of moment matrices, and we study second  liftings with smaller moment matrices than those proposed in the literature \cite{jarre2020set,josz2018lasserre}. Despite this decrease in size, we show  that here considered liftings are equivalent to those proposed in the literature.
    Moreover,  for $n = 4$ (the smallest $n$ for which $\cutPolytopeNoSS^n_\infty \subsetneq \elliptopeNoSS^n$), we prove that the second semidefinite lifting of $\cutPolytopeNoSS^n_\infty$
        excludes all rank 2 extreme points present in $\elliptopeNoSS^4_\infty$, and that matrices in this set satisfy a certain valid cut for $\cutPolytopeNoSS^4_\infty$. 

     We also show, via a constructive proof, that $\elliptope$ contains rank 2 extreme points for all integer $n,m \geq 3$. This shows the strict inclusion of $\cutPolytope$ in $\elliptope$ for these values of $n$ and $m$. For $n = 3$, we provide necessary and sufficient conditions for matrices to be rank 2 extreme points of $\elliptopeNoSS^3_m$. \\

    This paper is organized as follows. Notation is given in \Cref{section_notation}. We provide the  definitions of $\cutPolytope$ and $\elliptope$ in \Cref{section_preliminaries}. 
    In \Cref{section_strFramework}, we introduce a framework for finding valid inequalities for $\cutPolytope$ and provide some valid cuts.
     In \Cref{section_exactDescriptionCut33}, we provide an exact description 
    of $\cutPolytopeNoSS^3_3$,  and use the derived facets of $\cutPolytopeNoSS^3_3$ to strengthen $\elliptopeNoSS^n_3$ for general $n \geq 3$. 
         {In Section~\ref{sect:efficientReformulation} we provide an efficient reformulation of CSDPs whose convex feasible sets are  closed under complex conjugation.}
    In \Cref{section_secondLiftingCutInf} we investigate the sets $\cutPolytopeNoSS^n_\infty$ and second semidefinite liftings of $\cutPolytopeNoSS^n_\infty$. In \Cref{section_liftingsOfE3m}, we  study rank 2 extreme points of $\elliptope$ for integer $m > 2$. In \Cref{section_numericalResults}, we numerically investigate the effect of adding cuts to $\elliptope$ for various optimization problems from the literature. Lastly, in \Cref{section_conclusions} we draw conclusions, and propose future research directions.

\subsection{Notation}
\label{section_notation}
For $n \in \mathbb{N}$, $[n] := \{1, \ldots, n\}$. The imaginary unit is denoted by $\imagUnit := 
\sqrt{-1}$. 
The complex conjugate of $z \in \mathbb{C}$ is denoted by $\bar{z}$, and its modulus by $| z | = \sqrt{z \bar{z}}$. The Hermitian transpose of a complex matrix $A$ is denoted by  $A\conjTrans$. For $z \in \mathbb{C}^n$, $\| z \|  = \sqrt{z\conjTrans z}$. A matrix $A\in  \mathbb{C}^{n\times n}$ is called Hermitian if $A=A\conjTrans$.
A Hermitian matrix $A$ is said to be Hermitian PSD, denoted by $A\succeq 0$, if $x\conjTrans A x \geq 0$ for all $x\in \mathbb{C}^n$. For any $z \in \mathbb{C}$, $\mathrm{Re}(z) \in \mathbb{R}$ and $\mathrm{Im}(z) \in \mathbb{R}$ denote the real and imaginary part of $z$, respectively. Additionally, by slight abuse of notation, for sets $U \subseteq \mathbb{C}^n$, we define $\mathrm{Re}(U) := \{ \mathrm{Re}(u) \, | \, u \in U \} \subseteq \mathbb{R}^n$. The boundary of a set $U$ is denoted $\partial U$.

We denote by $\mathcal{H}^n$ the set of Hermitian matrices of order $n$, and $\mathcal{H}^n_+$ for the set of positive semidefinite Hermitian matrices. Similarly, $\mathcal{S}^n$ and $\mathcal{S}^n_+$ denote, respectively, the sets of symmetric real matrices, and symmetric real positive semidefinite matrices, both of order $n$. If the context is clear, we omit the superscript $n$. The vector space  $\mathcal{H}^n$  is equipped with the trace inner product, i.e.,  for $A,B \in \mathcal{H}^n$ we have $\langle A, B \rangle := \text{Tr}(AB)$. The rank of a matrix is denoted $\rankOp(A)$.
 
For $n \in \mathbb{R}_+$, $\lfloor n \rfloor$ and $\lfloor n \rceil$ denote the rounding down and rounding to nearest integer operators, respectively.
The~Hadamard product of two matrices $A=(a_{ij})$ and $B=(b_{ij})$ of the same size is denoted by $A \Hadamard B$ and is defined as   $(A \Hadamard B)_{ij}:=a_{ij}b_{ij}.$

Let $J_n$ (resp.~$I_n$) denote the all ones matrix (resp.~identity matrix) of order $n$. The vector of all ones (resp.~zeros) and length $n$ is denoted by $\mathbf{1}_n$ (resp.~$\mathbf{0}_n$). However, we omit $n$ when the size of a matrix is clear from the context. Matrix $E_{ij}$ denotes the matrix which is zero everywhere, except for entry $(i,j)$, which has value 1.
For any matrix $X \in \mathcal{H}^n$, $\text{diag}(X) \in \mathbb{R}^n$ denotes the vector containing the diagonal entries of $X$. Similarly, for $x \in \mathbb{R}^n$, $\text{Diag}(x) \in \mathcal{S}^n$ denotes the diagonal matrix with $x$ on the diagonal.

\section{Preliminaries}
\label{section_preliminaries}
We define, for fixed integer $m \geq 2$, the set
\begin{align}
    \label{eqn_discreteSymbolSet}
    \mathcal{B}_m := \left\{ \exp{( \theta \imagUnit)} \, \middle| \, \theta = \frac{2 \pi k}{m}, \, k \in [m] \right\} \subseteq \mathbb{C},
\end{align}
as the set of the complex $m$th roots of unity. We define $\mathcal{B}^n_m$ as the set containing $m^n$ vectors of length $n$, in which each entry is restricted to be an element of $\mathcal{B}_m$.

In this paper, we consider a generalization of the well-known cut polytope \cite{deza1997geometry}, to which we refer as the \textit{complex cut polytope}. For integers $n,m \geq 2$, the complex cut polytope is defined as
\begin{align}
    \label{eqn_cutPolytopeSet}
    \cutPolytope := \ConvHull \left \{ xx\conjTrans \, | \, x \in \mathcal{B}^n_m \right \}.
\end{align}
As $\mathcal{B}_2 = \{ \pm 1 \}$, the set $\cutPolytopeNoSS^n_2$ coincides with the well-known cut polytope, which is a feasible set for  the maximum-cut problem \cite{GoemansWilliamsonMax2Cut,Laurent1995OnAP}. 
Optimization problems over $\cutPolytope$, $m \geq 2$, are $\mathcal{NP}$-hard, as they include MAX-CUT. 

Let us  define  \textit{the complex elliptope} as follows:
\begin{align}
    \label{eqn_ElliptopeSDP}
    \mathcal{E}^n_m := \left\{ X \in \mathcal{H}^n_+ \, \middle| \, \text{diag}(X) = \mathbf{1}, X_{ij} \in \ConvHull\left (\mathcal{B}_m \right ) \right\}.
\end{align} 
Note that for $m=2$,  $ X \in \mathcal{H}^n_+$ such that $\text{diag}(X) = \mathbf{1}$ implies $X_{ij} \in  [-1,1]$. Thus, the complex elliptope $\mathcal{E}^n_2$ corresponds to the elliptope that is defined by \citeauthor{Laurent1995OnAP}~\cite{Laurent1995OnAP}. 
For $m=3$, the complex elliptope $\mathcal{E}^n_m$ corresponds to the feasible set of the complex SDP relaxation for MAX-3-CUT by~\citeauthor{goemans2001approximation}~\cite{goemans2001approximation}. It is clear that $\cutPolytope \subseteq \mathcal{E}^n_m$. 

Here, we derive strong approximations of $\cutPolytope$ by using SDP. Besides considering second semidefinite liftings, we also derive cuts in the complex plane that separate $\mathcal{E}^n_m$ from $\cutPolytope$. Cuts in the complex plane have recently been studied by \citeauthor{jarre2020set}~\cite{jarre2020set}, for  the set $\cutPolytopeNoSS^n_\infty$, defined as
\begin{align}
    \label{eqn_infiniteCutPolytope}
    \cutPolytopeNoSS^n_\infty := \ConvHull \left\{ xx\conjTrans \, \middle| \, x \in \mathbb{C}^n, | x_i | = 1 \,\, \forall i \in [n] \right\}.
\end{align}
We also define $\mathcal{B}_\infty := \left\{ \exp{( \theta \imagUnit)} \, \middle| \, \theta \in \mathbb{R} \right\}$ as a natural extension of \eqref{eqn_discreteSymbolSet},
and the complex elliptope
\begin{align}
    \label{eqn_infElliptope}
    \mathcal{E}^n_\infty := \left\{ X \in \mathcal{H}^n_+  \, | \, \text{diag}(X) = \mathbf{1}_n \right\}.
\end{align}
Note that for $X\in \mathcal{E}^n_\infty$, we have $X_{ij} \in \ConvHull(\mathcal{B}_\infty)= \{ x \in \mathbb{C} \, | \, | x | \leq 1 \}$. The complex elliptope $\mathcal{E}^n_\infty$ can be considered as the first  semidefinite lifting of $\cutPolytopeNoSS^n_\infty$. Additionally, one can define a second semidefinite lifting of $\cutPolytopeNoSS^n_\infty$, following \cite{lasserre2001global}, and also proposed by \citeauthor{jarre2020set} \cite{jarre2020set} for $n = 4$.

\subsection{Basic CSDP relaxations}
In this section, we present the basic semidefinite program whose feasible set is the complex elliptope $\mathcal{E}^n_m$ for integer $m\geq 2$, see \eqref{eqn_ElliptopeSDP}. 
The basic SDP relaxation for $m=2$ was introduced by 
\citeauthor{GoemansWilliamsonMax2Cut}~\cite{GoemansWilliamsonMax2Cut}, for $m=3$ by
\citeauthor{goemans2001approximation}~\cite{goemans2001approximation}, and for general $m\geq 3$ by \citeauthor{lu2018argument}~\cite{lu2018argument}. 
In the sections that follow, we will derive cuts that strengthen the basic SDP. 
Let $n,m \geq 2$ and $C \in \mathcal{H}^n$. From the definitions of $\cutPolytope$ and $\mathcal{E}^n_m,$ we have 
\begin{align}
\tag{CSDP-P}
    \label{eqn_mimoSDPrelax}
    \max_{x \in \mathcal{B}^n_m} x\conjTrans C x &= \max_{X \in \cutPolytope} \langle C, X \rangle \leq \max_{X \in \mathcal{E}^n_m} \langle C, X \rangle.
\end{align}
Note that the above upper bound (referred to as \ref{eqn_mimoSDPrelax}, with \textit{P} for primal) is computable in polynomial time up to desired accuracy by the interior point method.  The complex elliptope $\mathcal{E}^n_m$ contains positive definite matrices, e.g., the identity.  For $X\in \mathcal{E}^n_m$, we require that $X_{ij} \in \ConvHull(\mathcal{B}_m)$, see \eqref{eqn_discreteSymbolSet}. One way to enforce this is to set
\begin{align}
    X_{ij} = \sum_{k = 1}^m \lambda_k e^{\theta_k \imagUnit}, \text{ with } \sum_{k=1}^m \lambda_k = 1, \, \lambda \geq 0, \, \lambda \in \mathbb{R}^m \text{ and } \theta_k = \frac{2 \pi k}{m} \text{ (i.e., } \exp{(\theta_k \imagUnit)} \in {\mathcal B}_m \text{)}.
\end{align} 
Alternatively, $X_{ij}$ can be restricted to lie in certain half-spaces. This perspective follows from the well-known fact that $\mathbb{C}$ is isomorphic to $\mathbb{R}^2$ via the bijective mapping
\begin{align}
    \label{eqn_bijectiveCtoR}
    g : \mathbb{R}^2 \to \mathbb{C}, \,\, g(a) = a_1 +  a_2 \imagUnit,
\end{align}
and that, for $a, b \in \mathbb{R}^2$, $a^\top b = \mathrm{Re}(\overline{g(a)} g(b))$. Now, it is easy to see that the set $\ConvHull(\mathcal{B}_m)$ is given by an $m$-sided regular convex polygon in $\mathbb{C}$. For the edge connecting $\exp{(\theta_k \imagUnit)}$ and $\exp{( \theta_{k-1} \imagUnit)}$, its normal vector (complex number) is given by $\nu_k := \exp{[ (\theta_k + \theta_{k-1})\imagUnit/2]} = \exp{[ (2k-1) \pi \imagUnit /m]}$ for $k\in [m]$. Thus,
\begin{align}
    \label{eqn_convBineqElement}
    X_{ij} \in \ConvHull(\mathcal{B}_m) \iff \mathrm{Re} \left( \overline{\nu}_k X_{ij} \right ) \leq \cos{\left (\frac{\pi}{m} \right )}\, \quad \forall k \in [m],
\end{align}
see also \cite{lu2018argument}. 

 {\begin{remark}
    The set $\ConvHull(\mathcal{B}_m)$ is the regular $m$-sided polygon, and as \eqref{eqn_convBineqElement} shows, this set can be described by its $m$ facets. A more efficient encoding is derived by viewing $\ConvHull(\mathcal{B}_m)$ as the projection of a higher dimensional polytope with possibly fewer facets. The earliest result in this direction is due to \citeauthor{ben2001polyhedral} \cite{ben2001polyhedral}. They showed that, for $m$ a power of two, $\ConvHull(\mathcal{B}_{m})$ is equivalent to the projection of a polytope with $\mathcal{O}(\log m)$ number of facets. Later results extended this to general $m$ \cite[Theorem 2]{fiorini2012extended}, and projections of higher dimensional spectrahedra \cite{fawzi2017equivariant,fawzi2016sparse}. Since $m$ is small in our numerical experiments, we do not use such liftings of the regular polygon in our CSDP relaxations.
\end{remark}}

To state \eqref{eqn_convBineqElement} in terms of matrix inner products, we define, for  $k \in [m],$  $i,j \in [n]$, $i<j$ the Hermitian matrices
\begin{align}
    \label{eqn_WijkDefinition}
    W^k_{ij} := \frac{1}{2} \left( \nu_k E_{ij}  +  \overline{\nu}_k  E_{ji} \right)
\end{align}
so that $\mathrm{Re}( \overline{\nu}_k X_{ij} ) = \langle W^k_{ij}, X \rangle$.
Now, from the SDP duality theory, it follows that the corresponding dual problem of \ref{eqn_mimoSDPrelax} is given by 
\begin{equation}
\tag{CSDP-D}
\label{eqn_mimoSDPdual}
\begin{aligned} 
    & \text{min } \mathbf{1}^\top \mu + \cos{\left (\frac{\pi}{m} \right )} \sum_{ij \in [n]^2,i<j, k \in [m]} \omega^k_{ij}, \\
    &    \text{s.t. } S = \text{Diag}(\mu) + \sum_{ij \in [n]^2, i < j,k \in [m]} \omega^k_{ij} W^k_{ij} - Q \succeq 0, \\
    &   \hphantom{\text{s.t. }} \mu \in \mathbb{R}^{n},  \omega_{ij} = (\omega^1_{ij}, \ldots, \omega^m_{ij})^\top \in \mathbb{R}^m_+, \quad \forall  i,j \in [n], \, i < j.
\end{aligned}
\end{equation}
One can strengthen \ref{eqn_mimoSDPrelax} and \ref{eqn_mimoSDPdual} (\textit{D} for dual) via the moment and sum of squares hierarchies by \citeauthor{lasserre2001global} \cite{lasserre2001global}. 
We consider this in more detail in \Cref{section_secondLiftingCutInf,section_liftingsOfE3m}, where we consider a second semidefinite lifting of $\cutPolytopeNoSS^4_\infty$ and $\cutPolytopeNoSS^3_m$ for finite $m$, respectively.
In \Cref{section_strFramework} we strengthen  \ref{eqn_mimoSDPrelax} by adding valid cuts to $\mathcal{E}^n_m$, which can be considered as the first semidefinite lifting of $\cutPolytope$.

\section{Framework for finding valid inequalities for \texorpdfstring{$\cutPolytope$}{CUT polytope}}
\label{section_strFramework}
In this section we introduce a general framework to derive valid inequalities for $\cutPolytope$, see \eqref{eqn_cutPolytopeSet}.  These inequalities can be then used to strengthen the SDP  relaxation~\ref{eqn_mimoSDPrelax}.  {We show that these inequalities can be classified in equivalence classes, which we derive from the group structure of $\cutPolytope$.}

{Similar to the classical result by \citeauthor{rockafellar1997convex} \cite[Theorem 18.8]{rockafellar1997convex}, stating that any real closed convex set is the intersection of all its half-spaces containing it, the set $\cutPolytope$ has an equivalent description as follows.}
\begin{proposition}
\label{lemma_alternativeCutPolytope}
     \begin{align}
\label{eqn_alternativeCutPolytope}
    \cutPolytope = \left \{ X \in \mathcal{H}^n \, \middle| \, \langle Q, X \rangle \leq \max_{x \in \mathcal{B}^n_m} x\conjTrans Q x, \,  \forall Q \in \mathcal{H}^n \right \}.
\end{align}
\end{proposition}
{The proof of \Cref{lemma_alternativeCutPolytope} is similar to the proof of \citeauthor{rockafellar1997convex} \cite[Theorem 18.8]{rockafellar1997convex} and therefore omitted.}
Observe that, by Hermiticity of $Q$, the values $\langle Q, X \rangle$ and $\max_{x \in \mathcal{B}^n_m} x\conjTrans Q x$ are real. Therefore,  the inequalites in \eqref{eqn_alternativeCutPolytope} are well defined.

Let us exploit the formulation of  $\cutPolytope$ given by~\eqref{eqn_alternativeCutPolytope} for deriving cuts that can be added to \ref{eqn_mimoSDPrelax} in order to 
improve that relaxation.  
We define the function $\strFunc: \mathcal{H}^n \backslash {0} \times \mathbb{N} \to [1,\infty)$ ($\strFunc$ for strength), as follows:
    \begin{align}
        \label{eqn_strFunctionDef}
        \strFunc(Q,m) := \frac{\max_{X \in \mathcal{E}^n_m} \langle Q,X \rangle} {\max_{X \in \cutPolytope} \langle Q,X \rangle}.
    \end{align}
\BlueUpdate{
We assume w.l.o.g.~that both the numerator and denominator in fraction \eqref{eqn_strFunctionDef} are strictly positive. This is the case for all here considered matrices $Q$. 
Clearly, for a matrix $Q$ for which ${\max_{X \in \cutPolytope} \langle Q,X \rangle} \leq 0$, one can  appropriately adjust its  diagonal.
}

Observe that $\strFunc$ returns the approximation ratio of maximization over $\elliptope$ and maximization over $\cutPolytope$, for a specific problem instance given by $Q$ (see also \cite[Section 4]{Laurent1995OnAP}). Since $\max_{X \in \mathcal{E}^n_m} \langle Q,X \rangle$ is an upper bound for  $\max_{X \in \cutPolytope} \langle Q,X \rangle$ \BlueUpdate{(both these values are assumed to be strictly positive)}, we have that $\strFunc(Q,m) \geq 1$. To improve the quality of this upper bound, one can find  valid inequalities for $\cutPolytope$, that are violated by $\argmax_{X \in \mathcal{E}^n_m} \langle Q,X \rangle$. Thus, if $\strFunc(Q,m)>1$, then by adding the cut 
\begin{align}
\label{deriveCuts}
 \langle Q,X \rangle \leq \max_{X \in \cutPolytope} \langle Q,X \rangle,
\end{align}
to~\ref{eqn_mimoSDPrelax} one may strengthen that  relaxation. Note that it is, in general,  $\mathcal{NP}$-hard to compute $\strFunc(Q,m)$. However, for some $Q$ we can find optimal solutions of both maximization problems in \eqref{eqn_strFunctionDef} analytically, and thus evaluate $\strFunc(Q,m)$, see Section~\ref{generalized triangle inequality}. 

\begin{remark}
    \label{remark_fairScaling}
    For any $c \in \mathbb{R}^n_+$,  $1 \leq \strFunc(Q + \text{Diag}(c),m) \leq  \strFunc(Q,m)$. In order to fairly compare cuts, we consider matrices $Q$ that satisfy $\langle Q, I \rangle = 0$.
\end{remark}

 {\subsection{Symmetries of \texorpdfstring{$\cutPolytope$}{CUT n,m}}
\label{section_groupStructure}
In this section, we formally specify symmetries of the set $\cutPolytope$, for finite $m$, that follow from the underlying group structure of $\mathcal{B}_m$.
These symmetries will be exploited later in   \Cref{section_ROCequivalence} and \Cref{section_exactDescriptionCut33}.

As starting point, consider the set $\mathcal{B}_m$ defined in \eqref{eqn_discreteSymbolSet}. With the usual multiplication of complex numbers, $\mathcal{B}_m$ forms a cyclic group of order $m$ with identity element $1$. Set $\mathcal{B}^n_m$ with the Hadamard product forms a finite abelian group, with identity element $\mathbf{1}_n$. Indeed, if $x \in \mathcal{B}^n_m$, its inverse element is given by $\overline{x}$, since $\overline{x} \Hadamard x = \mathbf{1}_n$.

We define, for $\alpha \in \mathcal{B}^n_m$, the linear group action of $f_{\alpha}  : \mathcal{H}^n \to \mathcal{H}^n$ as $f_{\alpha}(Z) =  (\alpha \alpha\conjTrans) \Hadamard Z  = \text{Diag}(\alpha) \, Z \, \text{Diag}(\overline{\alpha})$. To see that $f_\alpha$ defines a group action, let $Z \in \mathcal{H}^n$ and note that $f_{\mathbf{1}_n}(Z) = J_n \Hadamard Z = Z$, and 
\begin{align}
    \label{eqn_groupActionAxiom}
    f_{\alpha}(f_\beta(Z)) = \left( \alpha \alpha\conjTrans \right) \Hadamard \left( \beta \beta\conjTrans \right) \Hadamard Z = \left( \alpha \Hadamard \beta \right) \left( \alpha \Hadamard \beta \right)\conjTrans \Hadamard Z = f_{\alpha \Hadamard \beta}(Z).    
\end{align}
Since $f_\alpha$ defines a group action, $f_\alpha$ is also invertible, with inverse $f_\alpha^{-1} = f_{\overline{\alpha}}$, which follows from \eqref{eqn_groupActionAxiom} by taking $\beta = \overline{\alpha}$. That is, $f_{\alpha}(f_{\overline{\alpha}}(Z)) = f_{\alpha \Hadamard \overline{\alpha}}(Z) = f_{\mathbf{1}_n}(Z) = Z$.

Note that for $Z, Z^\prime \in \mathcal{H}^n$, we have
\begin{align}
    \label{eqn_invariantProd}
    \langle f_\alpha(Z), f_\alpha(Z^\prime) \rangle = \langle Z, Z^\prime \rangle.
\end{align}
Informally, function $f_\alpha$ can be considered as acting on $\mathcal{H}^n$ by applying some rotation to the elements of its input matrix. The number of such rotations is equal to the number of different functions $f_\alpha$. This number is given by $m^{n-1}$, and not $|\mathcal{B}^n_m| = m^n$, because $f_{\alpha} = f_{x \alpha}$ for any $x \in \mathcal{B}_m$ and $\alpha \in \mathcal{B}^n_m$. Therefore, in the context of $f_\alpha$, one element of $\alpha$ can be assumed fixed, say $\alpha_1 = 1$.

The sets $\cutPolytope$ and $\elliptope$ 
are closed under the group action $f_\alpha$. To see this for  $\cutPolytope$, note that $f_\alpha(xx\conjTrans) \in \cutPolytope$ for $x \in \mathcal{B}^n_m$. Therefore, the action $f_\alpha$ on a convex combination of rank one matrices in $\cutPolytope$ returns a convex combination of (possibly different) rank one matrices in $\cutPolytope$. To see that $\elliptope$ is also closed under $f_\alpha$, note that the Hadamard product of positive semidefinite matrices is again positive semidefinite, due to the well-known Schur Product Theorem. Moreover, the elements of $f_\alpha(X)$, $X \in \elliptope$, are contained in $\ConvHull\left( \mathcal{B}_m \right)$, since $\ConvHull\left( \mathcal{B}_m \right)$ itself is also closed under the Hadamard product. Specifically,
$(f_\alpha(X))_{ij} = X_{ij} \alpha_i \overline{\alpha}_j \in \ConvHull\left( \mathcal{B}_m \right)$.

\begin{remark}
    We have not explicitly covered the symmetries of $\cutPolytopeNoSS^n_\infty$, see \eqref{eqn_infiniteCutPolytope}. However, the group structures of $\mathcal{B}_\infty$ and $\mathcal{B}^n_\infty$  are similar to those of $\mathcal{B}_m$ and $\mathcal{B}^n_m$ for $m$ finite, and therefore do not warrant special consideration. In particular, $\mathcal{B}^n_\infty$ with the Hadamard product is abelian like $\mathcal{B}^n_m$, although its order is infinite, in contrast to $\mathcal{B}^n_m$. Also the group action $f_\alpha$ for $\alpha \in \mathcal{B}^n_\infty$ is defined similarly as $f_\alpha$ for $\alpha \in \mathcal{B}^n_m$, and there is an infinite number of such group actions.
\end{remark}}

\subsection{Classes of valid inequalities}
\label{section_ROCequivalence}
We show here that the strength of a valid inequality, generated by $Q \in \mathcal{H}_n$, is invariant under rotation of elements in $Q$,  i.e., $f_\alpha(Q)$, and taking the conjugate of $Q$. Thus, each $Q$ in \eqref{deriveCuts} induces a class of valid inequalities.

Consider, for $\cutPolytopeNoSS^n_2$, the triangle inequalities \cite{laurent1996graphic}, given by
\begin{align}
    \label{eqn_realTriangleIneq}
    c_1 X_{ij}+c_2 X_{ik} + c_3 X_{jk} \geq -1, \,\, c \in \{ \pm 1 \}^3, \,\, c_1c_2c_3 = 1.
\end{align}
There are four ways to choose the vector $c$, and we say that triangle inequalities induced by different $c$ are equivalent under rotation of coefficients (\textit{ROC equivalent}). We generalize the notion of ROC equivalence to $\cutPolytope$, $m \geq 2$, see also \cite{jarre2020set}, {by using the symmetries of $\cutPolytope$, as outlined in \Cref{section_groupStructure}.}

\begin{lemma}
    \label{lemma_ROCresult}
        Let  $m, n\geq 2$ be integer numbers,
        $Q \in \mathcal{H}^n$, and $\alpha \in \mathcal{B}^n_m$. Then
        \begin{align}
            \normalfont{\texttt{str}}(Q,m) = \normalfont{\texttt{str}}(Q \Hadamard (\alpha \alpha\conjTrans),m),
        \end{align}
        see \eqref{eqn_strFunctionDef}.
        We say that the cuts induced by $Q$ and $Q \Hadamard (\alpha \alpha\conjTrans)$ are ROC equivalent.
    \end{lemma}
    \begin{proof}
     {It follows from \eqref{eqn_invariantProd} that
       \begin{align}
        \label{eqn_groupTransformProp}
           {\max_{X \in \cutPolytope} \langle Q,X \rangle} ={\max_{X \in \cutPolytope} \langle f_{\alpha}(Q),f_{\alpha}(X) \rangle} = {\max_{X \in \cutPolytope} \langle f_{\alpha}(Q),X \rangle},
       \end{align}
       where the last inequality is due to  the fact that $\cutPolytope$ is closed under the action of $f_{\alpha}$. Likewise, $\elliptope$ is also closed under the action of $f_{\alpha}$, as shown in \Cref{section_groupStructure}. Therefore, \eqref{eqn_groupTransformProp} also holds when replacing $\cutPolytope$ by $\elliptope$.} Thus, by definition of the function \texttt{str}, the lemma follows.
    \end{proof}
      
We provide an explicit example of such an ROC transformation. Let $Q$ and $X$ be Hermitian matrices of order $n$, with $\text{diag}(Q) = \mathbf{0}$. Then,
\begin{align}
    \label{eqn_innerProdMatWorkedOut}
    \langle Q, X \rangle = 2 \sum_{ij \in [n]^2, i < j} \mathrm{Re} \Big( \overline{Q_{ij}}X_{ij} \Big).
\end{align}
Using $(Q \Hadamard (\alpha \alpha\conjTrans) )_{ij} =  Q_{ij} \alpha_i \overline{\alpha_j}$ for $\alpha \in \mathcal{B}^n_m$, it is easy to see how the Hadamard product  transforms~\eqref{eqn_innerProdMatWorkedOut}. However, we simplify  {notation} by considering $(\alpha_0, \alpha_1,\ldots, \alpha_{n-1})^\top \in \mathcal{B}^n_m$ and $\beta = \alpha_0 (1,\alpha_1,\ldots,\alpha_{n-1})^\top \in \mathcal{B}^n_m$. Note that the first column of $\beta \beta\conjTrans$ is given by $\begin{bmatrix}
    1 & \alpha_1 & \ldots & \alpha_{n-1}
\end{bmatrix}^\top$, so that
\begin{align}
    \label{eqn_innerProdBetaHada}
    \frac{1}{2} \langle Q \Hadamard (\beta \beta \conjTrans), X \rangle = \mathrm{Re} \left [ \sum_{j=2}^n  \overline{Q_{1j}} \alpha_{j-1}  X_{1j} + \sum_{ij \in [n]^2, 1<i < j} {\overline{Q_{ij}}\overline{\alpha_{i-1}} }\alpha_{j-1}  X_{ij} \right].
\end{align}
We exploit the above equality to derive the ROC equivalent inequalities in the next section. The following lemma shows that one can also consider the conjugate of matrix $Q$ without changing the strength of the corresponding valid inequality, resulting in {\em the conjugate equivalent} inequality.
\begin{lemma}
\label{lemma_ROCresult2}
       Let  $m, n\geq 2$ be integer numbers and $Q \in \mathcal{H}^n$. Then       
      $\normalfont{\texttt{str}}(Q,m) = \normalfont{\texttt{str}}(\overline{Q},m).$        
\end{lemma}
\begin{proof}
 {For $Z\in {\mathcal H}^n$, the complex conjugate function $Z \to \overline{Z}$ is 
conjugate-linear, invertible, and satisfies an equation similar to \eqref{eqn_invariantProd}, i.e., $\langle \overline{Z}, \overline{Z^\prime} \rangle = \langle Z, Z^\prime \rangle$. Thus, \eqref{eqn_groupTransformProp} is also valid when $f_{\alpha}$ is replaced with the complex conjugate function. Hence, the same arguments that prove \Cref{lemma_ROCresult} also prove \Cref{lemma_ROCresult2}.}
\end{proof}

\begin{example}[MAX-3-CUT]
\label{ex_max3cut}
The maximum-three-cut problem (MAX-3-CUT) is to  partition the vertex set of a graph into $3$ subsets such that the total weight of edges joining different sets is maximized. MAX-3-CUT can be modeled using $\cutPolytopeNoSS^n_3$ as noted by  \citeauthor{goemans2001approximation} \cite{goemans2001approximation}. 
The same authors also derived a  complex SDP relaxation for MAX-3-CUT whose feasible set is $\mathcal{E}^n_3$, see \eqref{eqn_ElliptopeSDP}.

To model MAX-3-CUT on some graph $G = (V,E)$, $|V| = n$, we may associate to each vertex $i \in V$ a variable $x_i \in \mathcal{B}_3$, see \eqref{eqn_discreteSymbolSet}. The value of any variable assignment (i.e., cut) equals the number of edges $\{i,j\} \in E$ for which $x_i \neq x_j$. Note that, if $x_i \neq x_j$, then $\overline{x_i} x_j \in \mathcal{B}_3 \setminus \{ 1 \}$. Since $\{ \mathrm{Re}(z) \, | \, z \in \mathcal{B}_3 \setminus \{ 1 \} \} = -1/2$, we have
\begin{align}
    \label{eqn_triangleStrengthE}
    \frac{2}{3} \mathrm{Re}(1-\overline{x_i} x_j) = \begin{cases}
        1, \text{ if } x_i \neq x_j \\
        0, \text{ else.}
    \end{cases}
\end{align}
Thus, for a graph $G$, the value of the cut induced by $x \in \mathcal{B}^n_3$ is given as follows
\begin{align}
    \label{eqn_complexCutValue}
    v(G,x) = \frac{2}{3} \sum_{\{i,j\} \in E} \mathrm{Re}(1-\overline{x_i} x_j).
\end{align}
For  the complete graph of order 4, denoted by $K_4$, it is not difficult  to verify that $v(K_4,x) \in \{0,3,4,5\}$ for all $x \in \mathcal{B}^4_3$. That is, any 3-cut of $K_4$ cuts either $0$, $3$, $4$ or $5$ edges. By rewriting \eqref{eqn_complexCutValue} for $G = K_4$, we find
\begin{align}
    \sum_{i < j} \mathrm{Re}( \overline{x_i} x_j) = 6 - \frac{3}{2} v(K_4, x) \in \Big\{0, \pm\frac{3}{2},6 \Big\}.
\end{align}
Therefore, the inequality $\mathrm{Re}(X_{ij}+X_{ik}+X_{i\ell}+X_{jk}+X_{j \ell}+X_{\ell k}) \geq -3/2$ is valid for $\cutPolytopeNoSS^n_3$, along with its ROC equivalent inequalities. 
We show in the next section  that this inequality  is not implied by $\mathcal{E}^n_3$, by proving that the strength of the inequality is positive. 
\end{example}

\subsection{Generalized complex triangle and quadrangle inequalities}  
\label{generalized triangle inequality}
In this section, we first generalize the gap inequalities   \cite{laurent1996gap} from $\cutPolytopeNoSS^n_2$ to $\cutPolytopeNoSS^n_m$, with $m > 2$ integer.
Then, we derive some valid inequalities for $\cutPolytope$ for different values of $m$ by exploiting \eqref{deriveCuts},  and compute their strength. In particular,  we show that the generalized complex triangle and  complex quadrangle inequalities may strengthen $\cutPolytope$ for finite $m \geq 2$.

To derive the gap inequalities  from \cite{laurent1996gap}, we set
\begin{align}
    \label{eqn_gammaSigmaDef}
    \gamma(b) := \min_{x \in \{ \pm 1 \}^n} | b^\top x | ~~\text{ and } ~~\sigma(b) := \sum_{i \in [n]}{b_i},
\end{align}
for any $b \in \mathbb{R}^n$, and $B = bb^\top - \mathrm{Diag}\left( b_1^2,\ldots, b_n^2 \right)$. If the context is clear, we omit $b$ in $\gamma(b)$ and $\sigma(b)$. The gap inequality is then defined as
\begin{align}
    \label{eqn_gapLaurent}
    \langle B, X \rangle \geq 2 \sum_{1 \leq i < j \leq n} b_i b_j + \gamma^2 - \sigma^2 \quad \forall X \in \cutPolytopeNoSS^n_2.
\end{align}
Note that \citeauthor{laurent1996gap} \cite{laurent1996gap} define the gap inequality in terms of $\{0,1\}$ variables, rather than $\{ \pm 1 \}$, which explains the discrepancy between \eqref{eqn_gapLaurent} and the gap inequality presented in \cite{laurent1996gap}. We generalize the above inequality to $\mathbb{C}$ in the following lemma.
\begin{lemma}
\label{lemma_gapIneqGeneralized}
Let $b \in \mathbb{C}^n$, and set $B = bb\conjTrans - \mathrm{Diag}\left( |b_1|^2,\ldots, |b_n|^2 \right)$. Then, for 
\begin{align}
    \gamma(b) := \min_{x \in \mathcal{B}^n_m} | b\conjTrans x |,    
\end{align}
and $\sigma(b)$ as in \eqref{eqn_gammaSigmaDef}, we have
\begin{align}\label{complexGap}
    \min_{X \in \cutPolytopeNoSS^n_m} \langle B, X \rangle =2 \, \mathrm{Re} \left( \sum_{1 \leq i < j \leq n} b_i \overline{b}_j \right)+ \gamma^2 - \sigma \overline{\sigma}.
\end{align}
\end{lemma}
\begin{proof} The result follows from the fact that 
    $\gamma^2 = \min_{X \in \cutPolytopeNoSS^n_m} \left\langle bb\conjTrans, X \right\rangle = \min_{X \in \cutPolytopeNoSS^n_m} \left\langle B, X \right\rangle + \| b \|^2$, and $\| b \|^2 = \sigma \overline{\sigma} - 2 \, \mathrm{Re}  \left( \sum_{1 \leq i < j \leq n} b_i \overline{b}_j \right)$.
\end{proof}
We use \Cref{lemma_gapIneqGeneralized} also to prove the following result.

\begin{proposition}
    \label{lemma_completeGraphFourCut}
    Let $m \geq 2$, $n\in \{3,4\}$, $Q_n =  I_n - J_n$. Then
    \begin{align}
        &\max_{X \in \mathcal{E}^n_m} \langle Q_n, X \rangle = n \text{ and } \nonumber \\
        &\max_{X \in \cutPolytope} \langle Q_n, X \rangle = \begin{cases}
            -4 \cos{\left( \frac{2 \lfloor m/3 \rceil \pi}{m}\right)}-2\cos{\left(\frac{4 \lfloor m/3 \rceil \pi}{m}\right)} &\text{ if } n = 3 \text{ and } m \text{ not a multiple of } 3, \\ 
            -2- 8\cos{ \left( \frac{2 \lfloor m/2 \rfloor \pi}{m} \right)} - 2\cos{ \left(\frac{4 \lfloor m /2 \rfloor \pi}{m} \right)} &\text{ if } n = 4 \text{ and } m \text{ odd}, \label{eqn_minCutValue} \\
            n &\text{ else}.
        \end{cases} 
    \end{align}
\end{proposition}
\begin{proof}
For any $Y \in \elliptope$, the value $\langle Q_n, Y \rangle$ provides a lower bound on $\max_{X \in \elliptope} \langle Q_n, X \rangle $. Specifically for $Y =(nI_n -J_n)/(n-1)$, we have $\max_{X \in \elliptope} \langle Q_n, X \rangle \geq     \langle Q_n, Y \rangle = n$. Moreover, we have for all $X \in \elliptope$, $\langle Q_n, X \rangle = n - \langle J_n, X \rangle \leq n$, since $J_n \succeq 0$. Thus $\max_{X \in \elliptope} \langle Q_n, X \rangle = n$.

    For optimization over $\cutPolytope$, note that $(-Q_n)= \mathbf{1}_n \mathbf{1}\conjTrans_n - \text{Diag}(\mathbf{1}_n)$, and we may apply \Cref{lemma_gapIneqGeneralized}, for $b = \mathbf{1}_n$. Consequently, $\sigma(\mathbf{1}_n) = n$, and
    \begin{align}
        \label{eqn_inequalityExpression}
        \max_{X \in \cutPolytope} \langle Q_n, X \rangle = -\min_{X \in \cutPolytope} \langle -Q_n, X \rangle =n^2 - 2 \binom{n}{2}  - \gamma(\mathbf{1}_n)^2  = n- \min_{x \in \mathcal{B}^n_m} |\mathbf{1}\conjTrans x |^2.
    \end{align}
    It remains to determine $\gamma(\mathbf{1}_n) = \min_{x \in \mathcal{B}^n_m} | \mathbf{1}\conjTrans x|$. It is clear that when $n =3$ and $m$ a multiple of 3, or $n=4$ and $m$ even, $\gamma(\mathbf{1}) = 0$ (since then there exist $n$ $m$th roots of unity that sum to 0).

     For $n = 3$ and $m$ not a multiple of 3, geometric arguments from  \cite{myerson1986small} show that the optimal value is attained for $x^* = (1,z,\bar{z})^\top$, where $z = \exp{(  \frac{2 \lfloor m/3 \rceil \pi}{m}\imagUnit)}$. Then,
    \begin{align}
        \gamma(\mathbf{1}_3)^2 = |\mathbf{1}\conjTrans_3 x^*|^2 = \left(1+2 \cos{\left( \frac{2 \lfloor m/3 \rceil \pi}{m} \right) }\right)^2 = 3 +   4 \cos{\left( \frac{2 \lfloor m/3 \rceil \pi}{m}\right)}+2\cos{\left(\frac{4 \lfloor m/3 \rceil \pi}{m}\right)},
    \end{align}
    and the result follows from substitution in \eqref{eqn_inequalityExpression}.

    For $n = 4$ and $m$ odd, similar geometric arguments from \cite{myerson1986small} show that the minimizer of $\gamma(\mathbf{1}_4)$ is given by $x^* = (1,1,z,\bar{z})^\top$, where $z = \exp{(\frac{2 \lfloor m/2 \rfloor \pi}{m} \imagUnit)}$. Using this to compute $\gamma(\mathbf{1}_4)^2$, and substituting the result in \eqref{eqn_inequalityExpression} yields the proof.
\end{proof}

The coefficients of these valid inequalities can be multiplied by elements from $\mathcal{B}^n_m$  without altering their strength, see  \Cref{lemma_ROCresult}. Let us present these ROC equivalent inequalities explicitly below.
\begin{corollary}\label{ROCgeneralizedtriangle}
Let $m \geq 2$, $n\in \{3,4\}$, $Q_n = I_n - J_n$.
    For $n = 3$, the ROC equivalent inequalities of the inequality induced by \Cref{lemma_completeGraphFourCut} read
    \begin{align}
        \label{eqn_ROC_triangleIneq}
        -2 \,\, \mathrm{Re}(\alpha_1 X_{12} + \alpha_2 X_{13} + \overline{\alpha_1} \alpha_2 X_{23}) \leq \max_{X \in \cutPolytopeNoSS^3_m} \langle Q_3, X \rangle,
    \end{align}
    where $\alpha \in \mathcal{B}^2_m$.  For $n = 4$, we have the following  ROC equivalent inequalities
    \begin{align}
        \label{eqn_complexQuadrangleIneq}
       -2 \,\, \mathrm{Re}(\alpha_1 X_{12} + \alpha_2 X_{13} + \alpha_3 X_{14} + \overline{\alpha_1} \alpha_2 X_{23} + \overline{\alpha_1} \alpha_3 X_{24} + \overline{\alpha_2} \alpha_3 X_{34}) \leq \max_{X \in \cutPolytopeNoSS^4_m} \langle Q_4, X \rangle,
    \end{align}
    where $\alpha \in \mathcal{B}^3_m$. Lastly, $\strFunc(Q_n,m) > 1$, see \eqref{eqn_strFunctionDef}, if and only if $\mathrm{gcd}(n,m) = 1$.
\end{corollary}
\begin{proof}
The inequalities \eqref{eqn_ROC_triangleIneq} and \eqref{eqn_complexQuadrangleIneq} are obtained from 
\eqref{deriveCuts} and \eqref{eqn_innerProdBetaHada} where $Q:= I_n - J_n$.

To show that $\texttt{str}$ is positive whenever $\mathrm{gcd}(n,m) = 1$, we consider again separate cases. Let first $n = 3$ and $m \equiv 1 \mod 3$. Along with the earlier assumption that $m \geq 2$, this implies that $m \geq 4$. Then $\lfloor m/3 \rceil = (m-1)/3$. Substituting this in \eqref{eqn_minCutValue} for $n=3$, and using that $\cos{(2z)} = 2\cos^2{(z)}-1$, we find
    \begin{align}
         \max_{X \in \cutPolytopeNoSS^3_m} \langle Q_3, X \rangle = 2-4 \cos{(z_m)} - 4 \cos^2{(z_m)} := g(m), \text{ for } z_m = \frac{2(m-1)\pi}{3m} \text{ and } m \equiv 1 \mod 3.
    \end{align}
    Observe that $g(m)$ is a concave quadratic function in $\cos{(z_m)}$ that is maximized for $\cos{(z_m)} = -1/2 \Rightarrow z_m = 2 \pi /3 + 2k \pi$, $k \in \mathbb{Z}$. The maximum equals $3$, but 
    \begin{align}
        m \geq 4 \text{ and } m \equiv 1 \mod 3 ~~\Rightarrow~~  \cos{(z_m)} \neq \cos{\left( \frac{2 \pi}{3} \right)}.
    \end{align}
    Hence, the maximum value of $3$ is not attained for finite $m \geq 4$ in case $m \equiv 1 \mod 3$. Thus, for $m \equiv 1 \mod 3$,  $\max_{X \in \cutPolytopeNoSS^3_m} \langle Q_3, X \rangle < 3$, which proves that the strength of the corresponding inequality is strictly greater than 1. The proof for other values of $n$ and $m$ follows similarly.
\end{proof}

Thus, the inequalities given by \Cref{ROCgeneralizedtriangle} separate $\mathcal{E}^n_m$ from $\cutPolytope$ only when gcd$(n,m) = 1$. The strength of these inequalities is greater for smaller values of $m$, as in the limit to infinity, the optimal value of the discrete programming problem in \Cref{lemma_completeGraphFourCut} equals $n$.
For numerical evaluation of the strength of these inequalities, see \Cref{table_strTable} in \Cref{section:StrengthOfCursNumeric}.
Note that the inequalities from \Cref{ex_max3cut} can be also derived from \Cref{lemma_completeGraphFourCut} for $n=4$ and $m=3$.

Let us highlight \Cref{lemma_completeGraphFourCut} for the real case, i.e., for $m = 2$. Considering $n=3$, the expressions in \Cref{lemma_completeGraphFourCut} then provide
\begin{align}
    \max_{X \in \cutPolytopeNoSS^3_2}  \langle Q_3, X \rangle = 2,
\end{align}
and since $\mathcal{B}_2 = \{ \pm 1\}$, the inequalities \eqref{eqn_ROC_triangleIneq} then reduce to the well-known triangle inequalities \eqref{eqn_realTriangleIneq} (after appropriate scaling). Hence, the inequalities \eqref{eqn_ROC_triangleIneq} may be considered as {\em generalized complex triangle inequalities}.

Similarly, the inequalities \eqref{eqn_complexQuadrangleIneq} for $n = 4$ can be considered as {\em complex quadrangle inequalities}. For the real case, $m = 2$, we have that $\text{gcd}(n,m) = \text{gcd}(4,2)=2 > 1$. Thus, the quadrangle inequalities are  implied by $\mathcal{E}^4_2$. This clarifies why in the real case, the triangle, pentagonal, heptagonal (etc.) inequalities are well-known, in contrast to real quadrangle inequalities.
Note that real triangle, pentagonal, heptagonal, etc.,  inequalities belong to the family of hypermetric inequalities that are considered as a special case of the gap inequalities \eqref{eqn_gapLaurent}.

 {
The inequalities derived in \Cref{ROCgeneralizedtriangle} are valid for $n \in \{3,4\}$, and therefore can be  applied to the principal submatrices of matrices in $\elliptopeNoSS_{m}^{\widetilde{n}}$ for $\widetilde{n} > 4$.
Thus, the complex triangle and quadrangle inequalities apply to all $n$.
We present this idea more formally in the next section, see \eqref{eqn_triangleSetE}, and exploit it in \Cref{section_numericalResults}.
}
    
\section{An exact description of \texorpdfstring{$\cutPolytopeNoSS^3_3$}{CUT33}}
\label{section_exactDescriptionCut33}
We study $\cutPolytopeNoSS^3_3$ by studying the set
\begin{align}
    \label{eqn_upperTriuPart}
    \mathcal{V}\left( \cutPolytopeNoSS^3_3 \right) := \left\{ x \in \mathbb{C}^3 \,\, \middle| \,\, 
        \begin{bmatrix}
1 & x_1 & x_2 \\
\overline{x_1} & 1 & x_3 \\
\overline{x_2} & \overline{x_3} & 1 
\end{bmatrix} \in \cutPolytopeNoSS^3_3 \right\}.
\end{align}
We define the sets $\mathcal{V}\left( \cutPolytope \right)$, in the reminder of the paper, analogously. It is clear that there exists a bijection between the sets $\mathcal{V}\left( \cutPolytopeNoSS^3_3 \right)$ and $ \cutPolytopeNoSS^3_3$. Since $ \mathcal{V}(\cutPolytopeNoSS^3_3)$ is small, we can tractably compute its facets. 

We first require some intermediate lemmas.
\begin{proposition}
\label{lemma_strengthOfFacet}
    The inequality
    \begin{align}
        \label{eqn_cut33facet}       
        \mathrm{Re} \left( \imagUnit x_1 + e^{\pi \imagUnit/6} x_2 + \imagUnit x_3 \right) \leq \frac{\sqrt{3}}{2},
    \end{align}
    is facet defining for $\mathcal{V}(\cutPolytopeNoSS^3_3)$. Additionally, the three linear inequalities that ensure $x_i \in \ConvHull(\mathcal{B}_3)$ for $i\in[3]$, see \eqref{eqn_convBineqElement},  are also facet-defining.   
    
    The strength of inequality \eqref{eqn_cut33facet} equals $\frac{\sqrt{3}\cos\left(\frac{\pi}{18}\right)}{\cos\left(\frac{\pi}{9}\right)}
 \approx  ~ 1.81521.$
\end{proposition}
\begin{proof}
    We consider $\mathcal{V}(\cutPolytopeNoSS^3_3)$ as a real space of dimension 6. 
    \BlueUpdate{Consider the six vectors $z_{\theta} := (e^{\theta_1 \imagUnit}, e^{\theta_2 \imagUnit}, e^{(\theta_2-\theta_1)\imagUnit})^\top$, 
   where $\theta=(\theta_1,\theta_2)$  and
    \begin{align}
        \theta \in  \left\{ \left(0,0 \right), \left(\frac{2 \pi}{3},0 \right), \left(\frac{4 \pi}{3},0 \right), \left(0,\frac{4 \pi}{3} \right), \left(\frac{4 \pi}{3},\frac{2 \pi}{3} \right), \left(\frac{4 \pi}{3},\frac{4  \pi}{3} \right) \right\}.       
      \end{align}
     These six vectors satisfy the following properties: each $z_{\theta}$ corresponds to the upper triangular entries of
    \begin{align}
        Z_{\theta} := \begin{bmatrix}
            1 \\ e^{-\theta_1 \imagUnit} \\ e^{-\theta_2 \imagUnit}
        \end{bmatrix} \begin{bmatrix}
            1 \\ e^{-\theta_1 \imagUnit} \\ e^{-\theta_2 \imagUnit}
        \end{bmatrix}\conjTrans \in \cutPolytopeNoSS^3_3,
    \end{align}
    as in \eqref{eqn_upperTriuPart}, and therefore, $z_\theta \in \mathcal{V}\left( \cutPolytopeNoSS^3_3 \right)$. Moreover, since $Z_\theta$ is an extreme point of $\cutPolytopeNoSS^3_3$, $z_\theta$ is an extreme point of $\mathcal{V}\left( \cutPolytopeNoSS^3_3 \right)$. Additionally, all the six vectors $z_\theta$ satisfy \eqref{eqn_cut33facet} with equality.

    Let us now consider the six real vectors $\widetilde{z}_\theta := \begin{bmatrix}
            \mathrm{Re}(z_\theta)^\top & \mathrm{Im}(z_\theta)^\top 
        \end{bmatrix}^\top \in \mathbb{R}^6$.}    
    It is not difficult to verify that \BlueUpdate{the $\widetilde{z}_\theta$} are affinely independent.  This fact, together with the fact that all extreme points of $\mathcal{V}(\cutPolytopeNoSS^3_3)$ satisfy  the inequality \eqref{eqn_cut33facet}, implies that \eqref{eqn_cut33facet} is facet-defining. The proof that $x_i \in \ConvHull(\mathcal{B}_3)$ induces three facets follows similarly.

    For computing the strength of the inequality, let $Q$ be the unique Hermitian matrix corresponding to \eqref{eqn_cut33facet}, given by
    \begin{align}
        \label{eqn_QcutDefinition}
        Q = \frac{1}{2} \begin{bmatrix}
            0 & \imagUnit & e^{\pi \imagUnit / 6} \\
            -\imagUnit & 0 & \imagUnit \\
            e^{-\pi \imagUnit / 6} & -\imagUnit & 0
        \end{bmatrix}.
    \end{align}
    We can show that $\max_{X \in \mathcal{E}^3_3} \langle Q, X \rangle =    \frac{3\cos\left(\frac{\pi}{18}\right)}{2\cos\left(\frac{\pi}{9}\right)}$, see 
   \Cref{lemma_maxValueOverE}. By complete enumeration, we obtain 
   $\max_{X \in \cutPolytopeNoSS^3_3} \langle Q, X \rangle = \sqrt{3}/2$, which proves the result.
\end{proof}

\begin{remark}
    \label{remark_facetCut4}
    By similar arguments, one can also show that the cut from \Cref{lemma_completeGraphFourCut}, for $n = 3$ and $m = 4$, is facet-defining for $\mathcal{V}(\cutPolytopeNoSS^3_4)$.
\end{remark}
\begin{lemma}
The ROC equivalent inequalities (see \Cref{lemma_ROCresult})  and the conjugate equivalent inequalities (see \Cref{lemma_ROCresult2}) of facet-defining inequalities of $\mathcal{V}(\cutPolytope)$, are again facet-defining.
\end{lemma}
\begin{proof}
    Let $g(x) \leq c$, $c \in \mathbb{R}$, be a facet-defining inequality for $\mathcal{V}(\cutPolytope) \subseteq \mathbb{C}^{\binom{n}{2}}$. Then  there exist vectors $y^j \in \mathcal{V}(\cutPolytope)$, $j \in [2n]$, that satisfy $g(y^j) = c$ and are affinely independent over the reals. That is,
    \begin{align}
        \label{eqn_affIndepSyst}
    \begin{bmatrix}
        1 & 1 & \cdots & 1 \\
        y^1 & y^2 & \cdots & y^{2n}
    \end{bmatrix} v = \mathbf{0}, \, v \in \mathbb{R}^{2n} \iff v = \mathbf{0}.
    \end{align}
    Additionally, for each such $y^j$, there exists a $Y^j \in \cutPolytope$ such that the vector $y^j$ corresponds to the upper triangular entries of $Y^j$. Let us slightly abuse the notation of \eqref{eqn_upperTriuPart}, and write this relation  as $\mathcal{V}(Y^j) = y^j$, where the linear function $\mathcal{V} : \cutPolytope \to \mathbb{C}^{\binom{n}{2}}$ returns the upper triangular entries of its input matrix.  {We can write $g(x) \leq c$ in terms of matrices as $\langle G, X \rangle \leq c$ for some $G\in {\mathcal H}^n$, and $X \in \cutPolytope$. By construction, 
    \begin{align}
        \label{eqn_facetBeforeTransform}
        \langle G, Y^j \rangle = c.
    \end{align}
    Recall now the symmetries of $\cutPolytope$ (and therefore also $\mathcal{V}(\cutPolytope)$), as outlined in \Cref{section_groupStructure}. In particular, recall the group action $f_{\alpha}(Z) = Z \Hadamard (\alpha \alpha\conjTrans)$. Denote by $\widetilde{g}(x) \leq c$ the inequality that is ROC equivalent with $g(x) \leq c$, following a rotation with some $\alpha \in \mathcal{B}^n_m$. This inequality may be written as $\langle f_\alpha(G), X \rangle \leq c$ for $X \in \cutPolytope$. From \eqref{eqn_invariantProd} and \eqref{eqn_facetBeforeTransform},
 we have $\left\langle f_\alpha(G), f_{\alpha}\left(Y^j \right) \right\rangle = c$.}
   
    Therefore, the vectors $\widetilde{y}^j := \mathcal{V}\left( f_\alpha \left(Y^j \right) \right) \in \mathcal{V}(\cutPolytope)$ satisfy $\widetilde{g}(\widetilde{y}^j) = c$. Note that
    \begin{align}
        \label{eqn_yTildeTransform}
        \widetilde{y}^j = \text{Diag}(\alpha_1 \overline{\alpha}_2, \alpha_1 \overline{\alpha}_3, \ldots, \alpha_{n-1} \overline{\alpha}_n) y^j.
    \end{align}
    Using \eqref{eqn_affIndepSyst} and \eqref{eqn_yTildeTransform}, it follows that the vectors $\widetilde{y}^j$ are also affinely independent, since
    \begin{align}
        \begin{bmatrix}
        1 &  \cdots & 1 \\
        \widetilde{y}^1  & \cdots & \widetilde{y}^{2n}
    \end{bmatrix}v = \text{Diag}(1, \alpha_1 \overline{\alpha}_2, \alpha_1 \overline{\alpha}_3, \ldots, \alpha_{n-1} \overline{\alpha}_n) \begin{bmatrix}
        1 &  \cdots & 1 \\
        y^1  & \cdots & y^{2n}
    \end{bmatrix} v = \mathbf{0}, \, v \in \mathbb{R}^{2n} \iff v = \mathbf{0}.
    \end{align}
    Hence, the result follows. The proof for conjugate equivalent inequalities is similar.
\end{proof}
Let $F$ denote the number of facets of $\mathcal{V}(\cutPolytopeNoSS^3_3)$.
Note that  \eqref{eqn_cut33facet} has 9 ROC equivalent inequalities (counting itself), see \eqref{eqn_innerProdBetaHada}, and its conjugate equivalent inequality also has 9 ROC equivalent inequalities (counting itself). Moreover, each of  the three linear inequalities that ensure $x_i \in \ConvHull(\mathcal{B}_3)$ for $i\in[3]$ has 3 ROC equivalent inequalities (counting itself). Thus,
\begin{align}
    \label{eqn_fLowerB}
    F \geq 18+9 = 27.
\end{align}
We are now ready to show that these 27 inequalities fully describe the set $\mathcal{V}(\cutPolytopeNoSS^3_3)$.
\begin{theorem} 
\label{ThmCut33}
    The set $\mathcal{V}(\cutPolytopeNoSS^3_3)$ admits the following linear description:
\begin{align}
    \label{eqn_facetDefining}
     \mathcal{V}(\cutPolytopeNoSS^3_3) = \Set{ x \in \mathbb{C}^3 | \begin{array}{l} x \in \ConvHull(\mathcal{B}^3_3), \, \mathrm{Re}( \eta x ) \leq \frac{\sqrt{3}}{2}, \mathrm{Re}( \overline{\eta}x ) \leq \frac{\sqrt{3}}{2}, \\[0.2cm]
         \eta = \left( \alpha_1e^{\pi \imagUnit/2}, \alpha_2e^{\pi \imagUnit/6},\overline{\alpha_1} \alpha_2 e^{ \pi \imagUnit/2} \right), \, \alpha \in \mathcal{B}^2_3.
       \end{array}}.
\end{align}
\end{theorem}
\begin{proof}
    The Upper-bound theorem for convex polytopes \cite{mcmullen1970maximum} states the following: for any convex $d$-dimensional polytope $P$ with $v$ vertices, the number of $j$-dimensional faces (see \Cref{def_convexFace} in the appendix) is upper bounded by some explicit number $f_j(v,d)$. For our purposes, we consider $\mathcal{V}(\cutPolytopeNoSS^3_3)$ as 6-dimensional real polytope. As its facets are 5-dimensional faces, the number of facets $F$ is upper bounded by
    \begin{align}
        F \leq f_5(9,6) = 30,
    \end{align}
    see, e.g., \cite[Section 1, Theorem 4]{gale1963neighborly}. Combined with \eqref{eqn_fLowerB}, this implies $27 \leq F \leq 30$. We prove now, by contradiction, that $F = 27$. Thus, assume that $27 < F \leq 30$. If that is the case, then there must exist some facet-defining inequality $\mathrm{Re}( \beta_1 x_1 + \beta_2 x_2 + \beta_3 x_3 ) \leq c$, which is missing from the right hand side of~\eqref{eqn_facetDefining}. Note that the vector $\beta \in \mathbb{C}^3$ contains at least two nonzero entries: if $\beta$ were to contain only a single nonzero entry, the inequality concerns only a single variable, say $x_1$. But the restriction $x_1 \in \ConvHull(\mathcal{B}_3)$ is already included in \eqref{eqn_facetDefining}, and clearly cannot be made tighter. 
    
    Thus, $\beta$ contains two or three nonzero entries. Now there must exist at least 8 other ROC equivalent inequalities, that are also facet-defining. This contradicts the result $F \leq 30$, which completes the proof.
\end{proof}
 {In \Cref{sectionAppendix_facetEnumeration}, we verify \Cref{lemma_strengthOfFacet} and \Cref{ThmCut33} by listing all 27 facets found using \texttt{SageMath}~\cite{sagemath}.}

We refer to the inequalities in \eqref{eqn_facetDefining}, induced by $\eta$, as \textit{the triangle facets (of  $\cutPolytopeNoSS_3^3$)}. 
One can strengthen the CSDP relaxation \ref{eqn_mimoSDPrelax}  by adding the triangle facets to the complex elliptope $\mathcal{E}^n_3$. Let us  denote the resulting feasible set by:
\begin{align}
    \label{eqn_triangleSetE}
    \complexTriSet = \left\{ X \in \mathcal{E}^n_3 
    \, | \, X_J \in \cutPolytopeNoSS^3_3, \, \forall J \subseteq [n], \, |J| = 3 \right\}.
\end{align}
Here, $X_J$ denotes the $|J| \times |J|$ principal submatrix of $X$, with rows and columns indicated by $J$.

\section{An efficient reformulation of a class of  CSDPs}
\label{sect:efficientReformulation}

It is well known that MAX-3-CUT can be modeled using $\cutPolytopeNoSS^n_3$, as demonstrated in \Cref{ex_max3cut}, and first shown by \citeauthor{goemans2001approximation} \cite{goemans2001approximation}. 
To approximate MAX-3-CUT, one can solve a CSDP over $\elliptopeNoSS^n_3$. 
 {On the other hand, \citeauthor{frieze1997improved} \cite{frieze1997improved} approximate MAX-3-CUT by a real SDP having matrix variables of order $n$, that is equivalent to the CSDP over $\elliptopeNoSS^n_3$. Here we specify a class of CSDPs that  can be efficiently reformulated  as  real SDPs.}

Modern SDP solvers solve CSDPs by representing $n \times n$ Hermitian matrices as $2n \times 2n$ symmetric matrices, via
\begin{align}
    \label{eqn_realComplexPSDmapping}
    X \in \mathcal{H}^n, \, X \succeq 0 \iff \widetilde{X} = \begin{bmatrix}
        \mathrm{Re}(X) & \mathrm{Im}(X) \\
        -\mathrm{Im}(X) & \mathrm{Re}(X)
    \end{bmatrix} \in \mathcal{S}^{2n}, \, \widetilde{X} \succeq 0,
\end{align}
 see also \cite{gilbert2017plea}. Consequently, CSDPs with matrix  order $n$,    {require doubling the order to $2n$, and are therefore} computationally more challenging than real SDPs with matrix order $n$. 
      \citeauthor{wang2023real} \cite{wang2023real} introduce a real moment-Hermitian-sum-of-squares hierarchy for complex polynomial optimization problems with real coefficients, without doubling the order.     Moreover, \citeauthor{wang2023real} show that their real hierarchy is equivalent to the complex hierarchy from the literature.
Here, we provide another class of problems for which a CSDP can be equivalently formulated as a real SDP of the same size.

\begin{proposition}
    \label{proposition_realReformulationCondition}   
    Let $U \subseteq \mathcal{H}^n_+$ be a  {closed convex set} that is closed under complex conjugation, and $W \in \mathcal{S}^n$. Then
    \begin{align}
        \max_{X \in U} \, \left\langle W, X \right\rangle = \max_{X \in \mathrm{Re}(U)} \langle W, X \rangle.
    \end{align}
\end{proposition}
 {
\begin{proof}
    As $U$ is closed under complex conjugation, $X \in U \iff \overline{X} \in U$. Additionally, by convexity of $U$, $\mathrm{Re}(X) = (X + \overline{X})/2 \in U$. Since $W$ is a real matrix, $\left\langle W, X \right\rangle = \left\langle W, \mathrm{Re}(X) \right\rangle$. This implies $\max_{X \in U} \left\langle W, X \right\rangle = \max_{X \in U}  \left\langle W, \mathrm{Re}(X) \right\rangle = \max_{X \in \mathrm{Re}\left(U\right)}  \left\langle W, X \right\rangle$.
\end{proof}
}

 {\subsection{MAX-3-CUT}

We investigate the implications of \Cref{proposition_realReformulationCondition}, for the case that the underlying (C)SDP corresponds to a relaxation of MAX-3-CUT, see \Cref{ex_max3cut}.
}
\label{section_max3cutCSDP}

Let us first formulate MAX-3-CUT as a real program. Without loss of generality, we assume that the graph underlying MAX-3-CUT is the complete graph on $n$ vertices, with edge weights $w_{ij} \in \mathbb{R}$, $i, \,j \in [n]$, $i < j$. Following \cite{frieze1997improved}, let $\mathbf{a}^1$, $\friezeVec^2$ and $\friezeVec^3$ be a set of unit vectors in $\mathbb{R}^3$ satisfying 
\begin{align}
    \label{eqn_friezeProp}
     (\friezeVec^i)^\top \friezeVec^j = \begin{cases}
        1, & \text{ if } i = j, \\
        -\frac{1}{2}, & \text{ else.}
    \end{cases}
\end{align}
\citeauthor{frieze1997improved} \cite{frieze1997improved}  model MAX-3-CUT as
\begin{alignat}{2}
\label{eqn_friezeMax3Cut}
    &\max_y \quad && \frac{2}{3} \sum_{i<j} w_{ij} (1-y_i^\top y_j) \\   
    &\text{s.t. } && y_i \in \left\{ \friezeVec^1, \, \friezeVec^2, \, \friezeVec^3 \right\} \quad \forall i \in [n].
\end{alignat}
We investigate the feasible set of this program in terms of matrices, denoted $\mathrm{Re}(\cutPolytopeNoSS^n_3)$. This set is given by
\begin{align}
    \mathrm{Re}(\cutPolytopeNoSS^n_3) & =  \left\{ \mathrm{Re}(Y) \, \mid \, Y \in \cutPolytopeNoSS^n_3 \right\} \\
    &= \ConvHull\left\{ Y \in {{\mathcal S}^n_+}  \, \middle| \, \exists y_1, \ldots, y_n \in \left\{ \friezeVec^1, \, \friezeVec^2, \, \friezeVec^3 \right\} \text{ s.t. } Y_{ij} = y^\top_i y_j \quad \forall i,j \in [n] \right\}.
\end{align}
To understand the second equality above, note that 
the objective in the \citeauthor{frieze1997improved} model and \eqref{eqn_complexCutValue} are similar.
That is, $\mathrm{Re}(\overline{x}_i x_j)$ is equal to the right-hand side of \eqref{eqn_friezeProp}, for $x_i, \, x_j \in \mathcal{B}_3$.

 {As $\cutPolytopeNoSS^n_3$ satisfies the conditions for $U$ in \Cref{proposition_realReformulationCondition}, we have, for $W \in \mathcal{S}^n$,}
\begin{align}
    \label{eqn_optImprovement}
       \max_{X \in \cutPolytopeNoSS^3_n} \left\langle W, X \right\rangle  = \max_{Y \in \mathrm{Re}(\cutPolytopeNoSS^3_n)} \langle W, Y \rangle.
\end{align}
Thus, $\mathrm{Re}(\cutPolytopeNoSS^n_3)$ is strictly smaller than $\cutPolytopeNoSS^n_3$, but  attains the same maxima of real linear forms, as it is the case for MAX-3-CUT.
The same principle holds for the sets
\begin{align}
    \mathrm{Re}(\mathcal{E}^n_3) := \left\{ \mathrm{Re}(X) \, \middle| \, X \in \mathcal{E}^n_3 \right\} = \left\{ X \in \mathcal{S}^n_+ \, \middle| \, \text{diag}(X) = \mathbf{1}_{n}, \, X_{ij} \geq -\frac{1}{2}, \, \forall i,j \in [n] \right\},
\end{align}
and $\mathcal{E}^n_3$, as it was already observed by \citeauthor{goemans2001approximation} \cite{goemans2001approximation}. Note that $\mathrm{Re}(\mathcal{E}^n_3)$ corresponds to the feasible set of the SDP relaxation for MAX-3-CUT by \citeauthor{frieze1997improved} \cite{frieze1997improved}.
However, if the objective matrix $W$ satisfies $\mathrm{Im}(W) \neq \mathbf{0}$, then the complex SDP cannot be reformulated to a real SDP with same size.

Let us now study a relation between $\triangleOp(\mathcal{E}^n_3)$, see \eqref{eqn_triangleSetE}, and 
\begin{align}
\label{eqn_defRealTriSet}
\realTriSet := \{ \mathrm{Re}(X) \, \mid \, X \in \triangleOp(\mathcal{E}^n_3) \}.
\end{align}
To do so, we determine the facets of $\mathrm{Re}\left( \mathcal{V}(\cutPolytopeNoSS^3_3) \right)$ in the following lemma.

\begin{lemma}
    \label{lemma_facetsRealCut3}
    The set $\mathrm{Re}\left( \mathcal{V}(\cutPolytopeNoSS^3_3) \right) := \{ \mathrm{Re}(x) \, | \, x \in  \mathcal{V}(\cutPolytopeNoSS^3_3) \}$, see \eqref{eqn_facetDefining}, is given by 
    \begin{align}
    \label{eqn_upperTriuReCut}
     \mathrm{Re}\left( \mathcal{V}(\cutPolytopeNoSS^3_3) \right) = \Set{ x \in \mathbb{R}^3 | \begin{array}{l}
        x_i \geq -\frac{1}{2}~~ \forall i \in [3], \quad x_1+x_2-x_3 \leq 1, \\[0.2cm]
        x_1-x_2+x_3 \leq 1,\quad -x_1+x_2+x_3 \leq 1
\end{array}}.
\end{align} 
\end{lemma}
\begin{proof}
    Starting from \eqref{eqn_facetDefining}, we consider the following three vectors: $\eta =  \left( e^{\pi \imagUnit/2}, e^{\pi \imagUnit/6},  e^{ \pi \imagUnit/2} \right)$,
\begin{align}
     \eta_1 = \left( e^{4 \pi \imagUnit/3} e^{\pi \imagUnit/2}, e^{\pi \imagUnit/6},  e^{-4\pi \imagUnit/3}e^{\pi \imagUnit/2} \right)
     \text{ and } \eta_2 = \left( e^{2\pi \imagUnit/3}e^{-\pi \imagUnit/2}, e^{-\pi \imagUnit/6},  e^{-2\pi \imagUnit/3}e^{-\pi \imagUnit/2} \right).
\end{align}
Note that $\eta_1$ can be obtained from $\eta$ by performing a rotation of coefficients with $(\alpha_1, \alpha_2) = \left( \exp{(4\pi \imagUnit/3)}, 1 \right)$. Similarly, $\eta_2$ can be obtained by taking $\overline{\eta}$, and then performing the rotation of coefficients with $(\alpha_1, \alpha_2) = \left( \exp{(2\pi \imagUnit/3)}, 1 \right)$.

Thus $\mathrm{Re}(\eta_1 x) \leq \sqrt{3}/2$, and $\mathrm{Re}(\eta_2 x) \leq \sqrt{3}/2$ are both valid inequalities for $\mathcal{V}(\cutPolytopeNoSS^3_3)$, see \Cref{section_ROCequivalence}. Consequently, also the sum of these inequalities is valid for $\mathcal{V}(\cutPolytopeNoSS^3_3)$. That is,
\begin{align}
    \label{eqn_derivedFacet}
    \mathrm{Re}\left( (\eta_1+\eta_2) x\right) = \mathrm{Re}\left( \sqrt{3} x_1 + \sqrt{3} x_2 - \sqrt{3} x_3 \right) \leq \sqrt{3} ~~\Rightarrow ~~\mathrm{Re}\left(  x_1 +  x_2 - x_3 \right) \leq 1,
\end{align}
which corresponds to one of the inequalities given in \eqref{eqn_upperTriuReCut}. The above inequality describes a facet of $\mathrm{Re}\left( \mathcal{V}(\cutPolytopeNoSS^3_3) \right)$, since the vectors $(1,1,1)^\top, \, \left(1, -\frac{1}{2} , -\frac{1}{2} \right)^\top, \, \left(-\frac{1}{2},1, -\frac{1}{2} \right)^\top$ are affinely independent, contained in $\mathrm{Re}\left( \mathcal{V}(\cutPolytopeNoSS^3_3) \right)$, and satisfy \eqref{eqn_derivedFacet} with equality. The other facets in \eqref{eqn_upperTriuReCut} can be found in a similar manner. 

Lastly, it can be shown that \eqref{eqn_upperTriuReCut} contains all facets via a similar argument as the one used in the proof of \Cref{ThmCut33}.
\end{proof}

The facets provided in \Cref{lemma_facetsRealCut3} are also given in \cite[Equation 1.3]{chopra1995facets} (they are stated in terms of $\{0,1\}$ variables rather than $\{ -\frac{1}{2},1 \}$ as in \eqref{eqn_friezeProp}).  However, our derivation from complex space is new.
Using facets from \eqref{eqn_upperTriuReCut}, one can optimize over $\realTriSet$.  {Note also that $\complexTriSet$ satisfies the conditions for $U$ in \Cref{proposition_realReformulationCondition}}. Hence, it is beneficial to optimize over $\realTriSet$ instead of $\complexTriSet$ if the matrix $W$ is real. 

\Cref{table_runTimeDiff} 
investigates the difference in solving times for optimization over $\realTriSet$ and $\complexTriSet$. For various values of $n$, we generate uniformly at random a real matrix $C \in \{-5,-4,\ldots,4,5\}^{n \times n}$, and solve the problem of maximizing $\langle C, X \rangle$ over $X \in \realTriSet$, and over $X \in \complexTriSet$. This maximization is repeated 5 times per value of $n$, and the average running time of those 5 runs is reported in \Cref{table_runTimeDiff}. As solver, we used MOSEK \cite{aps2023mosek}. Note that optimization over $\realTriSet$ and $\complexTriSet$ returns the same objective value since $C$ is real, see 
  \Cref{proposition_realReformulationCondition}. \Cref{table_runTimeDiff} clearly demonstrates that optimization over $\realTriSet$ is more efficient compared to optimization over $\complexTriSet$. The first reason for this is that solving real SDPs is computationally cheaper than solving complex SDPs, see \eqref{eqn_realComplexPSDmapping}. The other reason is that $\realTriSet$ contains less inequalities than $\complexTriSet$; compare  \eqref{eqn_facetDefining} with \eqref{eqn_upperTriuReCut}.

 {The CSDP reformulation approach by \citeauthor{wang2023real} \cite{wang2023real} mentioned in \Cref{section_max3cutCSDP}} does not apply to CSDPs over $\complexTriSet$, as the facets provided in \eqref{eqn_facetDefining} have non-real coefficients. Our generalization,  {\Cref{proposition_realReformulationCondition},} shows that a real reformulation of same size is possible when only the objective is real, and the feasible set is closed under complex conjugation (as is the case for $\complexTriSet$).  

\begin{table}[ht]
\centering
\begin{tabular}{cl|rrrrrrrrr}
\hline
\multicolumn{2}{c|}{Matrix size $n$} & 20 & 30 & 40 & 50 & 60 & 70 & 80 & 90 & 100 \\ \hline
\multicolumn{1}{c|}{\multirow{2}{*}{Solving time (s)}} & $\realTriSet$ & 0.03 & 0.10 & 0.33 & 0.84 & 1.83 & 3.67 & 6.73 & 12.73 & 22.74 \\
\multicolumn{1}{c|}{} & $\complexTriSet$ & 0.22 & 0.80 & 3.02 & 6.98 & 15.22 & 33.17 & 59.34 & 114.07 & 199.72 \\ \hline
\end{tabular}
\caption{Comparison of solving times of optimization over $\realTriSet$ and $\complexTriSet$. }
\label{table_runTimeDiff}
\end{table}

\section{Second semidefinite lifting of \texorpdfstring{$\cutPolytopeNoSS^n_\infty$}{CUT polytope}}
\label{section_secondLiftingCutInf}

In this section we study approximations of {$\cutPolytopeNoSS^n_\infty$}, see \eqref{eqn_infiniteCutPolytope}.
The approximation of $\cutPolytopeNoSS^4_\infty$ obtained from the second semidefinite lifting as proposed by \citeauthor{jarre2020set} \cite{jarre2020set} is denoted here by $\liftingSet{2}$. The matrices in set $\liftingSet{2}$ are obtained as projections of certain Hermitian PSD matrices of order seven.  We propose an approximation of $\cutPolytopeNoSS^4_\infty$ denoted by $\liftingSet{1}$, whose elements are the projections of certain Hermitian PSD matrices of order six.
Despite this difference in size of the lifted space, we show that $\liftingSet{1} = \liftingSet{2}$ (\Cref{lemma_equivalentSize6}).
 Additionally, we show that $\liftingSet{1}$ is also equivalent to the second  semidefinite  lifting of the complex Lasserre hierarchy proposed in \cite{josz2018lasserre} (\Cref{lemma_equivalentSecondDegLift}), whose elements are the projections of certain Hermitian PSD matrices of order ten.
 
 The results from this section imply that one may appropriatly decrease a size of matrices in an CSDP relaxation of {$\cutPolytopeNoSS^n_\infty$}, while keeping the strength of the  relaxation unchanged, see \Cref{lemma_degCompletionGeneral}.
We also show that $\liftingSet{1}$ excludes all the rank 2 extreme points of $\elliptopeNoSS^4_\infty$ (\Cref{thm_noR2EPinLE}). Lastly, we show that all elements of $\liftingSet{1}$ satisfy a valid inequality for $\cutPolytopeNoSS^4_\infty$, derived in \cite{jarre2020set} (\Cref{lemma_validCutForL}).

We begin our analysis with the  following well-known result on  a rank of extreme points of $\mathcal{E}^4_\infty$.  The extreme points of $\mathcal{E}^n_\infty$ have been widely studied, see e.g., \cite{christensen1979note,grone1990extremal,li1994note,loewy1980extreme}.
\begin{lemma}[\cite{loewy1980extreme}]
\label{lemma_extremePointsRootN}
    The extreme points of $\mathcal{E}^n_\infty$ have rank at most $\sqrt{n}$. Moreover, for every $k \leq \sqrt{n}$, the set $\mathcal{E}^n_\infty$ contains rank $k$ extreme points.
\end{lemma}
In case $n \leq 3$, the extreme points of $\mathcal{E}^n_\infty$ have rank 1, and thus $\mathcal{E}^n_\infty = \cutPolytopeNoSS^n_\infty$ for $n \leq 3$. Therefore, in the sequel, we consider the smallest non-trivial case, that is $n = 4$. In this case, $\mathcal{E}^4_\infty$ contains rank~2 extreme points (see \eqref{eqn_extremalRank2Matrices} below), unlike $\cutPolytopeNoSS^4_\infty$, which shows that $\cutPolytopeNoSS^4_\infty$ is strictly contained in $\mathcal{E}^4_\infty$. This motivates the authors of \cite{jarre2020set} to investigate a second semidefinite lifting approximation to $\cutPolytopeNoSS^4_\infty$. To present their lifting, we first require some notation and definitions. 

 {For some $p \in \mathbb{N}$, let $\mathscr{B} \subseteq \mathbb{Z}^{p}$ be a finite basis satisfying $\mathbf{0}_{p} \in \mathscr{B}$. Consider a complex (truncated pseudo-moment) sequence}
\begin{align}
\label{eqn_sequenceProp}
(\mathsf{y}_{\alpha})_{\alpha \in \mathscr{B} - \mathscr{B}}, \text{ satisfying } \mathsf{y}_{\mathbf{0}} = 1 \text{ and } \mathsf{y}_\alpha = \overline{\mathsf{y}}_{-\alpha}, \text{ where } \mathscr{B} - \mathscr{B} := \{ \alpha - \beta \, | \, \alpha, \beta \in \mathscr{B} \}.
\end{align}
We define the \textit{complex moment matrix} $M_{\mathscr{B}}(\mathsf{y})$, indexed by the elements of $\mathscr{B}$,  satisfying
\begin{align}
    \label{eqn_momentMatrix}
    (M_{\mathscr{B}}(\mathsf{y}))_{\alpha, \beta} = \mathsf{y}_{\alpha - \beta}.
\end{align}
By the properties of $\mathsf{y}$, $M_{\mathscr{B}}(\mathsf{y}) \in \mathcal{H}^{|\mathscr{B}|}$ and $\text{diag}(M_{\mathscr{B}}(\mathsf{y})) = \mathbf{1}_{|\mathscr{B}|}$.
Let $\polySet$ be the space of polynomials defined by
\begin{align}
    \label{eqn_complexPolySet}
    \polySet := \left\{ \sum_{\alpha \in \mathbb{Z}^p} f_\alpha x^\alpha \, \middle| \, f_\alpha \in \mathbb{C} \,\, \forall \alpha \in \mathbb{Z}^p \right\}, \, \text{ for } x^\alpha := \prod x^{\alpha_i}_i, \text{ where } x^{\alpha_i}_i = \begin{cases}
        x^{\alpha_i}_i &\text{ if } \alpha_i \geq 0, \\
        (\overline{x_i})^{-\alpha_i} &\text{ if } \alpha_i < 0.
    \end{cases}
\end{align}
Note that $\mathrm{Re}(x) = (x + \overline{x})/2 \in \polySet$.
We set 
\begin{align}
    \label{eqn_posDefMomSet}
    \mathcal{F}(\mathscr{B}) := \left\{ M_\mathscr{B}(\mathsf{y}) \, | \, \mathsf{y}_{\mathbf{0}} = 1 \text{ and } \mathsf{y}_\alpha = \overline{\mathsf{y}}_{-\alpha} \right\} \cap \mathcal{H}^{|\mathscr{B}|}_+.
\end{align}
In this section, we study the sets
\begin{align}
    \label{eqn_twoLiftings}
    \liftingSet{i} := \left \{ X \in \mathcal{E}^4_\infty \mid \exists Z \in \mathcal{F}(\mathscr{B}_i)  \text{ satisfying } Z_{1:4,1:4} = X \right \} \text{ for } i \in [6],
\end{align}
which are defined in terms of the (ordered) bases
\begin{align}
    \label{eqn_explicitBases}
    \mathscr{B}_1 = \left\{ \mathbf{0}_3, \begin{bmatrix}
        1 \\ 0\\0
    \end{bmatrix}, \begin{bmatrix}
        0 \\1 \\0
    \end{bmatrix},\begin{bmatrix}
         0\\0 \\1
    \end{bmatrix}, \begin{bmatrix}
         -1\\1 \\0
    \end{bmatrix},
    \begin{bmatrix}
         -1\\0 \\1
    \end{bmatrix} \right\}, \quad \mathscr{B}_2 = \mathscr{B}_1 \cup \left\{  \begin{bmatrix}
         0\\-1 \\1
    \end{bmatrix} \right\},
\end{align}
and $\mathscr{B}_3$ up to $\mathscr{B}_6$, which will be given later. 

 {Observe that $\mathscr{B}_1$ and $\mathscr{B}_2$ do not contain monomial squares, i.e., $\alpha \in \mathscr{B}_k \Rightarrow |\alpha_i| < 2$ for all $i \in [3]$, $k\in\{1,2\}.$
\Cref{lemma_equivalentSecondDegLift} shows that by adding monomial squares to $\mathscr{B}_1$ does not lead to a tighter approximation  of $\cutPolytopeNoSS^4_\infty$. A similar result  follows for $\mathscr{B}_2$, see \Cref{corollary_equivalentLiftings}.}
An example that will be used throughout is the following:
\begin{align}
\label{eqn_big7Matrix} 
    \left(M_{\mathscr{B}_2}(\mathsf{y})\right)_{\alpha, \beta} = L_\mathsf{y} \left( X_{\alpha, \beta} \right), \text{ for } X =  
    \left [{\begin{smallmatrix}  1& \overline{x_1}& \overline{x_2}& \overline{x_3}& x_1\overline{x_2}& x_1\overline{x_3}& x_2\overline{x_3}\\[0.5 cm] x_1& 1& x_1\overline{x_2}& x_1\overline{x_3}& x_1^2\overline{x_2}& x_1^2\overline{x_3}& x_1x_2\overline{x_3}\\[0.5 cm] x_2& \overline{x_1}x_2& 1& x_2\overline{x_3}& x_1& x_1x_2\overline{x_3}& x_2^2\overline{x_3}\\[0.5 cm] x_3& \overline{x_1}x_3& \overline{x_2}x_3& 1& x_1\overline{x_2}x_3& x_1& x_2\\[0.5 cm] \overline{x_1}x_2& \overline{x_1^2}x_2& \overline{x_1}& \overline{x_1}x_2\overline{x_3}& 1& x_2\overline{x_3}& \overline{x_1}x_2^2\overline{x_3}\\[0.5 cm] \overline{x_1}x_3& \overline{x_1^2}x_3& \overline{x_1}\overline{x_2}x_3& \overline{x_1}& \overline{x_2}x_3& 1& \overline{x_1}x_2\\[0.5 cm] \overline{x_2}x_3& \overline{x_1}\overline{x_2}x_3& \overline{x_2^2}x_3& \overline{x_2}& x_1\overline{x_2^2}x_3& x_1\overline{x_2}& 1\end{smallmatrix}} \right ],
\end{align} 
where $L_\mathsf{y} : \polySet \to \mathbb{C}$ is the linear \textit{Riesz functional}, defined by 
\begin{align}
    \label{eqn_rieszFunctional}
    L_\mathsf{y}(f) = \sum_{\alpha \in \mathbb{Z}^p} f_\alpha \mathsf{y}_\alpha,
\end{align}
see \eqref{eqn_complexPolySet}. Observe also that $M_{\mathscr{B}_1}(\mathsf{y})$ is the upper left $6 \times 6$ block of $M_{\mathscr{B}_2}(\mathsf{y})$.

We refer to the sets $\liftingSet{i}$ as semidefinite liftings of $\cutPolytopeNoSS^4_\infty$, since
\begin{align}
    \label{eqn_chainInclusionCut}
    \cutPolytopeNoSS^4_\infty \subseteq \liftingSet{2} \subseteq \liftingSet{1} \subseteq \mathcal{E}^4_\infty.
\end{align}
\citeauthor{jarre2020set} \cite{jarre2020set} propose $\liftingSet{2} $ as a tighter approximation of $\cutPolytopeNoSS^4_\infty$ than $\elliptopeNoSS^4_\infty$.

\begin{remark}
    \label{remark_basisDiff}
    \citeauthor{jarre2020set} originally present their relaxation as $\liftingSet{3}$, see \eqref{eqn_twoLiftings}, for
    \begin{align} \mathscr{B}_3 =   \left\{ \mathbf{0}_4, \, \begin{bmatrix} 1\\-1\\0\\0\end{bmatrix}, \, \begin{bmatrix} 1\\0\\-1\\0\end{bmatrix}, \, \begin{bmatrix} 1\\0\\0\\-1\end{bmatrix}, \, \begin{bmatrix} 0\\1\\-1\\0\end{bmatrix}, \, \begin{bmatrix} 0\\1\\0\\-1\end{bmatrix}, \, \begin{bmatrix} 0\\0\\1\\-1\end{bmatrix}\right\}. \end{align} 
    The linear function
    \begin{align}
        g(x) = \begin{bmatrix}
            1 & 1 & 1 \\
            -1 &0&0 \\
            0&-1&0 \\
            0&0&-1
        \end{bmatrix} x
    \end{align}
    maps the elements of $\mathscr{B}_2$ to $\mathscr{B}_3$, while preserving equalities in $M_{\mathscr{B}}(\mathsf{y})$, i.e., 
    \begin{align}
        (M_{\mathscr{B}_2}(\mathsf{y}))_{\alpha_1, \alpha_2} = (M_{\mathscr{B}_2}(\mathsf{y}))_{\alpha_3, \alpha_4} \Rightarrow (M_{\mathscr{B}_3}(\mathsf{y}))_{g(\alpha_1), g(\alpha_2)} = (M_{\mathscr{B}_3}(\mathsf{y}))_{g(\alpha_3), g(\alpha_4)}.
    \end{align}
    Hence, $\liftingSet{3} = \liftingSet{2}$. In the sequel, we will use $\liftingSet{2}$ in favour of $\liftingSet{3}$, due to its more compact representation.
\end{remark}

We show now that, despite the smaller size of $\mathcal{F}(\mathscr{B}_1)$ compared to $\mathcal{F}(\mathscr{B}_2)$, see \eqref{eqn_posDefMomSet}, their induced approximations of $\cutPolytopeNoSS^4_\infty$ are equally strong. To do so, we define the following partial order.

\begin{definition}
    \label{def_completionDef}
    Let $\mathscr{B} \subseteq \mathbb{Z}^p$, and let $\widetilde{\mathscr{B}}$ be any subset of $\mathscr{B}$, with $k := |\widetilde{\mathscr{B}}|$. We say that $\mathscr{B}$ completes $\widetilde{\mathscr{B}}$, denoted $\widetilde{\mathscr{B}} \models \mathscr{B}$, if and only if, for each $\widetilde{X} \in \mathcal{F}\left(\widetilde{\mathscr{B}}\right)$, there exists an $X \in \mathcal{F}(\mathscr{B})$ satisfying $X_{1:k,1:k} = \widetilde{X}$. Here, it is implicitly assumed that bases $\widetilde{\mathscr{B}}$ and $\mathscr{B}$ are ordered, and that the first $k$ elements of $\mathscr{B}$ are the elements of $\widetilde{\mathscr{B}}$, in the same order.
\end{definition}

It is not difficult to show the following implication
\begin{align}
    \label{eqn_completionConsequence}
    \mathscr{B}_1 \models \mathscr{B}_2 ~~\Rightarrow~~ \liftingSet{1} = \liftingSet{2},
\end{align}
see \eqref{eqn_twoLiftings}. The condition in \Cref{def_completionDef} may be stated alternatively as: any $\widetilde{X} \in \mathcal{F}\left(\widetilde{\mathscr{B}}\right)$ is \textit{completable} to an $X \in \mathcal{F}(\mathscr{B})$. We provide more details on this in the proof of the following result.

\begin{lemma}
\label{lemma_equivalentSize6}
    $\liftingSet{1} = \liftingSet{2}$
\end{lemma}
\begin{proof}
By \eqref{eqn_completionConsequence}, it suffices to show that $\mathscr{B}_1 \models \mathscr{B}_2$. Thus, we need to verify that  all $X \in \mathcal{F}(\mathscr{B}_1)$ can be completed to a matrix in $Z \in \mathcal{F}(\mathscr{B}_2)$, see \eqref{eqn_big7Matrix}. That is, for given any $X \in \mathcal{F}(\mathscr{B}_1)$, and the corresponding partially specified matrix
    \begin{align}
    \label{eqn_lemmaProofZMat}
    Z =  {\begin{bmatrix}
            X & \begin{smallmatrix}
                X_{3,4} \\
                X_{3,6} \\
                \textit{?} \\
                X_{3,1} \\
                \textit{?} \\
                X_{3,2}
            \end{smallmatrix} \\
            \begin{smallmatrix}
                X_{4,3} &
                X_{6,3} &
                \textit{?} &
                X_{1,3} &
                \textit{?} &
                X_{2,3}
            \end{smallmatrix} & 1
        \end{bmatrix}},
    \end{align}
    can we find (possibly distinct) values for \textit{?} such that $Z \in \mathcal{F}(\mathscr{B}_2)$? Note that the only unspecified entries of $Z$ are at position $(3,7)$ and $(5,7)$ (ignoring the lower triangular part of $Z$). We associate to this pattern of unspecified entries a graph $\mathcal{G}$ of order 7, defined as
    \begin{equation}
    \label{eqn_patternGraph}
    \begin{aligned}
        &\mathcal{G} = (V,E), \, V = [7] \text{ and } \\
        &E = \{ \{i,j\} \, | \, i,j \in V, \,  Z_{ij} \neq \textit{?} \} = \{ \{i,j\} \, | \, 1 \leq i < j \leq 7 \} \setminus \left( \{3,7\} \cup \{5,7 \}\right)
    \end{aligned}
    \end{equation}
    Observe that $\mathcal{G}$ is chordal. Then, by \cite[Theorem 7]{grone1984positive},
    $Z$ can be completed if and only if, every fully specified submatrix (i.e., not containing any \textit{?} values) of $Z$ is positive semidefinite. To investigate this condition, we write $Z_{J}$, $J \subseteq [7]$, for the submatrix of $Z$, indexed by rows and columns in $J$. Before we consider all such fully specified submatrices $Z_J$, we consider first $Z_{\mathcal{J}}$, for $\mathcal{J} := \{1,2,4,6,7 \}$. Note that $Z_{\mathcal{J}}$ is fully specified, and given by 
\begin{equation}
\begin{aligned}
\label{eqn_sumOfSquaresComparison}
\left( Z_{\mathcal{J}} \right)_{ij} = L_{\mathsf{y}}(X_{ij}), \text{ for } X = \left [ \begin{smallmatrix}  1& \overline{x_1}& \overline{x_3}& x_1\overline{x_3}& x_2\overline{x_3}\\[0.5 cm] x_1& 1& x_1\overline{x_3}& x_1^2\overline{x_3}& x_1x_2\overline{x_3}\\[0.5 cm] x_3& \overline{x_1}x_3& 1& x_1& x_2\\[0.5 cm] \overline{x_1}x_3& \overline{x_1^2}x_3& \overline{x_1}& 1& \overline{x_1}x_2\\[0.5 cm] \overline{x_2}x_3& \overline{x_1}\overline{x_2}x_3& \overline{x_2}& x_1\overline{x_2}& 1\end{smallmatrix} \right ],
\end{aligned}
\end{equation}
and $L_{\mathsf{y}}$ as in \eqref{eqn_rieszFunctional}. Note that $P^\top Z_\mathcal{J} P = \overline{Z}_{J^\prime}$ for $P = E_{14}+E_{25}+ E_{31}+E_{42}+E_{53}$ and $J^\prime = \{1,2,3,4,6\}$.
Thus, matrix $Z_\mathcal{J}$ is similar to $\overline{Z}_{J^\prime}$. It follows that
 \begin{align}
     Z_{J^\prime} \text{ is a (fully specified) submatrix of } X \Rightarrow Z_{J^\prime} \succeq 0 \iff \overline{Z}_{J^\prime} \succeq 0 \iff Z_{\mathcal{J}} \succeq 0.
 \end{align}
 Let us now show that for any $J \subseteq [7]$ such that $Z_J$ is fully specified, $Z_J \succeq 0$. We distinguish two cases:
 \begin{enumerate}
     \item $J \subseteq \mathcal{J}$. Then $Z_J$ is a submatrix of $Z_{\mathcal{J}}$, and therefore $Z_J \succeq 0$.
     \item $J \not\subseteq \mathcal{J}$. Since $J \not\subseteq \mathcal{J}$, $3 \in J$ or $5 \in J$. As both $Z_{3,7}$ and $Z_{5,7}$ are unspecified, and $Z_J$ is fully specified, it follows that $7 \notin J$. Thus $J \subseteq [6]$. Consequently, $Z_J$ is a submatrix of $X$, and $X \in\mathcal{F}(\mathscr{B}_1)$ implies $X \succeq 0$, which shows $Z_J \succeq 0$.
 \end{enumerate}
To conclude, every fully specified submatrix of $Z$ is positive semidefinite, and the associated graph $\mathcal{G}$ is chordal. By \cite[Theorem 7]{grone1984positive}, $X \in \mathcal{F}(\mathscr{B}_1)$ can always be completed to a matrix in $Z \in \mathcal{F}(\mathscr{B}_2)$, which implies that $\mathscr{B}_1 \models \mathscr{B}_2$. By \eqref{eqn_completionConsequence}, this completes the proof.    
\end{proof}

We now relate $\liftingSet{1}$,  see \eqref{eqn_twoLiftings}, to the second semidefinite lifting proposed in \cite{josz2018lasserre}. This second lifting is given by $\liftingSet{4}$, where {$\mathscr{B}_4 = \{ \mathbf{0}_3 \} \cup \{ \alpha \in \mathbb{N}^3 \, | \, \sum_{i = 1}^3 \alpha_i \leq 2 \}$}. Note that $|\mathscr{B}_4| = 10 > | \mathscr{B}_1| = 6$. Despite this difference, the induced relaxations of $\cutPolytopeNoSS^4_\infty$ are equivalent, as shown in the following result.
\begin{theorem}    
\label{lemma_equivalentSecondDegLift}
    For $\mathscr{B}_4 =\{ \mathbf{0}_3 \} \cup \{ \alpha \in \mathbb{N}^3 \, | \, \sum_{i = 1}^3 \alpha_i \leq 2 \}$, we have $\liftingSet{4} = \liftingSet{1}$, see \eqref{eqn_twoLiftings}.
\end{theorem}
\begin{proof}
    We start by considering the proof of \Cref{lemma_equivalentSize6} more abstractly. Let $\widetilde{\mathscr{B}} \subseteq \mathscr{B} \subseteq \mathbb{Z}^p$, where $|\widetilde{\mathscr{B}}| = k = |\mathscr{B}|-1$ (Note that this is the case for $\mathscr{B}_1$ and $\mathscr{B}_2$, see \Cref{lemma_equivalentSize6}). Consider the problem of completing some $X \in \mathcal{F}\left(\widetilde{\mathscr{B}}\right)$ to some $Z \in \mathcal{F}(\mathscr{B})$. This $Z$ should be thought of as in \eqref{eqn_lemmaProofZMat}, possibly containing \textit{?} values. Let
    \begin{align}
        \label{eqn_setJDef}
     \mathcal{J} :=  \left \{ i \in \mathbb{N} \, | \, Z_{i,k+1} \neq \textit{?} \right \},
    \end{align}
    so that matrix $Z_\mathcal{J}$, the submatrix of $Z$, is fully specified by $X$. Note that the associated graph $\mathcal{G}$, see e.g.,~\eqref{eqn_patternGraph}, is chordal and we may apply again \cite[Theorem 7]{grone1984positive}. By similar reasoning as in \Cref{lemma_equivalentSize6}, the condition that $Z_\mathcal{J}$ is similar to a submatrix of $X$, is sufficient (although not necessary) for $\widetilde{\mathscr{B}} \models \mathscr{B}$ to hold.

    Following the above steps for specific sets $\widetilde{\mathscr{B}}$ and $\mathscr{B}$, we are able to prove the following relations.
    Starting from 
    \begin{align} 
    \label{eqn_liftingB5}
    \mathscr{B}_5 = \left\{ \mathbf{0}_3, \, \begin{bmatrix} 1\\0\\0\end{bmatrix}, \, \begin{bmatrix} 0\\1\\0\end{bmatrix}, \, \begin{bmatrix} 0\\0\\1\end{bmatrix}, \, \begin{bmatrix} 1\\1\\0\end{bmatrix}, \, \begin{bmatrix} 0\\1\\1\end{bmatrix}\right\}, \end{align}
    we have (details omitted)
    \begin{align}
     \label{eqn_bigExtension}
    \begin{split}
        \mathscr{B}_5 &\models \mathscr{B}_5 \cup \left\{ 
        \begin{bmatrix}
            1 \\ 0 \\ 1 
        \end{bmatrix} \right\} \models \mathscr{B}_5 \cup \left\{ 
        \begin{bmatrix}
            1 \\ 0 \\ 1 
        \end{bmatrix},           
        \begin{bmatrix}
            2\\0 \\0
        \end{bmatrix} \right \} \models \mathscr{B}_5 \cup \left\{ 
        \begin{bmatrix}
            1 \\ 0 \\ 1 
        \end{bmatrix},   
        \begin{bmatrix}
            2\\0 \\0
        \end{bmatrix},  
        \begin{bmatrix}
            0\\2 \\0
        \end{bmatrix}\right \} \\
        &\models \mathscr{B}_5 \cup \left\{ 
        \begin{bmatrix}
            1 \\ 0 \\ 1 
        \end{bmatrix},   
        \begin{bmatrix}
            2\\0 \\0
        \end{bmatrix},   
        \begin{bmatrix}
            0\\2 \\0
        \end{bmatrix},  
        \begin{bmatrix}
            0\\0 \\2
        \end{bmatrix}\right \} = \mathscr{B}_4,
    \end{split} \\
    \label{eqn_B5ToB6}
        \mathscr{B}_5 &\models \mathscr{B}_5 \cup \left\{ 
        \begin{bmatrix}
            -1 \\ 0 \\ 1 
        \end{bmatrix} \right\} \models \mathscr{B}_5 \cup \left\{ 
        \begin{bmatrix}
            -1 \\ 0 \\ 1 
        \end{bmatrix},   
        \begin{bmatrix}
            -1\\1 \\0
        \end{bmatrix} \right \} := \mathscr{B}_6,
    \end{align}
    and starting from $\mathscr{B}_1$ as in \eqref{eqn_explicitBases}, we have
    \begin{align}
        \label{eqn_B1ToB6}
        \mathscr{B}_1 &\models \mathscr{B}_1 \cup \left\{
        \begin{bmatrix}
            1 \\ 1 \\ 0 
        \end{bmatrix} \right\} \models \mathscr{B}_1 \cup \left\{ 
        \begin{bmatrix}
            1 \\ 1 \\ 0 
        \end{bmatrix},   
        \begin{bmatrix}
            0\\1 \\1
        \end{bmatrix} \right \} = \mathscr{B}_6.
    \end{align}
    Combining the implication \eqref{eqn_completionConsequence} (which holds more generally for $\mathscr{B}_i$ and $\mathscr{B}_j$), with equations \eqref{eqn_B5ToB6} and \eqref{eqn_B1ToB6} yields $\liftingSet{1} = \liftingSet{5}$. Since $\liftingSet{5} = \liftingSet{4}$ by \eqref{eqn_bigExtension}, the result follows.
\end{proof}
By combining results of \Cref{lemma_equivalentSize6} and \Cref{lemma_equivalentSecondDegLift}, we obtain the following corollary.
\begin{corollary}
    \label{corollary_equivalentLiftings}
    For all $i,j \in [6]$, $\liftingSet{i} = \liftingSet{j}$.
\end{corollary}
In the sequel, we will only refer to $\liftingSet{1}$ for compactness.
Next, we show that $\liftingSet{1}$ does not contain any of the rank 2 extreme points of $\elliptopeNoSS^4_\infty$. Let us first characterize the set of rank 2 extreme points of $\mathcal{E}^4_\infty$. For this, we require the following definition, see matrix $F$ from \cite[Section 2.2]{li1994note}.
\begin{definition}
\label{def_EGF}
    We say that
    \begin{align}
    G = \begin{bmatrix}
            x_1 & u_{1} & w_{1} & v_{1} \\
            x_2& u_{2} & w_{2} & v_{2} 
        \end{bmatrix} \in \mathbb{C}^{2 \times 4}
        \end{align}
     is an Extremal Gram Factor (EGF) if and only if its columns $x$, $u$, $w$, $v$ have norm 1, and the matrix
    \begin{align}
        \label{eqn_matrixF}
        F := \begin{bmatrix}
            | x_{1} |^2 & x_{1} \overline{x_{2}} & \overline{x_{1}} x_{2} & | x_{2} |^2 \\
            | u_{1} |^2 & u_{1} \overline{u_{2}} & \overline{u_{1}} u_{2} & | u_{2} |^2 \\
            | w_{1} |^2 & w_{1} \overline{w_{2}} & \overline{w_{1}} w_{2} & | w_{2} |^2 \\
            | v_{1} |^2 & v_{1} \overline{v_{2}} & \overline{v_{1}} v_{2} & | v_{2} |^2 
        \end{bmatrix}
    \end{align}
    is non-singular.
\end{definition}
Now $P$, the set of rank 2 extreme points of $\mathcal{E}^4_\infty$, is given by the product of EGFs, i.e.,
\begin{align}
    \label{eqn_extremalRank2Matrices}
    P =\left \{ G\conjTrans G \, | \, G \text{ is an EGF, see  \Cref{def_EGF}} \right \},
\end{align}
as proven in \cite{li1994note} (note that EGFs are defined for general matrix sizes in \cite{li1994note}). Thus, if $A$ is a rank~2 extreme point of $\mathcal{E}^4_\infty$, it must be of the form $A = G\conjTrans G$, where $G$ is an EGF. Given such $A$, the corresponding matrix $G$ is unique up to unitary transformation of its columns. We will use MATLAB like notation for indexing submatrices of $G$, i.e., for some $J \subseteq [4]$, $G_{:,J} \in \mathbb{C}^{2 \times |J|}$ denotes the submatrix obtained by taking all rows of $G$, and columns of $G$ indexed by $J$.

Let us prove several results related to EFGs.
\begin{lemma}
    \label{lemma_invertMat}
    Let $G\in \mathbb{C}^{2 \times 4}$ be an EGF. Then for any  $J \subseteq [4]$, $|J|=2$ the matrix $G_{:,J}$ is invertible.    
\end{lemma}
\begin{proof}
    Proof by contradiction: assume that $G$ is an EGF, and that for some $J \subseteq [n]$, $|J| = 2$, matrix $G_{:,J} \in \mathbb{C}^{2 \times 2}$ is singular. 
    A $2 \times 2$ matrix can only be singular if its second column equals its first column multiplied by some $r \in \mathbb{C}$. Since the columns of $G$ have norm 1, we find that $|r| = 1$. But this implies that $F$, see \eqref{eqn_matrixF}, has two identical rows, and is thus singular. This contradicts the assumption that $G$ is an EGF.
\end{proof}

\begin{lemma}
    \label{lemma_simplePara}
    Let $A \in P$, see \eqref{eqn_extremalRank2Matrices}, i.e., $A$ is an extreme point of $\mathcal{E}^4_\infty$ with $\rankOp(A) = 2$. Then, there exists an EGF $G \in \mathbb{C}^{2 \times 4}$, satisfying 
    \begin{align}
    \label{eqn_EGFniceForm}
        G = \begin{bmatrix}
            1 & u_{1} & w_{1} & v_{1} \\
            0 & u_{2} & w_{2} & v_{2}
        \end{bmatrix} = \begin{bmatrix}
            \mathbf{e} & u & w & v
        \end{bmatrix},
    \end{align}
    where $u_{2}$, $w_{2}$, and $v_{2}$ are nonzero and $\mathbf{e} = [1,0]^\top$.
\end{lemma}
\begin{proof}
    Since $A \in P$, there exists an EGF $\widetilde{G}$ such that $A = \widetilde{G}\conjTrans \widetilde{G}$. Let $z := \widetilde{G}_{:,1} \in \mathbb{C}^2$, and consider the matrix
       $Q := \begin{bmatrix}
            \overline{z}_1 & \overline{z}_2 \\
            -z_2 & z_1
        \end{bmatrix}.$
    It is easy to see that $Q$ is unitary, and $Qz = [1,0]^\top$. Then $G := Q \widetilde{G}$ is an EGF satisfying the properties of the lemma. Note that the entries $u_{2}$,  $w_{2}$ and $v_{2}$ are nonzero, because each 2 by 2 submatrix of $G$ must be invertible (\Cref{lemma_invertMat}).
\end{proof}
In the sequel, we will thus only consider EGFs of the form \eqref{eqn_EGFniceForm}. Note that this simplifies matrix $F$ from~\eqref{eqn_matrixF}. We are now ready to prove the following.
\begin{theorem}
    \label{thm_noR2EPinLE}
    For $P$ as in \eqref{eqn_extremalRank2Matrices}, we have $          \liftingSet{1} \cap P = \emptyset.$   
    \end{theorem}
\begin{proof}
Let $A \in P$. Then, without loss of generality, $A = G\conjTrans G$, where $G$ is an EGF  of the form \eqref{eqn_EGFniceForm}. Proof by contradiction: suppose $A \in \liftingSet{1}$. Then $\exists Z \in \mathcal{F}(\mathscr{B}_1)$, see \eqref{eqn_posDefMomSet}, satisfying $Z_{1:4,1:4} = A$. Let $\ell \in \{5,6\}$ and denote by $Z_{\ell}$ the $5 \times 5$ principal submatrix of $Z$, with rows and columns indexed by $[4] \cup \ell$. Since $Z_{\ell} \succeq 0$, there exists a matrix $G_{\ell}$ such that $Z_{\ell} = G\conjTrans_{\ell} G_{\ell}$. We may assume that $G_{\ell}$ is of the form
    \begin{align}
        \label{eqn_vEllNorm}
        G_{\ell} = \begin{bmatrix}
            G & z_{\ell} \\
            \mathbf{0}^\top_4 & \alpha_{\ell}
        \end{bmatrix}, \text{ with } z_{\ell} \in \mathbb{C}^2, \, \alpha_{\ell} \in \mathbb{C} \text{ and } z_{\ell}\conjTrans z_{\ell} + |\alpha_{\ell}|^2 = 1.
    \end{align}
    Note that the last column of $Z_{\ell}$ is then given by $\begin{bmatrix}z_{\ell}\conjTrans G & 1 \end{bmatrix}\conjTrans$. Moreover, for each $\ell \in \{5,6\}$, precisely two of the entries in $G\conjTrans z_{\ell}$ are determined by $A$. For example, if $\ell = 5$, then we have
    \begin{align}
    \label{eqn_v5system}
        G\conjTrans_{:,\{1,3\}}z_{5} = \begin{bmatrix}
            u\conjTrans w \\ u\conjTrans \mathbf{e}
        \end{bmatrix}, \text{ with } G\conjTrans_{:,\{1,3\}} = \begin{bmatrix}
            \mathbf{e} & w
        \end{bmatrix}\conjTrans = \begin{bmatrix}
            1 & 0 \\
            \overline{w}_{1} & \overline{w}_{2} 
        \end{bmatrix} .
    \end{align}
    The above equations follow from the pattern of equalities in \eqref{eqn_big7Matrix}. In particular, $(G\conjTrans z_{5} )_1 = L_{\mathsf{y}}( x_1 \overline{x_2}) = A_{2,3} = (G\conjTrans G)_{2,3} = u\conjTrans w$.
    
    By \Cref{lemma_invertMat}, $G\conjTrans_{:,\{1,3\}}$ is invertible, hence $z_5$ is uniquely determined by this equation, and thus
    \begin{align}
        z_5 &=  
        \begin{bmatrix}
              u\conjTrans w \\ 
              \big(
              \overline{u}_{1}-\overline{w}_{1}\,
               u\conjTrans w \big)  / \overline{w}_{2} \end{bmatrix}
    \end{align}
        We now claim that $\| z_5 \| = 1$, in which case $\alpha_5 = 0$, by \eqref{eqn_vEllNorm}. To verify this claim, we compute first
        \begin{align}
        \label{eqn_intermediateRealTerm}
            \text{ Re}(u_{1} \overline{w}_{1} u\conjTrans w) = |u_{1}|^2 |w_{1}|^2 + \text{ Re}(u_{1} \overline{w}_{1} \overline{u}_{2} w_{2}),
        \end{align}
        which is a term appearing in the computation of $\| z_5 \|^2 = z_5\conjTrans z_5$. Thus, using \eqref{eqn_intermediateRealTerm}, we find
    \begin{align}
        z\conjTrans_5 z_5 &= |u\conjTrans w |^2 + \frac{|w_{2,1}|^2  +|w_{1}|^2 |u\conjTrans w |^2 - 2 \text{ Re}(u_{1} \overline{w}_{1} u\conjTrans w)}{|w_{2}|^2 } \\
        &= \frac{|u_{1}|^2  +|u\conjTrans w |^2 - 2 \text{ Re}(u_{1} \overline{w}_{1} u\conjTrans w)}{|w_{2}|^2 } \\
        &= \frac{|u_{1}|^2  + |u_{1}|^2 |w_{1}|^2 + |u_{2}|^2 |w_{2}|^2 + 2 \text{ Re}( \overline{u}_{1} w_{1} u_{2} \overline{w}_{2} )- 2 \text{ Re}(u_{1} \overline{w}_{1} u\conjTrans w)}{|w_{2}|^2 } \\
        &= \frac{|u_{1}|^2 - |u_{1}|^2 |w_{1}|^2 + (1-|u_{1}|^2)(1-|w_{1}|^2)}{|w_{2}|^2 } = 1.
    \end{align}
    Vector $z_6$ satisfies the system
    \begin{align}
       G\conjTrans_{:,\{1,4\}}  z_6 =  \begin{bmatrix}
            u\conjTrans v \\
            u\conjTrans \mathbf{e}
        \end{bmatrix},
    \end{align}
    which is similar to \eqref{eqn_v5system}. It is therefore also straightforward to show that $\| z_6 \| = 1$. This implies that $Z$ is of the form
    \begin{align}
        \label{eqn_Zfactorization}
        Z = V\conjTrans V, \text{ for } V = \begin{bmatrix}
            \mathbf{e} & u & w & v & z_5 & z_6
        \end{bmatrix}.
    \end{align}
    Now $Z \in  \mathcal{F}(\mathscr{B}_1) \Rightarrow Z_{5,6} = Z_{3,4} = w\conjTrans v$, see \eqref{eqn_big7Matrix}, while \eqref{eqn_Zfactorization} implies that $Z_{5,6} = z_5\conjTrans z_6$. Thus, it must hold that $z_5\conjTrans z_6$ equals $w\conjTrans v$. However, we have  
        $z_5\conjTrans z_6 = w\conjTrans v + {\text{det}(F)}/{ w_{2} \overline{v}_{2}}.$    
    Since $G$ is an EGF, $\text{det}(F) \neq 0$ (\Cref{def_EGF}) which provides the desired contradiction.
\end{proof}

We now provide a result on $\cutPolytopeNoSS^4_\infty$, showing that it contains all rank 2 points of $\mathcal{E}^4_\infty$, if these are not extreme. For this result, we require the notion of a \textit{perturbation}, see \cite{li1994note}. We say that a nonzero Hermitian matrix $B$ is a perturbation of some $A \in \mathcal{E}^n_\infty$, if there exists some $t > 0$ such that $A \pm t B \in \mathcal{E}^n_\infty$. Thus, if $A$ admits some perturbation $B$, it is not an extreme point of $\mathcal{E}^n_\infty$. Additionally, if $A = G\conjTrans G$, then the perturbation is of the form $B = G\conjTrans R G$ \cite[Theorem 1(a)]{li1994note}), with $\text{diag}(B) = \mathbf{0}$ and $R$ Hermitian.

The following result is also given in \cite{li1994note}, in the proof of the sufficiency part of Corollary 4.
\begin{lemma}
    \label{lemma_rank2PointsInCut}
    Let $A \in \mathcal{E}^4_\infty$, $\mathrm{rk}(A) = 2$. If $A$ is not an extreme point of $\mathcal{E}^4_\infty$ (i.e., $A \notin P$), then $A \in \cutPolytopeNoSS^4_\infty$.
\end{lemma}
\begin{proof}
    We write $A = G\conjTrans G$, where $G \in \mathbb{C}^{2 \times 4}$. Since $A$ is not extreme, it admits a nonzero perturbation $B~=~G\conjTrans R G$, for some $R \in \mathcal{H}^2$. Note that $R$ must be indefinite. Then, there exist strictly positive numbers $t_1, t_2 \in \mathbb{R}$ such that both $I + t_1 R$ and $I - t_2 R$ are rank 1. It follows that the matrix $A + t_1 B = G\conjTrans (I + t_1 R) G$, is also rank 1, and thus contained in $\cutPolytopeNoSS^4_\infty$. Similarly, also $A - t_2 B \in \cutPolytopeNoSS^4_\infty$. Then
    \begin{align}
        \label{eqn_rank2ConvexDecomp}
        A = \frac{t_2}{t_1+t_2}\underbrace{(A+t_1 B)}_{\in \cutPolytopeNoSS^4_\infty} + \frac{t_1}{t_1+t_2}\underbrace{(A-t_2 B)}_{\in \cutPolytopeNoSS^4_\infty} \in \cutPolytopeNoSS^4_\infty.
        \\[-35pt]
    \end{align}
\end{proof}
Equation \eqref{eqn_rank2ConvexDecomp} also shows that $A$ can be written as the convex combination of two extreme points of $\cutPolytopeNoSS^4_\infty$. More generally, it is known that any $A \in \cutPolytopeNoSS^n_\infty$, can be written as a convex combination of at most $n^2 -n+1$ extreme points of $\cutPolytopeNoSS^n_\infty$ \cite[Lemma 3]{jarre2020set}, which follows from Carath\'eodory’s theorem. It is stated in \cite{jarre2020set} that `a smaller bound would help in reducing
the size of the problem for finding a nearest matrix in $\cutPolytopeNoSS^n_\infty$'. We provide such a smaller (optimal) bound in the following result, for general $n$.
\begin{theorem}
    For any $A \in \cutPolytopeNoSS^n_\infty$, there exist $r := \rankOp(A)$ rank one matrices $A_1,\dots, A_r \in \cutPolytopeNoSS^n_\infty$,  such that $A \in \ConvHull\{ A_1, \ldots, A_r \}$.
\end{theorem}
\begin{proof}
    We fix some $n \in \mathbb{N}$, and prove the result by induction. The base case $r = 1$ clearly holds. We assume the induction hypothesis and consider the case $\rankOp(A) = r+1$. Let $A_1$ be any extreme point of $\cutPolytopeNoSS^n_\infty$ such that
    \begin{align}
        \lambda^* := \max \left\{ \lambda  \, \middle| \, (1-\lambda) A_1 + \lambda A \in \cutPolytopeNoSS^n_\infty \right\} > 1,
    \end{align}
    and define $C := (1-\lambda^*) A_1 + \lambda^* A$. Such a matrix $A_1$ exists, since $A$ is not an extreme point of $\cutPolytopeNoSS^n_\infty$ (due to its rank being strictly greater than 1). The matrices $A_1$ and $C$ are the endpoints of a line segment in $\cutPolytopeNoSS^n_\infty$, through $A$. By construction, $C \in \cutPolytopeNoSS^n_\infty$ and $\lambda^* > 1$. Hence,
    \begin{align}
        \label{eqn_AsmallConvHull}
        A = \frac{\lambda^*-1}{ \lambda^*} A_1 + \frac{1}{\lambda^*} C   \Rightarrow A \in \ConvHull\{ A_1, C \}.
    \end{align}
Since $\rankOp(A) = r+1$ and $\rankOp(A_1) = 1$, the rank of $C$ is either $r$ or $r+1$. If $\rankOp(C) = r$, the result follows trivially from \eqref{eqn_AsmallConvHull} and the induction hypothesis. In the case that $\rankOp(C) = r+1$, we have
\begin{align}
    \rankOp\left(\frac{\lambda^*-1}{ \lambda^*} A_1 + \frac{1}{\lambda^*} C\right) = \rankOp(C) \Rightarrow C \in \ConvHull\{ A_1, \widetilde{C} \},
\end{align}
for some $\widetilde{C} \in \cutPolytopeNoSS^n_\infty$ with $\rankOp(\widetilde{C}) = r$. Applying the induction hypothesis on $\widetilde{C}$ proves the result.
\end{proof}
Let us now return to the case $n = 4$, specifically the relation between $\cutPolytopeNoSS^4_\infty$ and $\liftingSet{1}$. We have the following result.

\begin{restatable}{lemma}{finalRankThreeException}
\label{lemma_finalRankThreeExcept}
    The set $\, \cutPolytopeNoSS^4_\infty$ is strictly contained in $\liftingSet{1}$ if and only if there exists a matrix $Y$ satisfying the following: $\rankOp(Y) = 3$, $Y \in  \liftingSet{1} \setminus \cutPolytopeNoSS^4_\infty$, $Y \in \partial \mathcal{E}^4_\infty$, and $Y = \lambda J_4 + (1-\lambda) A$, for some $\lambda \in (0,1)$ and $A \in P$, see \eqref{eqn_extremalRank2Matrices}.
\end{restatable}
\begin{proof}
    See \Cref{section_replacedProofs}.
\end{proof}

Unfortunately, we are not able to prove or disprove the existence of such rank 3 points. Numerical tests, see also \cite{jarre2020set}, lead us to the following conjecture:
\begin{conjecture}
    \label{conjecture_exactnessOfLifting}
    The second semidefinite lifting is exact for $\cutPolytopeNoSS^4_\infty$, i.e., $\liftingSet{1} = \cutPolytopeNoSS^4_\infty$.
\end{conjecture}

 {\begin{remark}
    \Cref{conjecture_exactnessOfLifting} can be connected to the notion of a flat extension from the literature, see e.g.,~\cite{curto1998flat}. The term extension is related to the definition of completion as given in \Cref{def_completionDef}. In that definition, the larger matrix $X$ is considered an extension of the smaller matrix $\widetilde{X}$. Matrix $X$ is considered a flat extension of $\widetilde{X}$ if, along with the properties outlined in \Cref{def_completionDef}, we also have $\rankOp(X) = \rankOp(\widetilde{X})$. In \cite{josz2018lasserre}, there are several results related to such flat extensions, specifically for the bases $\mathbb{N}^p_r := \left\{ \alpha \in \mathbb{N}^p \, | \, \sum_{i=1}^p \alpha_i \leq r \right\}$. 
    Now \cite[Theorem 5.2]{josz2018lasserre} leads to an equivalent reformulation of \Cref{conjecture_exactnessOfLifting}. Namely, \Cref{conjecture_exactnessOfLifting} is true if and only if all $\widetilde{X} \in \liftingSet{1} \subseteq \elliptopeNoSS^n_\infty$ admit a flat extension $X \in \mathcal{F}(\mathbb{N}^3_{3})$, for $\mathcal{F}(\cdot)$ as in \eqref{eqn_posDefMomSet}.
\end{remark}}

We show now that all $X \in \liftingSet{1}$ satisfy a valid inequality for $\cutPolytopeNoSS^4_\infty$, found by the authors of \cite{jarre2020set}. This inequality is given as follows:
\begin{align}
    \label{eqn_Hcut}
    \langle H, X \rangle \leq 6 \quad \forall X \in \cutPolytopeNoSS^4_\infty, \text{ where } H = 
    \begin{bmatrix}
        0 & -\imagUnit & \imagUnit & 1 \\
        \imagUnit & 0 & -\imagUnit & 1 \\
         -\imagUnit & \imagUnit & 0 & 1 \\
          1&1&1&0 
    \end{bmatrix}.
\end{align}
The validity of this cut is proven in \cite{jarre2020set}, and we provide an alternative proof in \Cref{lemma_strengthOfH}.

It is shown in \cite{jarre2020set} that the inequality $\langle H, X \rangle \leq 6$ is not satisfied for all $X \in \mathcal{E}^4_\infty$. We show here that matrices in $\liftingSet{1}$ do satisfy this inequality.
\begin{lemma}
    \label{lemma_validCutForL}
    Let $X \in \liftingSet{1}$. Then $\langle H, X \rangle \leq 6$ for $H$ as in \eqref{eqn_Hcut}. Additionally, for all integers $m \geq 3$  or $m=\infty$, 
    \begin{align}
            \strFunc(H,m)   = \frac{2}{\sqrt{3}} \approx 1.15470.
    \end{align}
\end{lemma}
\begin{proof}
    Let $X \in \liftingSet{1}$, and $Z \in \mathcal{F}(\mathscr{B}_1)$ be the matrix satisfying $Z_{1:4,1:4} = X$, see \eqref{eqn_twoLiftings}. We have
    \begin{align}
        \langle H, X \rangle = 6 - \langle Q, Z \rangle, \text{ where } Q = \frac{1}{2} {
            \begin{bmatrix} 4&0&-2\imagUnit&-2&2\imagUnit&-2\\
            0&0&0&0&0&0\\
            2\imagUnit&0&2&-1-\imagUnit&-1-\imagUnit&0\\
            -2&0&-1+\imagUnit&2&0&1-\imagUnit\\
            -2\imagUnit& 0   &-1+\imagUnit&0&2&-1+\imagUnit\\
            -2&0&0&1+\imagUnit&-1-\imagUnit&2
            \end{bmatrix}.}
    \end{align}
    We claim that $Q \succeq 0$. Then, since also $Z \succeq 0$, we have $6 - \langle Q, Z \rangle \leq 6$, which proves the lemma. To show that $Q \succeq 0$, we compute the Schur complement of $Q$ with respect to $Q_{11} = 2$. The resulting matrix is given by
    \begin{align}%
         \frac{1}{2}\begin{bmatrix}
0 & 0 & 0 & 0 & 0 \\
0 & 1 & -1 & -\imagUnit & \imagUnit \\
0 & -1 & 1 & \imagUnit & -\imagUnit \\
0 & \imagUnit & -\imagUnit & 1 & -1 \\
0 & -\imagUnit & \imagUnit & -1 & 1 
\end{bmatrix} = \frac{1}{2}\begin{bmatrix}
0 \\
\imagUnit \\
-\imagUnit \\
-1 \\
1 
\end{bmatrix} \begin{bmatrix}
0 \\
\imagUnit \\
-\imagUnit \\
-1 \\
1 
\end{bmatrix}\conjTrans \succeq 0.
    \end{align}
    Computing the strength of the inequality $\langle H, X \rangle \leq 6$ is left to the appendix, \Cref{lemma_strengthOfH}.   
\end{proof}
Additionally, elements of $\liftingSet{1}$ also satisfy all the infinite ROC equivalent cuts induced by $H$, see \Cref{lemma_ROCresult}.

To conclude this section, we provide a generalization of \Cref{lemma_equivalentSecondDegLift} for any $n \geq 4$. We define, for $p \geq 3$, bases
\begin{align}
    \label{eqn_secondLiftA}
    \widetilde{\mathscr{A}}^p := \left\{ \alpha \in \{0,1\}^p \, \middle| \, \sum_{i = 1}^p \alpha_i \leq 2 \right\} \subsetneq \mathscr{A}^p := \left\{ \alpha \in \{0,1,2\}^p \, \middle| \, \sum_{i = 1}^p \alpha_i \leq 2 \right\},
\end{align}
where the first $p+1$ elements are $\{ \mathbf{0}_{p} \}$ and the $p$ unit vectors.  Sets $\mathcal{F}(\widetilde{\mathscr{A}}^p)$ and $\mathcal{F}(\mathscr{A}^p)$ are defined analogously to \eqref{eqn_posDefMomSet}. Note that $\mathscr{A}^3 = \mathscr{B}_4$, for $\mathscr{B}_4$ as in \Cref{lemma_equivalentSecondDegLift}. The above bases can be used to approximate $\cutPolytopeNoSS^n_\infty$. If we define sets, for $n \geq 4$,
\begin{align}
    \label{eqn_secondLiftFull}
    \mathbf{L}^n(\widetilde{\mathscr{A}}) = \left\{ X \in \elliptopeNoSS^n_\infty \, \middle| \, \exists Z \in \mathcal{F}(\widetilde{\mathscr{A}}^{n-1}) \text{ satisfying } Z_{1:n,1:n} = X \right\},
\end{align}
and similarly $\mathbf{L}^n(\mathscr{A})$, then $\cutPolytopeNoSS^n_\infty \subseteq \mathbf{L}^n(\mathscr{A}) \subseteq \mathbf{L}^n(\widetilde{\mathscr{A}}) \subseteq \elliptopeNoSS^n_\infty$. We are now ready to present the following result.
\begin{restatable}{lemma}{degCompletionGeneral}
\label{lemma_degCompletionGeneral}
    $\mathbf{L}^n(\widetilde{\mathscr{A}}) = \mathbf{L}^n(\mathscr{A})$
\end{restatable}
\begin{proof}
    See \Cref{section_replacedProofs}.
\end{proof}

 {\begin{remark}
    \Cref{lemma_degCompletionGeneral} allows us to choose a smaller monomial basis, namely a basis without squared variables, without weakening the corresponding CSDP relaxation. There are several results in the literature on choosing a (smaller) monomial basis, though they are limited to real variables. For unconstrained polynomial optimization, the Newton polytope (see e.g., \cite[Section III.A]{lofberg2009pre}) offers a monomial basis which is guaranteed to find sum of squares decompositions of sum of squares polynomials. Another, more broadly applicable, method is proposed in \cite[Algorithm 4.1]{chordalTSSOS}. In contrast to our work, it is not proven in~\cite{chordalTSSOS} that the basis returned by that algorithm preserves the relaxation strength, in comparison with the standard monomials basis.
\end{remark}}

\section{Extreme points of \texorpdfstring{$\mathcal{E}^3_m$}{E3m}} 
\label{section_liftingsOfE3m}

In this section we derive necessary and sufficient conditions for a matrix to be an extreme rank 2 point of $\mathcal{E}^3_m$,  $m > 2$ finite. For any such $m$, we provide an explicit rank 2 extreme point  of $\mathcal{E}^3_m$ (\Cref{lemma_generalRank2EPofE}).
Further, we extend this result for any finite $n$ and $m$, which proves the strict inclusion of $\cutPolytope$ in $\elliptope$ (\Cref{lemma_firstLiftingNeverTight}). 

For $m > 2$, we consider a general rank 2 matrix, parameterized as
\begin{align}
    \label{eqn_rank2Parametrization}
    N = \begin{bmatrix}
        1 & N_{12} & N_{13} \\
        \overline{N}_{12} & 1 & N_{23} \\
        \overline{N}_{13} & \overline{N}_{23} & 1
    \end{bmatrix} = G\conjTrans G \in \mathcal{E}^3_m, ~\text{  for  }~ G= \begin{bmatrix}
        \mathbf{e} & u & v
    \end{bmatrix} = \begin{bmatrix}
        1 & u_{1} & v_{1} \\
        0 & u_{2} & v_{2}
    \end{bmatrix},
\end{align}
where $\| u \| = \| v \| = 1$. We assume that at least one of $u_{2}$ and $v_{2}$ is nonzero (to ensure $\rankOp(N) = 2$).  Note that the above parametrization always exists, see e.g.,~\Cref{lemma_simplePara} and \cite{li1994note}.  We investigate under what conditions $N$ is an extreme point.

A  perturbation of $N \in \mathcal{E}^3_m$  {(with respect to $\elliptopeNoSS^3_m$)} is a  {nonzero Hermitian} matrix $B = G\conjTrans R G$, satisfying $\text{diag}(B) = \mathbf{0}_3$, $R\in {\mathcal H}^2$,  {and for which there exists a $t > 0 $ such that $N \pm tB \in \elliptopeNoSS^3_m$,} see~\cite{li1994note} and also \Cref{section_secondLiftingCutInf}.
The constraint $\text{diag}(B) = \mathbf{0}_3$ implies $\mathbf{e}\conjTrans R \mathbf{e} = R_{11} = 0$, and $u\conjTrans R u = v\conjTrans R v = 0$. The latter system may be written in the following form:
\begin{align}
    \label{eqn_RmatrixEq}
    R=\begin{bmatrix}
        0 & \overline{\alpha} \\
        \alpha & c
    \end{bmatrix}, \qquad
    \begin{bmatrix}
        u_{1} \overline{u}_{2} & \overline{u}_{1} u_{2} & | u_{2} |^2 \\
        v_{1} \overline{v}_{2} & \overline{v}_{1} v_{2} & | v_{2} |^2
    \end{bmatrix} \begin{bmatrix}
        \alpha \\ \overline{\alpha} \\ c
    \end{bmatrix} = 0 \,\, \text{ for } \,\, c \in \mathbb{R}. 
\end{align}
Note the similarity with \eqref{eqn_matrixF}. Any possible perturbation $B$  {of $N$} is of the following form
\begin{align}
    \label{eqn_BmatrixDef}
    B = \begin{bmatrix}
        0 & b_{12} & b_{13} \\
        \overline{b}_{12} & 0 & b_{23} \\
        \overline{b}_{13} & \overline{b}_{23} & 0
    \end{bmatrix} = G\conjTrans
        R
     G = G\conjTrans \begin{bmatrix}
        0 & \overline{\alpha}  \\
        \alpha & c
    \end{bmatrix} G.
\end{align} 
 Recall that $N$ is not an extreme point of $\mathcal{E}^3_m$  {if and only if it admits a perturbation. There exist however simple sufficient conditions that show that a matrix $N$ is not an extreme point, which we provide below.}

\begin{lemma} \label{lemmaRank2N}
    Let  $m > 2$,  and $N \in \mathcal{E}^3_m$ such that $\rankOp(N) = 2$. If all the off-diagonal elements of $N$ are interior points of $\ConvHull{(\mathcal{B}_m)}$, or  {any} off-diagonal element of $N$ is contained in $\mathcal{B}_m$, then $N$ is not an extreme point of $\mathcal{E}^3_m$.
\end{lemma}
\begin{proof}
    By \Cref{lemma_extremePointsRootN}, $N$ is not an extreme point of $\mathcal{E}^3_\infty$. Thus, $N$ admits some perturbation matrix $B$  {with respect to $\elliptopeNoSS^n_\infty$, i.e., there exists some $t^* > 0$ such that $N \pm t B \in \elliptopeNoSS^3_\infty$ for all $t \in [0, t^*]$.}     
    Assuming all off-diagonal elements of $N$ are interior points of $\ConvHull{(\mathcal{B}_m)}$, there exists some $t \in [0, t^*]$ small enough such that $N \pm t B \in \elliptopeNoSS^3_m$, and the result follows.
    
    Let us now assume that $N$ has exactly one  {upper-triangular} off-diagonal element contained in $\mathcal{B}_m$. Then, without loss of generality, we have
    \begin{align}
        \label{eqn_rank2DecompG}
        N = G\conjTrans G, ~\text{ for }~ G  = \begin{bmatrix}
            1 & \kappa & u_{1} \\
            0 & 0 & u_{2}
        \end{bmatrix},
    \end{align}
    where $\kappa$ is one of the $m$ roots of unity, and $u_{2} \neq 0$. The off-diagonal elements of $N$ are given by $\kappa, \, u_1$, and $\overline{\kappa} u_1$ and their complex conjugates.  {Therefore, $N$  cannot have more than one upper-triangular off-diagonal element in $\mathcal{B}_m$.} We  distinguish two cases:
    \begin{enumerate}
        \item The complex number $u_{1}$ is an interior point of $\ConvHull{(\mathcal{B}_m)}$. Again, there exists a perturbation matrix $B$ and $t^* > 0 $ such that $N \pm t B \in \elliptopeNoSS^3_\infty$ for all $t \in [0,t^*]$. Note that, since $N_{12} = \kappa \in \mathcal{B}_\infty$, $B_{12} = 0$. Note that the other off-diagonal elements of $N$ are all interior points of $\ConvHull(\mathcal{B}_m)$. Thus, there exists some small enough $t \in [0, t^*]$ such that $N \pm t B \in \elliptopeNoSS^3_3$, and hence, $N$ is not an extreme point of $\elliptopeNoSS^3_m$.
        
        \item The complex number $u_{1} \in \partial \ConvHull{(\mathcal{B}_m)} \setminus \mathcal{B}_m$. Then $u_{1}$ can be written as $u_{1} = \lambda \delta + (1-\lambda) \eta$, where $\lambda \in (0,1)$ and $\delta, \, \eta$ are distinct $m$-roots of unity.
        \begin{align}
            N &= \begin{bmatrix}
                1 & \kappa & \lambda \delta + (1-\lambda) \eta \\
                \overline{\kappa} &1 & \lambda \overline{\kappa}\delta + (1-\lambda) \overline{\kappa}\eta \\
                \lambda \overline{\delta} + (1-\lambda) \overline{\eta} & \lambda \kappa \overline{\delta} + (1-\lambda) \kappa \overline{\eta} & 1 
            \end{bmatrix} 
            =\lambda \begin{bmatrix}
                1 \\ \overline{\kappa} \\ \overline{\delta}
            \end{bmatrix} \begin{bmatrix}
                1 \\ \overline{\kappa} \\ \overline{\delta}
            \end{bmatrix}\conjTrans + (1-\lambda) \begin{bmatrix}
                1 \\ \overline{\kappa} \\ \overline{\eta}
            \end{bmatrix}
            \begin{bmatrix}
                1 \\ \overline{\kappa} \\ \overline{\eta}
            \end{bmatrix}\conjTrans,
        \end{align}
        so that clearly, $N$ is not an extreme point of $\mathcal{E}^3_m$.
    \end{enumerate}
\end{proof}

Let us denote the  boundary of $\ConvHull{(\mathcal{B}_m)}$ by $\partial \ConvHull{(\mathcal{B}_m)}$.  
Then,  the set that contains the elements from  $\partial \ConvHull{(\mathcal{B}_m)}$ without the elements in $\mathcal{B}_m$ is denoted by
\begin{align}
    \label{eqn_partialBoundarySet}
    \partial \ConvHull{(\mathcal{B}_m)} \setminus \mathcal{B}_m.
\end{align}

 {It follows from \Cref{lemmaRank2N} that any rank 2 extreme point of $\mathcal{E}^3_m$ must have at least one element which is contained in the set \eqref{eqn_partialBoundarySet}, and its off-diagonal elements cannot be contained in $\mathcal{B}_m$. This allows us to characterize rank 2 extreme points of $\elliptopeNoSS^3_m$. We first require the following preparatory lemma.

\begin{lemma}
    \label{lemma_prepatoryLemmaEP2}
    Let $m > 2$ and $N \in \elliptopeNoSS^3_m$ be a rank 2 matrix. Let $K := \{ \{i,j\} \in [3] \times [3] \,\, | \,\,  N_{ij} \in \partial \ConvHull{(\mathcal{B}_m)} \setminus \mathcal{B}_m \}$ and  $f : K \to [m]$ be the function that satisfies $\mathrm{Re}(\overline{\nu}_{f(ij)}N_{ij}) = \cos{(\pi /m )}$, for $\nu$ as in \eqref{eqn_convBineqElement}. If $K \neq \emptyset$, then any {possible} perturbation  $B$ of $N$ must satisfy $\mathrm{Re}(\overline{\nu}_{f(ij)}b_{ij}) = 0$ for all $\{i,j\} \in K$.
\end{lemma}
\begin{proof}
    Suppose $B$ is a perturbation of $N$ (with respect to $\mathcal{E}^3_m$) and $\{i,j\} \in K$. Then by definition of a perturbation, we must have $N \pm tB \in \mathcal{E}^3_m$ for some $t > 0$. In particular, $(N \pm tB)_{ij} \in \ConvHull\left( \mathcal{B}_m \right)$. Considering \eqref{eqn_convBineqElement}, this implies that
    \begin{align}
        \mathrm{Re}\left( \overline{\nu}_{f(ij)}(N \pm tB)_{ij}  \right) \leq \cos{\left( \frac{\pi}{m} \right)} \Rightarrow \cos{\left( \frac{\pi}{m} \right)} \pm t \, \mathrm{Re}\left( \overline{\nu}_{f(ij)} b_{ij} \right) \leq \cos{\left( \frac{\pi}{m}\right)} \Rightarrow \mathrm{Re}\left( \overline{w}_{f(ij)} b_{ij} \right) = 0.
    \end{align}
\end{proof}
}

We now present the characterization of rank 2 extreme points of $\elliptopeNoSS^3_m$.

\begin{proposition}
    \label{lemma_elliptopeExtPointsR2}
   Let  $m > 2$, and $N \in {\mathcal E}^3_m$ be a rank 2 matrix. Further, let $K := \{ \{i,j\} \in [3] \times [3] \,\, | \,\,  N_{ij} \in \partial \ConvHull{(\mathcal{B}_m)} \setminus \mathcal{B}_m \}$ and  $f : K \to [m]$ be the function that satisfies $\mathrm{Re}(\overline{\nu}_{f(ij)}N_{ij}) = \cos{(\pi /m )}$, for $\nu$ as in~\eqref{eqn_convBineqElement}. Matrix $N$ is an extreme point of $\mathcal{E}^3_m$, if and only if the following hold:
\begin{enumerate}
    \item \label{thm_item1} $K \neq \emptyset$;  
    \item \label{thm_item3}  {There does not exist a perturbation $B$ of $N$ satisfying  $\mathrm{Re}(\overline{\nu}_{f(ij)} b_{ij}) = 0$ for all $\{i,j\} \in K$.}
\end{enumerate}
\end{proposition}
\begin{proof}
    ($\Rightarrow$) Let $N$ be a rank 2 extreme point of $\mathcal{E}^3_m$. By \Cref{lemmaRank2N}, not all off-diagonal elements of $N$ can be in the interior of  $\ConvHull(\mathcal{B}_m)$,  {and none of the off-diagonal elements can be contained in $\mathcal{B}_m$.} Thus $K \neq \emptyset$, satisfying \Cref{thm_item1}. Since $N$ is an extreme point, it does not admit a perturbation.     In particular, it does not admit a perturbation that satisfies $\mathrm{Re}(\overline{\nu}_{f(ij)} b_{ij}) = 0$ $\forall \{i,j\} \in K$, so \Cref{thm_item3} is satisfied.

     {($\Leftarrow$) Let $N \in \elliptopeNoSS^3_m$ be a rank 2 matrix and $K\neq 0$. \Cref{lemma_prepatoryLemmaEP2} states that any possible perturbation of a rank 2 matrix must satisfy $\mathrm{Re}(\overline{\nu}_{f(ij)} b_{ij}) = 0$ $\forall \{i,j\} \in K$ when $K\neq \emptyset$. Because $N$ satisfies \Cref{thm_item3}, such a perturbation cannot exist.  Thus, $N$ admits no perturbation, and hence, is an extreme point.}  
    \end{proof} 
Using \Cref{lemma_elliptopeExtPointsR2}, we determine a rank 2 extreme point of $\mathcal{E}^3_m$, for any $m > 2$.
\begin{restatable}{lemma}{generalRankTwoEPofE}
\label{lemma_generalRank2EPofE}
    Fix some integer $m > 2$, and set 
    $$
    G = \begin{bmatrix}
    1 & \frac{1}{2} + \frac{1}{2} \exp(2 \pi \imagUnit / m) & {\sin{\left( \pi/m \right)}} \\
    0 & {\sin{\left( \pi/m \right)}} & \frac{1}{2} + \frac{1}{2} \exp(2 \pi \imagUnit / m)
\end{bmatrix}.$$
Then $N = G\conjTrans G$ is a rank 2 extreme point of $\mathcal{E}^3_m$.
\end{restatable}
\begin{proof}
    See \Cref{section_replacedProofs}.
\end{proof}
Now we can directly show the following.
\begin{corollary}
    \label{lemma_firstLiftingNeverTight}
    For finite $m$ and $n$, $m \geq 2$ and $n \geq 3$, we have $\cutPolytope \subsetneq \mathcal{E}^n_m$.  
\end{corollary}
\begin{proof}
 For the case $m = 2$ and $n = 3$, we take $N = \frac{3}{2}I_3 - \frac{1}{2} J_{3} \succeq 0 \Rightarrow N \in \mathcal{E}^3_2$. Since this $N$ does not satisfy the triangle inequality $N_{12}+ N_{13} + N_{23} \geq -1$, see \eqref{eqn_realTriangleIneq}, $N \notin \cutPolytopeNoSS^3_2$. 
 
\Cref{lemma_generalRank2EPofE} proves that $\cutPolytopeNoSS^n_m \subsetneq \mathcal{E}^n_m$ for all finite $m > 2$ and $n = 3$. The case  $m > 2$ and $n > 3$ follows by considering
\begin{align}
    \label{eqn_extensionN}
    \widetilde{N} = \begin{bmatrix}
        N & \mathbf{0}_{3 \times (n-3)} \\
        \mathbf{0}_{ (n-3) \times 3} & I_{n-3}
    \end{bmatrix} \in \mathcal{E}^n_m \text{, but not in } \cutPolytope,
\end{align}
for $N$ as in \Cref{lemma_generalRank2EPofE}. The same extension as \eqref{eqn_extensionN} for $N = \frac{3}{2}I_3 - \frac{1}{2} J_{3}$ shows that $\cutPolytopeNoSS^n_2 \subsetneq \mathcal{E}^n_2$ for $n > 3$.
\end{proof}

\section{Numerical results}
\label{section_numericalResults}

In this section, we provide some computational results  related to the previous sections.
All CSDPs are first reformulated to equivalent real SDPs and then solved using MOSEK \cite{aps2023mosek} with default settings  {on a server with Intel Xeon Gold 6126 CPU, running at 2.60GHz, with 512 GB RAM and using 8 cores.}

\subsection{Strength of cuts}
\label{section:StrengthOfCursNumeric}
We provide the numerical values of $\strFunc$ for the valid inequalities  stated in  \Cref{lemma_completeGraphFourCut,lemma_strengthOfFacet}, and \Cref{lemma_validCutForL}.
To provide a fair comparison, we have ensured that each matrix $Q$ satisfies $\langle I, Q \rangle = 0$, see also \Cref{remark_fairScaling}. 

Results are provided in \Cref{table_strTable}. Strength values that have not been analytically computed in the previous sections, have now been computed using MOSEK \cite{aps2023mosek}. The strength of the cuts in \Cref{lemma_completeGraphFourCut} tend to $1$ as $m \to \infty$. For $m = 3$, the strongest cut is given by the facet-defining inequalities from \Cref{lemma_strengthOfFacet}. 
\begin{table}[ht]
\centering
\begin{tabular}{l|lllllllll}
\hline
\multirow{2}{*}{Cut   as in:} & \multicolumn{9}{c}{$m$} \\ \cline{2-10} 
 & 2 & 3 & 4 & 5 & 6 & 7 & 8 & 9 & $\infty$ \\ \hline
\Cref{lemma_completeGraphFourCut}, $n=3$ & 1.500 & 1 & 1.500 & 1.146 & 1 & 1.114 & 1.061 & 1 & 1 \\
\Cref{lemma_completeGraphFourCut}, $n=4$ & 1 & 1.333 & 1 & 1.038 & 1 & 1.010 & 1 & 1.004 & 1 \\
\Cref{lemma_strengthOfFacet}, $n=3$ & 1 & 1.815 & 1.169 & 1.077 & 1.075 & 1.011 & 1 & 1 & 1 \\
\Cref{lemma_validCutForL}, $n=4$ & 1 & 1.155 & 1.155 & 1.155 & 1.155 & 1.155 & 1.155 & 1.155 & 1.155 \\ \hline
\end{tabular}
\caption{Numerical values of the strength of various cuts.}
\label{table_strTable}
\end{table}

\subsection{Random objective functions}
\label{Random optimization problems}
We consider the following optimization problem
\begin{align}
\label{RandomQ_example}
\max_{X \in K_m} \langle Q, X \rangle,
\end{align}
for $K_m = \elliptopeNoSS^{n}_m$ or $K_m = \triangleOp(\elliptopeNoSS^{n}_m)$, and $m \in \{ 3,4\}$. 
Here $Q\in  \mathcal{H}^n$, ${\rm Diag}(Q)=\mathbf{0}$, and $\mathrm{Im}(Q)\neq \mathbf{0}$.
The complex elliptope  $\elliptopeNoSS^{n}_m$  is defined in
\eqref{eqn_ElliptopeSDP}, and  $\triangleOp(\elliptopeNoSS^n_3)$ in \eqref{eqn_triangleSetE}.
The set $\triangleOp(\elliptopeNoSS^{n}_4)$  is defined as the set of matrices in $\elliptopeNoSS^{n}_4$ for which each $3 \times 3$ submatrix satisfies \eqref{eqn_ROC_triangleIneq}, the  {16} (ROC equivalent) facet defining inequalities from \Cref{lemma_completeGraphFourCut}, see \Cref{remark_facetCut4}.

We set $n = 100$, and generate 250 matrices $Q$ per value of $m$ in the following way. 
Upper triangular entries of a matrix $Q$ are of the form $a + b \imagUnit$, where $a$ and $b$ are independent random integer variables, drawn uniformly from the set $\{-10, -9,\ldots, 9,10 \}$.
For each such $Q$, we solve \eqref{RandomQ_example} for $K_m = \elliptopeNoSS^{100}_m$ and $K_m = \triangleOp(\elliptopeNoSS^{100}_m)$.
We perform a simple rounding procedure (see e.g., \cite{zhang2006complex}) on the optimal value of the corresponding optimization problem to obtain a lower bound on \eqref{RandomQ_example}, denoted $\texttt{LB}$. The resulting upper and lower bounds for fixed $m$ are averaged over the 250 runs and presented in \Cref{table_randomObjData}.  {The columns `Avg. time (s)' report the average computation time per relaxation in seconds.}
We observe that optimization over $\triangleOp(\elliptopeNoSS^{100}_m)$ provides significantly stronger bounds than optimization over $\elliptopeNoSS^{100}_m$, for both values of $m$. Thus, $\triangleOp(\elliptopeNoSS^{n}_m)$ approximates $\cutPolytopeNoSS^{n}_m$ better than the complex elliptope  $\elliptopeNoSS^{n}_m$.  {To compute bounds over $\triangleOp(\elliptopeNoSS^{100}_m)$, $m\in\{3,4\}$ we add all triangle facets to $\elliptopeNoSS^{100}_m$ at once.}

 {\Cref{table_randomObjData} also reports the total number of (in)equality constraints in the corresponding CSDP in the columns `\# (in)equality cons.'. These numbers can be computed as follows. For $\mathcal{E}^n_m$, we have $n$ equality constraints for the unit diagonal, and $m$ inequalities for each of the $n(n-1)/2$ upper triangular entries, to ensure $X_{ij} \in \ConvHull(\mathcal{B}_m)$. For $\mathbf{T}(\mathcal{E}^n_m)$ we have again $n$ unit diagonal constraints, and $m$ constraints for each of the $n(n-1)/2$ upper triangular entries. Moreover, for each of the $\binom{n}{3}$ principal $3 \times 3$ submatrices, we require an additional number of facets. In case $m = 3$, we need 18 additional facets that define $\cutPolytopeNoSS^3_3$, according to \Cref{ThmCut33} (note that the other 9 facets are already included to ensure $X_{ij} \in \ConvHull(\mathcal{B}_3)$). In case $m = 4$, we add the 16 ROC equivalent inequalities given by \Cref{lemma_completeGraphFourCut}, see \Cref{remark_facetCut4}, for each of the $\binom{n}{3}$ principal $3 \times 3$ submatrices. }

\begin{table}[ht]
\centering
\begin{tabular}{cr|c|c|r|cc|cc}
\hline
\multirow{2}{*}{} &  & \multicolumn{1}{c|}{\multirow{2}{*}{$\elliptopeNoSS^{100}_m$~~}} & \multicolumn{1}{c|}{\multirow{2}{*}{$\triangleOp(\elliptopeNoSS^{100}_m)$~~}} & \multicolumn{1}{c|}{\multirow{2}{*}{\texttt{LB}~~}} & \multicolumn{2}{c|}{ {Avg. time (s)}} & \multicolumn{2}{c}{ {\# (in)equality cons.}} \\
 &  & \multicolumn{1}{c|}{} & \multicolumn{1}{c|}{} & \multicolumn{1}{c|}{} &  {$\scriptstyle \elliptopeNoSS^{100}_m$} &  {$\scriptstyle \triangleOp(\elliptopeNoSS^{100}_m)$} &  {$\scriptstyle \elliptopeNoSS^{100}_m$} &  {$\scriptstyle \triangleOp(\elliptopeNoSS^{100}_m)$} \\ \hline
\multicolumn{1}{c}{\multirow{2}{*}{$m$}} & 3 & 14337.7 & 13290.2 & 9939.7 & 35.4 & 173.0 & 14950 & 2925550 \\
\multicolumn{1}{c}{} & 4 & 14849.3 & 14018.0 & 11509.3 & 40.5 & 169.4 & 19900 & 2607100 \\ \hline
\end{tabular}
\caption{Bounds, computation times and number of (in)equalities for \eqref{RandomQ_example} where  $K_m = \elliptopeNoSS^{n}_m$ or $K_m = \triangleOp(\elliptopeNoSS^{n}_m)$, and $m \in \{ 3,4\}$.
Results are averaged over 250 runs.}
\label{table_randomObjData}
\end{table}

\subsection{MIMO}
\label{section_mimoNumerics} 
The multiple-input multiple-output detection problem is a fundamental problem  in digital communications. The multiple-input multiple-output channel can be modelled as follows: given a complex channel matrix $D \in \mathbb{C}^{k \times n}$, we observe the vector of received signals $$r := Dc + \sigma v,$$ where $\sigma > 0$, $c \in \mathcal{B}^n_m$ is the unobserved sent signal and $v$ is an unobserved vector of noise. The parameter $\sigma$ governs the so-called signal to noise ratio, see~\cite{jiang2021tightness}. Observing only $D$ and $r$,  MIMO is to retrieve the original signal $c$. We refer to e.g.,~\cite{jiang2021tightness,lu2019tightness,Yang2015FiftyYO} for more details on  MIMO. The maximum likelihood estimator (MLE) of $c$ is
\begin{align}
    \label{eqn_mle1NpHard}
    \argmin_{x \in \mathcal{B}^n_m} \left\| Dx - r \right\|^2.
\end{align}
The above can be approximated by solving instead
\begin{align}
    \label{eqn_MLEsdp1}
    \min_{X \in K_m} \left\langle \begin{bmatrix}
        r\conjTrans r & -r\conjTrans D \\
        -D\conjTrans r & D\conjTrans D
    \end{bmatrix}, X \right\rangle,
\end{align}
for $K_m = \elliptopeNoSS^{n+1}_m$ or $K_m = \triangleOp(\elliptopeNoSS^{n+1}_m)$. 
The complex elliptope  $\elliptopeNoSS^{n+1}_m$  is defined in \eqref{eqn_ElliptopeSDP},  $\triangleOp(\elliptopeNoSS^{n+1}_3)$ in \eqref{eqn_triangleSetE}, and $\triangleOp(\elliptopeNoSS^{n+1}_4)$ in \Cref{Random optimization problems}. 

 We investigate  tightness of our new relaxations numerically. We consider $m \in \{3,4\}$ and solve \eqref{eqn_MLEsdp1} for different choices of $K_m$.
Specifically, we set $n = 99$, and let $\sigma \in \{1,2,3\}$. For each combination of $m$ and $\sigma$, we generate \numRunsMimo{} matrices $D \in \mathbb{C}^{109 \times n}$ and vectors $v \in \mathbb{C}^{109}$; these are generated by drawing independent standard complex Gaussians\footnote{The number $a + b \imagUnit$ is said to be a standard complex Gaussian if $a$ and $b$ are independent, normally distributed random variables with mean $0$ and variance $1/2$ \cite[Definition 24.2.1]{lapidoth2017foundation}.}. For each such instance, we solve \eqref{eqn_MLEsdp1} for the different choices of $K_m$, and track the rate at which these CSDP relaxations return a (numerical) rank 1 solution. A returned solution matrix is deemed numerically rank 1 if its second largest eigenvalue is strictly smaller than $\zeroPrecisionMimo{}$. If a CSDP relaxation returns a rank 1 solution, the CSDP is said to be tight, since the optimal rank 1 solution can be used to obtain a provably optimal solution to \eqref{eqn_mle1NpHard}.  {Recall that \Cref{table_randomObjData} contains the number of (in)equality constraints for each $K_m$.}

The results are presented in \Cref{table_mimoTable} for $m = 3$, and \Cref{table_mimoTableM4} for $m = 4$.  {Both tables report the average computation time in seconds, for each relaxation  and each $\sigma$.} We see that adding the facet-defining inequalities of $\cutPolytopeNoSS^3_3$, see \eqref{eqn_facetDefining}, for $m=3$ ensures that the CSDP relaxation is tight at a reasonable rate. A similar observation can be made for $m = 4$, see \Cref{table_mimoTableM4}.  As expected, for increasing values of $\sigma$, the CSDP with facet inequalities is tight less often.
However, without the facet inequalities, the CSDP is observed to be tight only once out of the 1200 trials.  {The computation times for $\mathbf{T}\left(\mathcal{E}^{100}_m\right)$ are significantly longer than for $\mathcal{E}^{100}_m$. Additionally, we observe that the computation times depend on the rank 1 rates. Indeed, for $\mathcal{E}^{100}_m$, $m \in \{3,4\}$, both the rank 1 rates and the computation times are approximately constant for varying $\sigma$. 
In contrast, for $\mathbf{T}\left(\mathcal{E}^{100}_m\right)$, $m \in \{3,4\}$,  the computation times differ for varying $\sigma$. In particular, $\mathbf{T}\left(\mathcal{E}^{100}_4\right)$, $\sigma =3$, attains the lowest rank 1 rate, and highest average computation time.}

\begin{table}[ht]
\centering
\begin{tabular}{l|rrr|rrr}
\hline
& & & & & & \\[-8pt]
\multicolumn{1}{c|}{\multirow{2}{*}{$K_m$}} & \multicolumn{3}{c|}{$\sigma$} & \multicolumn{3}{c}{ {Avg. time (s) per   $\sigma$}} \\
\multicolumn{1}{c|}{} & 1~~~ & 2~~~ & 3~~~ & 1~~~ & 2~~~ & 3~~~ \\ \hline\rule{0pt}{2.6ex}
$\mathcal{E}^{100}_3$ & 0.2\% & 0.0\% & 0.0\% & 56.8 & 52.7  & 51.0  \\
$\mathbf{T}\left(\mathcal{E}^{100}_3\right)$ & 50.8\% & 54.5\% & 51.3\% & 331.7 & 320.1 & 310.5 \\ \hline
\end{tabular}
\caption{Average rate  {and computation time} (over \numRunsMimo{} runs) at which \eqref{eqn_MLEsdp1}, the CSDP relaxation of  MIMO  for $m=3$ returns a rank 1 solution. 
}
\label{table_mimoTable}
\end{table}

\begin{table}[ht]
\centering
\begin{tabular}{l|rrr|rrr}
\hline
& & & & & & \\[-8pt]
\multicolumn{1}{c|}{\multirow{2}{*}{$K_m$}} & \multicolumn{3}{c|}{$\sigma$} & \multicolumn{3}{c}{ {Avg. time (s) per   $\sigma$}} \\
\multicolumn{1}{c|}{} & 1~~~ & 2~~~ & 3~~~ & 1~~~ & 2~~~ & 3~~~ \\ \hline\rule{0pt}{2.6ex}
$\mathcal{E}^{100}_4$ & 0.0\% & 0.0\% & 0.0\% & 65.0 & 60.0 & 59.3 \\
$\mathbf{T}\left(\mathcal{E}^{100}_4\right)$ & 49.7\% & 38.7\% & 2.5\% & 271.6 & 290.5 & 342.4 \\ \hline
\end{tabular}
\caption{Average rate  {and computation time} (over \numRunsMimo{} runs) at which \eqref{eqn_MLEsdp1}, the CSDP relaxation of MIMO  for $m = 4$ returns a rank 1 solution. 
}
\label{table_mimoTableM4}
\end{table}

\subsection{Angular synchronization}
In the angular synchronization problem \cite{bandeira2017tightness}, one is given a matrix $C := cc\conjTrans + \sigma W \in \mathbb{C}^{n \times n}$, where $c \in \mathcal{B}^n_{\infty}$ is an unobserved signal, $\sigma > 0$, and $W \in \mathcal{H}^n$ models noise in receiving the signal $c$, which one attempts to retrieve. The maximum likelihood estimator  of $c$ is given by $\argmax_{x \in \mathcal{B}^n_{\infty}}x \conjTrans C x$, which may be approximated by
\begin{align}
    \label{eqn_angSynchroExample}
    \argmax_{X \in K} \left\langle cc\conjTrans + \sigma W, X \right\rangle,
\end{align}
  for $K = \elliptopeNoSS^n_\infty$, or some  second lifting of $\cutPolytopeNoSS^n_\infty$ such as \eqref{eqn_secondLiftFull}.

We investigate, for various $\sigma$, the rate at which the above CSDP returns a rank~1 solution for different choices of $K$. Specifically, we investigate the strength of a parametrized relaxation of $\cutPolytopeNoSS^n_\infty$, induced by basis $\mathscr{C}_p$, for $p \in [0,1]$. This basis contains all $n$ vectors $\alpha_i \in \{0,1\}^{n-1}$ satisfying $\sum_{i = 1}^{n-1} \alpha_i \leq 1$, plus the fraction $p$ of vectors $\alpha \in \{0,1\}^n$ satisfying $\sum_{i = 1}^{n-1} \alpha_i = 2$, chosen uniformly at random (and rounded to nearest integer). The number of elements in this basis can therefore be computed as
\begin{align}
    |\mathscr{C}_p | = n+ \left \lfloor p \binom{n-1}{2} \right\rceil.
\end{align}
The induced relaxation of $\cutPolytopeNoSS^n_\infty$ is denoted $\mathbf{L}^n(\mathscr{C}_p)$, defined analogously to \eqref{eqn_secondLiftFull}. This relaxation is closely related to the relaxations considered in \Cref{section_liftingsOfE3m}; note that 
\begin{align}
    \cutPolytopeNoSS^n_\infty \subseteq \mathbf{L}^n(\mathscr{C}_1) = \mathbf{L}^n(\widetilde{\mathscr{A}}) \subseteq \mathbf{L}^n(\mathscr{C}_p) \subseteq \mathbf{L}^n(\mathscr{C}_0) = \elliptopeNoSS^n_\infty \quad \forall p \in [0,1].
\end{align}

We fix $n = 25$, $c = \mathbf{1}_n$, and vary the level of noise $\sigma \in \left\{ (2/3)\sqrt{n},\sqrt{n},(4/3)\sqrt{n} \, \right\}$. The chosen levels of $\sigma$ are in line with \cite[Figure 2]{bandeira2017tightness}, where it is empirically shown that for $\sigma = (1/3)\sqrt{n}$ and $K = \elliptopeNoSS^n_\infty$, \eqref{eqn_angSynchroExample} very often admits an optimal rank 1 solution. Since we test stronger relaxations than $\elliptopeNoSS^n_\infty$, we have therefore chosen larger values of $\sigma$. We generate 100 instances of Hermitian matrices $W$, for which the upper triangular entries are independent standard complex Gaussians, and track the rate at which the different relaxations, and varying levels of $\sigma$, return rank 1 solutions (with the same zero precision of $\zeroPrecisionMimo{}$ as in \Cref{section_mimoNumerics}). Results are presented in \Cref{table_angularSynchro}. There, `\#cons.' denotes the number of (complex) equality constraints appearing in the CSDP, and `Avg.~T.~(s)' stands for the average computation time per relaxation in seconds. Note also that $|\mathscr{C}_p|$ denotes the size of the corresponding CSDP, which is equivalent to a real SDP of size $2|\mathscr{C}_p|$.  

At the tested levels of $\sigma$ it can be observed that increasing the relaxation size (i.e., $p \to 1$) provides significantly more accurate solutions. For $p = 1$, the CSDP is always observed to return a rank 1 solution, and already $p = 0.75$ offers near-perfect rank 1 rates. The drawback is that the running times also greatly increase. However, in practice, if one is interested in computing the MLE of the unobserved signal $c$, one should not start by solving the $\mathbf{L}^n(\mathscr{C}_1)$ or $\mathbf{L}^n(\mathscr{C}_{0.75})$ relaxation; it is more efficient to solve a smaller relaxation first, say $\mathbf{L}^n(\mathscr{C}_{0.25})$, and inspect the optimal solution. If the optimal solution is rank 1, it provides the MLE of $c$. If it is not rank 1, one can increase the value of $p$ and try again, continuing so until an optimal rank 1 matrix is observed.

\begin{table}[ht]
\centering
\begin{tabular}{lrrr|rrr}
\hline
&&&&&& \\[-8pt]
\multicolumn{1}{c}{\multirow{2}{*}{$p$}} & \multicolumn{1}{c}{\multirow{2}{*}{$|\mathscr{C}_p|$}} & \multicolumn{1}{c}{\multirow{2}{*}{Avg. time (s)}} & \multicolumn{1}{c|}{\multirow{2}{*}{$\#$cons.}} & \multicolumn{3}{c}{$\sigma / \sqrt{n}$} \\
\multicolumn{1}{c}{}& \multicolumn{1}{c}{} & \multicolumn{1}{c}{} & \multicolumn{1}{c|}{} & $2/3$ & $1$ & $4/3$ \\ \hline
0 & 25& 0.16 &25& 18\% & 4\% & 1\% \\
0.25 & 94& 29 &612& 61\% & 34\% & 27\% \\
0.5 & 163& 316&1947& 96\% & 88\% & 80\% \\
0.75 & 232 & 1975&4063& 99\% & 98\% & 99\% \\
1 & 301& 7946 &6925& 100\% & 100\% & 100\% \\ \hline
\end{tabular}
\caption{Average rate (over 100 runs) at which the CSDP relaxation of the angular synchronization problem \eqref{eqn_angSynchroExample}, over feasible sets $\mathbf{L}^n(\mathscr{C}_p)$ returns a rank 1 solution.}
\label{table_angularSynchro}
\end{table}

\section{Conclusions and future work}
\label{section_conclusions}
In this paper we study the complex cut polytope $\cutPolytope$, and its approximations  by semidefinite liftings. The considered approximations of  $\cutPolytope$ are in general not exact, but we investigate under what conditions they are, see  \Cref{ThmCut33}.

Our first approximation of $\cutPolytope$ is the complex elliptope $\elliptope$. To strengthen it, we add valid inequalities. In \Cref{section_strFramework} we introduce a framework for numerically comparing valid inequalities, and derive a number of cuts. In \Cref{section_exactDescriptionCut33} we determine some facet defining inequalities of  $\cutPolytopeNoSS^3_3$, and prove that these facets lead to an exact description of $\cutPolytopeNoSS^3_3$ (\Cref{ThmCut33}).  {In \Cref{sect:efficientReformulation} we show} that a CSDP whose feasible set is closed under complex conjugation and objective  function contains only real coefficients, can be equivalently reformulated as a real SDP of the same size (\Cref{proposition_realReformulationCondition}).

In \Cref{section_secondLiftingCutInf}, we consider the complex cut polytope $\cutPolytopeNoSS^n_\infty$. We derive several   new results for $n = 4$, the smallest value for which $\mathcal{E}^n_\infty$ is not exact (\Cref{lemma_equivalentSecondDegLift,thm_noR2EPinLE}). For general $n$ we provide a method for reducing the size of a second semidefinite lifting
without weakening the approximation of $\cutPolytopeNoSS^n_\infty$ (\Cref{lemma_degCompletionGeneral}). In \Cref{section_liftingsOfE3m} we investigate the extreme points of $\elliptope$ (finite $m$). We find an infinite family of rank 2 extreme points, which proves that the first semidefinite lifting of $\cutPolytope$ is never exact (\Cref{lemma_firstLiftingNeverTight}). 

In \Cref{section_numericalResults} we investigate numerically the value  of adding   the valid inequalities introduced in \Cref{section_strFramework}  to  $\elliptopeNoSS^{n}_m$, $m=3,4$ for CSDPs with randomly generated objectives and the MIMO detection problem. 
The numerical results show that adding our cuts significantly improves the bounds  as well as greatly increases the rate at which the CSDPs return rank 1 solutions. We also test second semidefinite liftings for the angular synchronization problem, and observe  that those  induce much tighter CSDP relaxations as the size of a basis increases, at the cost of greater computational effort.

For future work, it would be interesting to have
 \Cref{conjecture_exactnessOfLifting} resolved. We are also interested in finding faster methods for solving large CSDPs arising from $\mathbf{L}^n(\mathscr{C}_{p})$. \Cref{table_angularSynchro} shows clearly that larger values of $p$ greatly improve the strength of relaxations, although the required computational effort (both time and memory) to solve them with off-the-shelf interior point method solvers quickly becomes prohibitive. A tailored solver might be able to handle much larger values of $n$ than $25$.  {Many approaches for improving the scalability of real SDP have been proposed in the literature, most of them via exploiting some form of sparsity in the objective function and/or constraints, see e.g., \cite{wang2021tssos,wang2022cs,chordalTSSOS,waki2006sums} (note that our numerical experiments involved only dense matrices).
 The authors of \cite{wang2023real} remark that some of those approaches can be also applied to CSDPs with only real coefficients.
  Future research is to investigate how these methods translate  for the case of CSDPs over relaxations of~$\cutPolytope$.
 }
 
\medskip\medskip

\noindent \textbf{Acknowledgements.} We thank two anonymous reviewers for providing us with insightful feedback on an earlier version of this manuscript.\\

\noindent \textbf{Conflict of interest statement.} The authors have no conflicts of interest to declare.

\bibliographystyle{abbrvnat}
\bibliography{myRefs.bib}

\setcounter{lemma}{0}
\renewcommand{\thelemma}{\Alph{section}\arabic{lemma}}
\appendix
\section{Proofs and auxiliary lemmas}
\label{section_appendixProofs}

\subsection{Proofs}
\label{section_replacedProofs}

For the proof of \Cref{lemma_finalRankThreeExcept}, on \cpageref{lemma_finalRankThreeExcept}, we require the following definitions from convex geometry.
\begin{definition}
    \label{def_convexFace}
    For a convex set $C$, and a subset $F \subseteq C$, we say that $F$ is a face of $C$ if it satisfies the following: if $x,y \in C$ and $t \in (0,1)$ are such that $tx+(1-t)y \in F$, then $x,y \in F$.
\end{definition}

\begin{definition}
    \label{def_exposedFace}
    A face $F$ of $C$ is said to be exposed, if $F$ is equal to the intersection of some hyperplane with $C$, or if $F = C$.
\end{definition}

We now present the proof.
\finalRankThreeException*
\begin{proof}
    The direction $\Leftarrow$ is trivial. For the reverse direction, assume that $\cutPolytopeNoSS^4_\infty$ is strictly contained in $\liftingSet{1}$. Then $\liftingSet{1}$ contains an extreme point, say $Y  \in \partial \liftingSet{1}$, which is not in $\cutPolytopeNoSS^4_\infty$. 
    Since matrices of rank one are elements of $\cutPolytopeNoSS^4_\infty$, it follows that $\rankOp(Y) \geq 2$.
       Since $Y \in \liftingSet{1}$, $Y \notin P$ (\Cref{thm_noR2EPinLE}). Then, by \Cref{lemma_rank2PointsInCut}, $\rankOp(Y) \geq 3$. 
    
    We now show that $\rankOp(Y) \leq 3$, using \Cref{def_convexFace,def_exposedFace}. Define
    \begin{align}
        \mathcal{F}^{-1}(Y) := \{ Z \in \mathcal{F}(\mathscr{B}_1) \, | \, Z_{1:4:,1:4} = Y \}
    \end{align}
    as the set of matrices in $\mathcal{F}(\mathscr{B}_1)$ with $Y$ as upper left submatrix. We show that $\mathcal{F}^{-1}(Y)$ is a face of $\mathcal{F}(\mathscr{B}_1)$. Let $U_1, \, U_2 \in \mathcal{F}(\mathscr{B}_1)$ and $t \in (0,1)$ such that $t U_1 + (1-t) U_2 \in \mathcal{F}^{-1}(Y)$. Denote by $Y_i$ the upper left $4 \times 4$ submatrix of $U_i$, $i \in [2]$. Since $U_i \in \mathcal{F}(\mathscr{B}_1)$, their submatrices $Y_i \in \liftingSet{1}$. Now
    \begin{align}
        t U_1 + (1-t) U_2 \in \mathcal{F}^{-1}(Y) \,\,\Rightarrow \,\, Y = tY_1 + (1-t)Y_2 \,\,\Rightarrow \,\, Y = Y_1 = Y_2 \,\,\Rightarrow \,\, U_1, U_2 \in \mathcal{F}^{-1}(Y),
    \end{align}
    where the second implication follows from the fact that $Y$ is an extreme point of $\liftingSet{1}$.    
    Thus $\mathcal{F}^{-1}(Y)$ is a face of $\mathcal{F}(\mathscr{B}_1)$ and the extreme points of $\mathcal{F}^{-1}(Y)$ form a subset of the extreme points of $\mathcal{F}(\mathscr{B}_1)$, see e.g., \cite[Proposition 8.6]{simon2011convexity}. 
    
    Let us consider the extreme points of $\mathcal{F}(\mathscr{B}_1)$ in more detail. The set $\mathcal{F}(\mathscr{B}_1)$ is a (complex) spectrahedron. By \cite[Corollary 1]{ramana1995some}, every face of a spectrahedron is exposed\footnote{Although the work \cite{ramana1995some} studies real spectrahedra, Section 1.4, Item 5 of the same work states that real and complex spectrahedra can be considered equivalent in the sense of their geometry.}. In particular, the extreme points of $F(\mathscr{B}_1)$ are faces of $F(\mathscr{B}_1)$, and are thus exposed. Therefore, if $V$ denotes such an extreme point, there exists a $Q \in \mathcal{H}^6$ such that 
    \begin{align}
        V = \argmax_{X \in \mathcal{F}(\mathscr{B}_1)} \langle Q, X \rangle,
    \end{align}
    i.e, $V$ can be written as the unique optimum of some CSDP over $\mathcal{F}(\mathscr{B}_1)$. Such optima have rank at most $\lfloor \sqrt{k} \rfloor$, for $k$ the number of affine constraints of the corresponding CSDP, see e.g. \cite[Theorem 5.1]{lemon2016low}. As $\mathcal{F}(\mathscr{B}_1)$ has 11 affine constraints, it follows that $\rankOp(V) \leq \lfloor \sqrt{11} \rfloor = 3$. Since $V$ was arbitrarily chosen, any extreme point of $\mathcal{F}(\mathscr{B}_1)$ has rank at most three.

    We now return to $Y$. Since $Y$ is a submatrix of matrices in $\mathcal{F}^{-1}(Y)$, it follows that
    \begin{align}
        \rankOp(Y) \leq \min_{Z \in \mathcal{F}^{-1}(Y)} \rankOp(Z).
    \end{align}
    We have shown that $\mathcal{F}^{-1}(Y)$ contains extreme points of $\mathcal{F}(\mathscr{B}_1)$, and that such extreme points have rank at most three. Therefore, $ \min_{Z \in \mathcal{F}^{-1}(Y)} \rankOp(Z) \leq 3$. Because we have previously deduced that $\rankOp(Y) \geq 3$, it follows that $\rankOp(Y) = 3$. Since the interior of $\mathcal{E}^4_\infty$ contains only rank 4 points, it follows that $Y \in \partial \elliptopeNoSS^4_\infty$. 

    {We now prove the last claim on $Y$, stating that $Y = \lambda J_4 + (1-\lambda) A$, for some $\lambda \in (0,1)$ and $A \in P$. Let us write $Y$ as 
    \begin{align}
        \label{eqn_conditionsXUWV}
        Y = B\conjTrans B, \text{ for } B = \begin{bmatrix}
            x & u & w & v
        \end{bmatrix}, \, x, \, u, \, w, \, v \in \mathbb{C}^3 \text{ and } \| x \| = \| u \| = \| w \| = \| v \| = 1.
    \end{align}
    Note that, given $Y$, $B$ is unique up to unitary multiplication, i.e., $B \to QB$, for $Q$ a unitary matrix in $\mathbb{C}^{3 \times 3}$. Moreover,
    \begin{align}
        \label{eqn_C_form}
        \lambda J_4 + (1-\lambda) A = C\conjTrans C, \text{ for } C = \begin{bmatrix}
            \sqrt{\lambda} \mathbf{1}^\top_4 \\[1ex]
            \sqrt{1 - \lambda} G
        \end{bmatrix},
    \end{align}
    for $G$ an EGF, see \Cref{def_EGF}. Thus, to prove the last claim on $Y$, we look for a unitary matrix $Q$ such that $QB$ is of the form presented in \eqref{eqn_C_form}. Let $z \in \mathbb{C}^3$ be such that $z\conjTrans$ is the first row of $Q$. This vector $z$ must satisfy
    \begin{align}
        \label{eqn_conditionZ}
         \| z \| = 1 \text{ and } | z\conjTrans x | = |z\conjTrans u | = |z\conjTrans w | = |z\conjTrans v | \, \,  (= \sqrt{\lambda}),
    \end{align}
    and we continue by proving its existence. Let us formulate a CSDP, with $zz\conjTrans$ as a feasible solution (if it exists):
    \begin{equation}
    \begin{alignedat}{2}
        \text{find }& Z \in \mathcal{H}^3_+ \\
        \text{ s.t }& \langle Z, xx\conjTrans  - uu\conjTrans \rangle &&= 0, \\
        &\langle Z, xx\conjTrans  - ww\conjTrans \rangle &&= 0, \\
        &\langle Z, xx\conjTrans  - vv\conjTrans \rangle &&= 0. \\
    \end{alignedat}
    \end{equation}
    By \eqref{eqn_conditionZ}, $zz\conjTrans$ (if it exists) is feasible to the above CSDP. Note also that $I$ is feasible to the above CSDP, by \eqref{eqn_conditionsXUWV}. Invoking \cite[Theorem 2.2]{ai2011new}, we find that the above CSDP admits a rank one solution, say $yy\conjTrans$. Letting $z = y / \| y\|$ shows that a  suitable vector $z$ exists. Let $Q$ now be any unitary matrix with such a suitable $z\conjTrans$ as its first row. By construction, the entries in the first row of $QB$ all have equal magnitude, but they are not necessarily purely real, as is required in \eqref{eqn_C_form}. 
  
    Let $r \in \mathbb{R}^4$ be the vector containing the arguments of the entries in the first row of $QB$,  {and set $\alpha = \exp{(r\imagUnit)} \in \mathcal{B}^4_\infty$,} where $\exp(\cdot)$ is evaluated element wise.  {Recall the group action $f_\alpha$ from \Cref{section_groupStructure}. By the fact that $\cutPolytopeNoSS^n_\infty$ is closed under the group action $f_\alpha$, we have}
    \begin{align}
        Y \notin \cutPolytopeNoSS^4_\infty \iff  {f_\alpha(Y) = (\alpha \alpha\conjTrans) \Hadamard Y = \text{Diag}(\alpha) Y \, \text{Diag}(\overline{\alpha}) \notin \cutPolytopeNoSS^4_\infty.}
    \end{align}
    It follows that $ {f_\alpha(Y)} = \widetilde{B}\conjTrans \widetilde{B}$, for $\widetilde{B} = QB \,  {\text{Diag}(\overline{\alpha})}$, and matrix $\widetilde{B}$ is of the desired form \eqref{eqn_C_form}. Note also that $ {f_\alpha(Y)}$ and $Y$ are both extreme points of $\liftingSet{1}$.}
\end{proof}

Below we prove \Cref{lemma_degCompletionGeneral}, on \cpageref{lemma_degCompletionGeneral}.
\degCompletionGeneral*
\begin{proof}
    It suffices to show that $\widetilde{\mathscr{A}}^p \models \mathscr{A}^p$ for all $p \geq 3$. We fix some $p \geq 3$. For notational convenience, we omit the superscript $p$ in sets $\widetilde{\mathscr{A}}^p$ and $\mathscr{A}^p$.
    
    The proof follows again from PSD matrix completion theory \cite{grone1984positive}. Let us first consider extending $\widetilde{\mathscr{A}}$ by a single vector $2\mathbf{e}_1 := 2 \cdot (1,0,\ldots,0)^\top \in \mathscr{A}$ (the unit vectors $\mathbf{e}_i$, $i \in [p]$ are defined similarly). We denote this new set $\mathscr{D} := \mathscr{A} \cup \{ 2 \mathbf{e}_1 \}$, and consider the problem of completing matrix  $X \in \mathcal{F}(\widetilde{\mathscr{A}})$ to matrix in $Z \in \mathcal{F}(\mathscr{D})$, for which $Z_{1:k,1:k} = X$, with $k = |\widetilde{\mathscr{A}}|$.
    
    As in the proof of \Cref{lemma_equivalentSecondDegLift}, the associated graph $\mathcal{G}$, see \eqref{eqn_patternGraph}, is again chordal. Thus it remains to verify that matrix $Z_\mathcal{J}$, for $\mathcal{J}$ as in \eqref{eqn_setJDef}, is similar to a submatrix of $X$. Note that
    \begin{align}
        \mathcal{J} = \left\{ \alpha \in \mathscr{D} \, \middle| \, Z_{\alpha, 2 \mathbf{e}_1} \neq \textit{?} \right\} = \left\{ (\alpha_1, \ldots, \alpha_p)^\top \in \mathscr{D} \, \middle| \, \alpha_1 \in \{1,2\} \right\},
    \end{align}
    where the rows of $Z$ are indexed by elements of $\mathscr{D}$.
    When entry $Z_{\alpha, 2 \mathbf{e}_1} \neq \textit{?}$, its value, in terms of the sequence $\mathsf{y}$, is contained in the following set
    \begin{align}
         V := \left\{ (M_{\mathscr{D}}(\mathsf{y}))_{\alpha, 2 \mathbf{e}_1} \, \middle| \, Z_{\alpha, 2 \mathbf{e}_1} \neq \textit{?} \right\} = \{ \mathsf{y}_{\beta} \, | \, \beta = -\mathbf{e}_1 + \mathbf{e}_i, i \in [n] \} \cup \{ \mathsf{y}_{-\mathbf{e}_1} \}.
    \end{align}
    Observe that all possible values in $V$ also appear in a single column of $X \in \mathcal{F}(\widetilde{\mathscr{A}})$. Specifically,
    \begin{align}
        V \subseteq \{ \left(M_{\widetilde{\mathscr{A}}}(\mathsf{y}) \right)_{\alpha, \mathbf{e}_1} \, | \, \alpha \in \widetilde{\mathscr{A}} \},
    \end{align}
    i.e., all values in $V$ also appear in the column of $X$ that is indicated by $\mathbf{e}_1 \in \widetilde{\mathscr{A}}$. Note that this implies that $Z_{\mathcal{J}}$ is similar to a submatrix of $X$, and thus PSD. Therefore, $\widetilde{\mathscr{A}} \models \widetilde{\mathscr{A}} \cup \{ 2\mathbf{e}_1 \}$. In a similar manner, it can be shown that
    \begin{align}
    \widetilde{\mathscr{A}} \models \widetilde{\mathscr{A}} \cup \{ 2\mathbf{e}_1 \} \models \widetilde{\mathscr{A}} \cup \{ 2\mathbf{e}_1, 2 \mathbf{e}_2 \} \models \cdots \models \widetilde{\mathscr{A}} \cup \{ 2\mathbf{e}_1, 2 \mathbf{e}_2, \ldots, 2 \mathbf{e}_p \} = \mathscr{A},
    \end{align}
    which proves the result.
\end{proof}

Below we prove \Cref{lemma_generalRank2EPofE}, on \cpageref{lemma_generalRank2EPofE}.
\generalRankTwoEPofE*
\begin{proof}
We show that $N$ satisfies  {\Cref{thm_item1,thm_item3}} of \Cref{lemma_elliptopeExtPointsR2}.
Denote by $\mathbf{e}, u, v \in \mathbb{C}^2$ the first three columns of $G$ (in that order). Observe that $u_1$ is a convex combination of $1$ and $\exp(2\pi \imagUnit / m)$, both $m$-roots of unity, and therefore 
\begin{align}
    N_{12} = \mathbf{e}\conjTrans u = u_1  = \cos{\left( \pi /m \right)} e^{\pi \imagUnit/m} \in \partial \ConvHull{(\mathcal{B}_m)} \setminus \mathcal{B}_m.
\end{align}
Hence, $N$ satisfies \Cref{thm_item1} of \Cref{lemma_elliptopeExtPointsR2}. Let us now verify \Cref{thm_item3} by computing a possible perturbation. We have
\begin{align}
    u_1 \overline{u}_2 = \frac{\sin(2 \pi /m)}{2} e^{\pi \imagUnit /m}.
\end{align}
To determine $\alpha$ up to real scaling, see \eqref{eqn_RmatrixEq}, we assume without loss of generality that $c = 1$. Then
\begin{align}
    \frac{\sin(2\pi /m )}{2} \begin{bmatrix}
        e^{\pi \imagUnit / m } & e^{-\pi \imagUnit / m } \\
        e^{-\pi \imagUnit / m } & e^{\pi \imagUnit / m }
    \end{bmatrix} \begin{bmatrix}
        \alpha \\ 
        \overline{\alpha}
    \end{bmatrix} = \begin{bmatrix}
        -\sin^2{(\pi/m)} \\[0.2cm]
        -\cos^2{(\pi/m)}
    \end{bmatrix} = \begin{bmatrix}
        \frac{\cos(2 \pi /m)-1}{2} \\[0.2cm]
        \frac{-\cos(2 \pi /m)-1}{2}
    \end{bmatrix},
\end{align}
and is thus given by
\begin{align}
   \begin{bmatrix}
       \alpha \\ \overline{\alpha}
   \end{bmatrix} &= \frac{1}{\sin(2\pi /m )} \begin{bmatrix}
        e^{\pi \imagUnit / m } & e^{-\pi \imagUnit / m } \\
        e^{-\pi \imagUnit / m } & e^{\pi \imagUnit / m }
    \end{bmatrix}^{-1}  \begin{bmatrix}
        \cos(2 \pi /m)-1 \\
        -\cos(2 \pi /m)-1
    \end{bmatrix} \\
    &= \frac{1}{2 \sin^2(2\pi /m ) \imagUnit} \begin{bmatrix}
        e^{\pi \imagUnit / m } & -e^{-\pi \imagUnit / m } \\
        -e^{-\pi \imagUnit / m } & e^{\pi \imagUnit / m }
    \end{bmatrix}  \begin{bmatrix}
        \cos(2 \pi /m)-1 \\
        -\cos(2 \pi /m)-1
    \end{bmatrix}.
\end{align}
This yields
\begin{align}
    \alpha &= \frac{-\imagUnit}{2 \sin^2\left(\frac{2\,\pi }{m}\right)} \left(  \,\cos\left(\frac{2\,\pi }{m} \right)  \left[ {e}^{-\frac{\pi \,\imagUnit}{m}} + {e}^{\frac{\pi \,\imagUnit}{m}} \right] + \left[ {e}^{-\frac{\pi \,\imagUnit}{m}} - {e}^{-\frac{\pi \,\imagUnit}{m}} \right] \right) \\
    &= \frac{-1}{\sin^2\left(\frac{2\,\pi }{m}\right)} \left(  \sin\left( \frac{\pi}{m} \right) + \cos\left( \frac{\pi}{m}\right) \cos\left( \frac{2\pi}{m} \right) \imagUnit \right).
\end{align}
Accordingly, $b_{12}$, see \eqref{eqn_BmatrixDef}, is computed as follows (using $c = 1$):
\begin{align}
    b_{12} = \mathbf{e}\conjTrans \begin{bmatrix}
        0 & \overline{\alpha} \\
        \alpha & 1
    \end{bmatrix} u = u_2 \overline{\alpha}.
\end{align}
It remains to show that $\mathrm{Re}(\overline{\nu} b_{12}) \neq 0$, for $\nu = \exp( \pi \imagUnit /m)$. To do so, note that  {$u_2 = \sin{(\pi/m)} \in \mathbb{R}$}, and thus,
\begin{align}
    \mathrm{Re}(\overline{\nu} b_{12}) &= \frac{-u_2}{\sin^2\left(\frac{2\,\pi }{m}\right)} \mathrm{Re}\left( e^{-\pi \imagUnit /m } \left[ \sin\left(\frac{\pi}{m}\right)-\imagUnit \cos\left(\frac{\pi}{m}\right)\cos\left(\frac{2\pi}{m}\right) \right] \right) \\
    &= \frac{-u_2}{\sin^2\left(\frac{2\,\pi }{m}\right)} \cos\left(\frac{\pi}{m}\right)\sin\left(\frac{\pi}{m}\right)\left[1-\cos\left(\frac{2\pi}{m}\right)\right] \neq 0, \text{ since } m > 2.
\end{align}
 {Since $b_{12}$ is uniquely determined (up to real scaling), it follows that there does not exist a perturbation satisfying $\mathrm{Re}(\overline{\nu} b_{12}) = 0$. Hence, }
$N$ satisfies \Cref{thm_item3} of \Cref{lemma_elliptopeExtPointsR2}. Thus $N$ is a rank 2 extreme point of $\mathcal{E}^3_m$.
\end{proof}

\renewcommand{\theHsection}{A\arabic{section}}
\subsection{Auxiliary lemmas}

The following result is used in the proof of \Cref{lemma_strengthOfFacet}, \cpageref{lemma_strengthOfFacet}.
\begin{lemma}
\label{lemma_maxValueOverE} 
For $Q$ as in \eqref{eqn_QcutDefinition},
\begin{align}
\max_{X \in \mathcal{E}^3_3} \langle Q, X \rangle =    \frac{3\cos\left(\frac{\pi}{18}\right)}{2\cos\left(\frac{\pi}{9}\right)}.
\end{align}
\end{lemma}
\begin{proof}
         For any $Y \in \mathcal{E}^3_3$, the value $\langle Q, Y \rangle$ provides a lower bound on $\max_{X \in \mathcal{E}^3_3} \langle Q, X \rangle $. Thus,
    \begin{align}
        \max_{X \in \mathcal{E}^3_3} \langle Q, X \rangle \geq \left\langle Q,  r \begin{bmatrix}
        r^{-1} & e^{4 \pi \imagUnit / 9 }& e^{2 \pi \imagUnit/9 }\\[0.5cm]
        e^{-4 \pi \imagUnit / 9 }& r^{-1} & e^{4 \pi \imagUnit / 9 } \\[0.5cm]
        e^{-2 \pi \imagUnit / 9 } & e^{-4 \pi \imagUnit /9}& r^{-1} 
    \end{bmatrix} \right\rangle =  \frac{3\cos\left(\frac{\pi}{18}\right)}{2\cos\left(\frac{\pi}{9}\right)}, \text{ for } r = \frac{1}{2\cos\left(\frac{\pi}{9}\right)}.
    \end{align}
    For the matching upper bound, we have that  for any $X \in \mathcal{E}^3_3$, the inner product $ \langle Q, X \rangle$ can be rewritten as follows: 
    \begin{align}
        \langle Q, X \rangle &= \frac{3\cos\left(\frac{\pi}{18}\right)}{2\cos\left(\frac{\pi}{9}\right)} - q \left\langle  \begin{bmatrix}
          1 & e^{-4 \pi \imagUnit / 9} & e^{-8 \pi \imagUnit / 9} \\
          e^{4 \pi \imagUnit / 9} & 1 & e^{-4 \pi \imagUnit / 9} \\
          e^{8 \pi \imagUnit / 9} & e^{4 \pi \imagUnit / 9} & 1
          \end{bmatrix}, X \right\rangle - \left( \sqrt{3} - \frac{2q}{r}\right) \sum_{1 \leq i < j \leq 3} \mathrm{Re}\left( \frac{1}{2} - X_{ij} \right) \\
          & \leq  \frac{3\cos\left(\frac{\pi}{18}\right)}{2\cos\left(\frac{\pi}{9}\right)} \quad \forall X \in \mathcal{E}^3_3,
    \end{align}
    where $q =(4\sin\left(2\pi/ 9\right))^{-1}$. The above upper bound follows from the fact that $\mathrm{Re}\left( \frac{1}{2} - X_{ij} \right) \geq 0$ for all $i$ and $j$, $ \sqrt{3} - \frac{2q}{r} \geq 0$, and the matrix in the inner product being positive semidefinite (one can verify that this matrix is rank 1). This proves the result.
\end{proof}

    The following result is used in the proof of \Cref{lemma_validCutForL}, \cpageref{lemma_validCutForL}.
\begin{lemma}
    \label{lemma_strengthOfH}
    For  any integer $m \geq 3$   or $m=\infty$, and for 
    $H$ as in \eqref{eqn_Hcut}, we have $\strFunc(H,m)  = \frac{2}{\sqrt{3}}$.
\end{lemma}
\begin{proof}
    Let $m \geq 3$ be an integer or $m = \infty$. For any $Y \in \mathcal{E}^4_m$, the value $\langle H, Y \rangle$ provides a lower bound on $\max_{X \in \mathcal{E}^4_m} \langle H, X \rangle $. Therefore,
    \begin{align}
        \max_{X \in \elliptopeNoSS^4_m} \langle H, X \rangle \geq \left\langle H, r \begin{bmatrix}
            r^{-1} & e^{-\pi \imagUnit / 2 } & e^{\pi \imagUnit / 2} & 1 \\
            e^{\pi \imagUnit / 2} & r^{-1 } & e^{-\pi \imagUnit / 2 } & 1 \\
            e^{-\pi \imagUnit / 2 } & e^{\pi \imagUnit / 2 } & r^{-1} & 1 \\
            1 & 1 & 1 & r^{-1}
        \end{bmatrix} \right\rangle = 4 \sqrt{3}, \text{ for } r = \frac{1}{\sqrt{3}}.
    \end{align}
    We verify that this $Y \in \mathcal{E}^4_m$. To show that $Y \succeq 0$, we compute the Schur complement of $Y$ with respect to the upper left entry. This gives
    \begin{align}
    \frac{2}{3}\begin{bmatrix}
        \begin{matrix}
            1 & e^{-2 \pi \imagUnit/3} \\
            e^{2 \pi \imagUnit / 3} & 1
        \end{matrix} & \mathbf{0}_2 \\
        \mathbf{0}_2^\top & 0
        \end{bmatrix}\succeq 0.
    \end{align}
    To verify that all elements of $Y$ are contained in $\ConvHull( \mathcal{B}_m)$, note that $1 \in \ConvHull( \mathcal{B}_m)$. For $m = 3$, $r e^{\pi \imagUnit / 2}$ is a convex combination of $1$ and $e^{2 \pi \imagUnit / 3} \in \mathcal{B}_3$. For $m = 4$, note that the polytope $\ConvHull(\mathcal{B}_4)$ contains an inscribed circle of radius
    \begin{align}
        \label{eqn_radiusB4}
        \min_{x \in \partial \ConvHull( \mathcal{B}_4)} \| x \| = \left\| \frac{1 + \imagUnit}{2} \right\| = \frac{1}{\sqrt{2}}.
    \end{align}
    Since
    \begin{align}
        \label{eqn_valueInR}
        \| r e^{\pi \imagUnit / 2} \| = \frac{1}{\sqrt{3}} < \frac{1}{\sqrt{2}},
    \end{align}
    it follows that $r e^{\pi \imagUnit / 2} \in \ConvHull(\mathcal{B}_4)$. Now for the case $m > 4$: let $R_m$ denote the radius of the inscribed circle of $\ConvHull(\mathcal{B}_m)$ (by \eqref{eqn_radiusB4}, $R_4 = 1 / \sqrt{2})$. Note that $R_m$ is increasing in $m$. Therefore, following \eqref{eqn_valueInR}, we have
    \begin{align}
        \| r e^{\pi \imagUnit / 2} \| \leq R_4 \leq R_m ~~\Rightarrow~~ r e^{\pi \imagUnit / 2} \in \ConvHull(\mathcal{B}_m) \quad \forall m \geq 4.
    \end{align}
    Clearly, also $r e^{-\pi \imagUnit / 2} \in \ConvHull(\mathcal{B}_m)$, since $\ConvHull(\mathcal{B}_m)$ is closed under complex conjugation. Thus, we have shown that all elements of $Y$ are contained in $\ConvHull(\mathcal{B}_m)$ for all valid $m$. Therefore, $Y \in \elliptopeNoSS^4_m$.
    
    Moreover, for any $X \in \elliptopeNoSS^4_m$, we have
    \begin{align}
        \label{eqn_dualRewriting}
        \langle H, X \rangle = 4 \sqrt{3} - \left\langle  \sqrt{3} I_4 - H, X \right\rangle \leq 4 \sqrt{3},
    \end{align}
    since matrix $\sqrt{3} I_4 - H \succeq 0$. This proves $\max_{X \in \elliptopeNoSS^4_m} \langle H, X \rangle = 4 \sqrt{3}$, for all integers $m \geq 3$. 
    To show that $\max_{X \in \cutPolytopeNoSS^4_m} \langle H, X \rangle =  6$, we take $J_4 \in \cutPolytopeNoSS^4_m$, for which it follows
    \begin{align}
        \max_{X \in \cutPolytopeNoSS^4_m} \langle H, X \rangle \geq \langle H, J_4 \rangle = 6.
    \end{align}
    For the matching upper bound, we have
    \begin{align}
        \max_{X \in \cutPolytopeNoSS^4_m} \langle H, X \rangle \leq \max_{X \in \cutPolytopeNoSS^4_\infty} \langle H, X \rangle  \leq \max_{X \in \liftingSet{1}} \langle H, X \rangle = 6,
    \end{align}
    which follows from \Cref{lemma_validCutForL}. Lastly, to compute the strength, see \eqref{eqn_strFunctionDef}, we have $4 \sqrt{3} / 6 = 2 / \sqrt{3}$.
\end{proof}

\section{Facet enumeration of \texorpdfstring{$\mathcal{V}(\cutPolytopeNoSS^3_3)$}{V(CUT,3,3)}}
\label{sectionAppendix_facetEnumeration}
\Cref{table_facetsCut33} provides a complete enumeration of the facets of $\mathcal{V}(\cutPolytopeNoSS^3_3)$, see \eqref{eqn_upperTriuPart}. These facets are given implicitly in \Cref{ThmCut33}, and one facet is given explicitly in \Cref{lemma_strengthOfFacet}. 

The facets in \Cref{table_facetsCut33} were verified using the \texttt{SageMath} software \cite{sagemath}. The precise code is provided in \Cref{listing_sageCode}. 

\begin{lstlisting}[caption ={\texttt{SageMath} code for computing the facets of $\mathcal{V}(\mathrm{CUT}^3_3).$}, captionpos=b, label={listing_sageCode}]
# Create the field extension to symbolically evaluate sqrt(3)
x = polygen(ZZ, 'x')
K.<sqrt3> = NumberField(x^2 - 3, embedding=AA(3)**(1/2))

# Create the polyhedron with points of the cut polytope. 
# Points are given as (Re(x), Im(x)).
P = Polyhedron([(1,1,1,0,0,0),(1,-0.5,-0.5,0,-sqrt3/2,-sqrt3/2),
(1,-0.5,-0.5,0,sqrt3/2,sqrt3/2),(-0.5,1,-0.5,-sqrt3/2,0,sqrt3/2),
(-0.5,-0.5,1,-sqrt3/2,-sqrt3/2,0),(-0.5,-0.5,-0.5,-sqrt3/2,sqrt3/2,-sqrt3/2),
(-0.5,1,-0.5,sqrt3/2,0,-sqrt3/2),(-0.5,-0.5,-0.5,sqrt3/2,-sqrt3/2,sqrt3/2),
(-0.5,-0.5,1,sqrt3/2,sqrt3/2,0)])

# compute the facets (i.e., half-space representation) of the cut polytope
P.Hrepresentation()

\end{lstlisting}

In \Cref{table_facetsCut33}, facets 1 up to and including 18 are the triangle facets, using the terminology of \Cref{section_exactDescriptionCut33}.
 Facets 19 up to and including 27 are the facets that ensure $x_i \in \ConvHull(\mathcal{B}_3)$ for all $i \in [3]$, see \eqref{eqn_convBineqElement}.

\begin{table}[ht]
    \centering
    \begin{tabular}{llll}
 \hline
& $\eta_1$ & $\eta_2$ & $\eta_3$ \rule[-0.9ex]{0pt}{0pt} \\ \hline 
\multicolumn{1}{l|}{1.}  & $\imagUnit  $                     & $e^{ \pi \imagUnit  /6}$          & $\imagUnit$\rule{0pt}{2.6ex} \\
\multicolumn{1}{l|}{2.}  & $e^{5 \pi \imagUnit  /6}$         & $e^{- \pi \imagUnit  /6}$         & $e^{ \pi \imagUnit  /6}$          \\
\multicolumn{1}{l|}{4.}  & $e^{ \pi \imagUnit  /6}$          & $e^{- \pi \imagUnit  /6}$         & $e^{5 \pi \imagUnit  /6}$         \\
\multicolumn{1}{l|}{5.}  & $\imagUnit  $                     & $-\imagUnit  $                    & $e^{- \pi \imagUnit  /6}$         \\
\multicolumn{1}{l|}{5.}  & $e^{ \pi \imagUnit  /6}$          & $e^{-5 \pi \imagUnit  /6}$        & $e^{ \pi \imagUnit  /6}$          \\
\multicolumn{1}{l|}{6.}  & $e^{- \pi \imagUnit  /6}$         & $-\imagUnit  $                    & $\imagUnit  $                     \\
\multicolumn{1}{l|}{7.}  & $e^{-5 \pi \imagUnit  /6}$        & $e^{ \pi \imagUnit  /6}$          & $e^{- \pi \imagUnit  /6}$         \\
\multicolumn{1}{l|}{8.}  & $-\imagUnit  $                    & $\imagUnit  $                     & $e^{ \pi \imagUnit  /6}$          \\
\multicolumn{1}{l|}{9.}  & $e^{ \pi \imagUnit  /6}$          & $\imagUnit  $                     & $-\imagUnit  $                    \\
\multicolumn{1}{l|}{10.} & $e^{- \pi \imagUnit  /6}$         & $e^{ \pi \imagUnit  /6}$          & $e^{-5 \pi \imagUnit  /6}$        \\
\multicolumn{1}{l|}{11.} & $e^{- \pi \imagUnit  /6}$         & $e^{5 \pi \imagUnit  /6}$         & $e^{- \pi \imagUnit  /6}$         \\
\multicolumn{1}{l|}{12.} & $-\imagUnit  $                    & $e^{- \pi \imagUnit  /6}$         & $-\imagUnit  $                    \\
\multicolumn{1}{l|}{13.} & $e^{5 \pi \imagUnit  /6}$         & $\imagUnit  $                     & $e^{5 \pi \imagUnit  /6}$         \\
\multicolumn{1}{l|}{14.} & $e^{-5 \pi \imagUnit  /6}$        & $e^{5 \pi \imagUnit  /6}$         & $\imagUnit  $                     \\
\multicolumn{1}{l|}{15.} & $\imagUnit  $                     & $e^{5 \pi \imagUnit  /6}$         & $e^{-5 \pi \imagUnit  /6}$        \\
\multicolumn{1}{l|}{16.} & $e^{5 \pi \imagUnit  /6}$         & $e^{-5 \pi \imagUnit  /6}$        & $-\imagUnit  $                    \\
\multicolumn{1}{l|}{17.} & $e^{-5 \pi \imagUnit  /6}$        & $-\imagUnit  $                    & $e^{-5 \pi \imagUnit  /6}$        \\
\multicolumn{1}{l|}{18.} & $-\imagUnit  $                    & $e^{-5 \pi \imagUnit  /6}$        & $e^{5 \pi \imagUnit  /6}$         \\
\multicolumn{1}{l|}{19.} & $\sqrt{3}e^{ \pi \imagUnit  /3}$  & $0$                               & $0$                               \\
\multicolumn{1}{l|}{20.} & $\sqrt{3}e^{- \pi \imagUnit  /3}$ & $0$                               & $0$                               \\
\multicolumn{1}{l|}{21.} & $-\sqrt{3}$                       & $0$                               & $0$                               \\
\multicolumn{1}{l|}{22.} & $0$                               & $\sqrt{3}e^{- \pi \imagUnit  /3}$ & $0$                               \\
\multicolumn{1}{l|}{23.} & $0$                               & $\sqrt{3}e^{ \pi \imagUnit  /3}$  & $0$                               \\
\multicolumn{1}{l|}{24.} & $0$                               & $-\sqrt{3}$                       & $0$                               \\
\multicolumn{1}{l|}{25.} & $0$                               & $0$                               & $\sqrt{3}e^{ \pi \imagUnit  /3}$  \\
\multicolumn{1}{l|}{26.} & $0$                               & $0$                               & $\sqrt{3}e^{- \pi \imagUnit  /3}$ \\
\multicolumn{1}{l|}{27.} & $0$                               & $0$                               & $-\sqrt{3}$                       \\ \hline
\end{tabular}
    \caption{Coefficients of the facets $\mathrm{Re}\left(\sum_{i = 1}^3 \eta_i x_i \right) \leq \sqrt{3}/2$ for $\mathcal{V}(\cutPolytopeNoSS^3_3)$.}
    \label{table_facetsCut33}
\end{table}

\end{document}